\documentclass[10pt,a4paper,reqno]{amsart}

\usepackage{avant}
\usepackage{amsmath}
\usepackage{amsfonts}
\usepackage{palatino}
\usepackage{a4wide}
\usepackage{amssymb}
\usepackage{amsthm}

\usepackage{xcolor}
\definecolor{darkgreen}{rgb}{0.8,0.2,0}
\definecolor{citegreen}{rgb}{0,0.6,0}
\definecolor{refred}{rgb}{0.8,0,0}
\usepackage[colorlinks, citecolor=citegreen, linkcolor=refred]{hyperref}
\usepackage{tikz,fp,ifthen}
\usepackage{tikz-3dplot}
\usepackage[fixamsbugs]{mathtools}

\usepackage{mathptmx}
\usepackage{mathpazo}
\usepackage{palatino}
\usepackage{a4wide}
\numberwithin{figure}{section}

\theoremstyle{plain}

\newtheorem{teo}{Theorem}[section]
\newtheorem{lemma}[teo]{Lemma}
\newtheorem{prop}[teo]{Proposition}
\newtheorem{cor}[teo]{Corollary}
\newtheorem{ackn}{Acknowledgements\hspace{-.2em}}

\theoremstyle{definition}
\newtheorem{dfnz}[teo]{Definition}

\newtheorem{oss}[teo]{Remark}
\theoremstyle{remark}

\usepackage{enumitem}
\newcommand{\subscript}[2]{$#1 _ #2$}
\renewcommand{\phi}{\varphi}
\numberwithin{equation}{section}

\def\NN{{{\mathbb N}}}

\def\R{{\mathbb{R}}}

\newcommand{\EEE}{\color{black}}

\def\mathbb RR{\mathrm{R}}
\usepackage{scalerel}[2014/03/10]
\usepackage[usestackEOL]{stackengine}
\def\intavg{\,\ThisStyle{\ensurestackMath{%
			\stackinset{c}{0\LMpt}{c}{0\LMpt}{\SavedStyle-}{\SavedStyle\phantom{\int}}}%
		\setbox0=\hbox{$\SavedStyle\int\,$}\kern-\wd0}\int}

\title{Fractured membranes under determinant constraints: the case of cohesive surface energies}

\author[Nicola Pio Melillo]{Nicola Pio Melillo}
\address[Nicola Pio Melillo]{ Scuola Superiore Meridionale, Universitá di Napoli Federico II Largo
	San Marcellino 10, 80138 Naples, Italy.}
    \email{n.melillo@ssmeridionale.it}
\author[Dario Reggiani]{Dario Reggiani}
\address[Dario Reggiani]{Applied Mathematics M\"unster, University of M\"unster\\
	Einsteinstrasse 62, 48149 M\"unster, Germany}
\email{dario.reggiani@uni-muenster.de}


\begin{document}

	\begin{abstract}
        This paper is devoted to the variational derivation of reduced models for elastic membranes with fracture under constraints on the determinant of the deformation gradient. We consider two physically relevant settings: the \emph{orientation-preserving} regime, in which the deformation is required to preserve orientation locally ($\det \nabla u > 0$), and the \emph{incompressible} regime, in which the deformation preserves volume ($\det \nabla u = 1$). In both cases, the surface energy density is allowed to depend on the jump amplitude, thus encompassing \emph{cohesive fracture models with activation threshold}. The main technical contribution is the construction of recovery sequences that simultaneously satisfy the determinant constraint and optimize the surface energy. This is achieved through a combination of $C^\infty$ diffeomorphisms converging to the identity (which rotate the normal to the jump set so as to minimize the reduced surface energy), and a new smooth approximation result for $GSBV^p$ functions.
	\end{abstract}

	\maketitle
	
	\section{Introduction}

    The mathematical study of dimension reduction in elasticity has a long history, dating back to the pioneering contributions of L.~Euler and D.~Bernoulli, although most rigorous variational derivations appeared only in the 1990s~\cite{acerbi1991variational,anzellotti1994dimension,ledret1993modele}. In this paper we focus on the derivation of a reduced model for cohesive elastic membranes under determinant constraints. Our approach relies on $\Gamma$-convergence~\cite{dal2012introduction}, a method that has been successfully applied to a broad class of dimension-reduction problems, including plates~\cite{conti2009gamma,eleuteri2024asymptotic,freddi2010dimension,friedrich2026kirchhoff,friedrich2020derivation,friesecke2002theorem,friesecke2006hierarchy,hornung2014derivation,kruvzik2026effects}, rods and ribbons~\cite{cicalese2017global,cicalese2017hemihelical,davoli2011thin,friedrich2024voids,friedrich2022ribbon}, and membranes~\cite{cesana2015effective,engl2026membrane,hafsa2008nonlinear,hafsa2008nonlinear2,trabelsi2006modeling}. We mention also its use in elastoplasticity \cite{davoli2014linearized,davoli2014quasistatic,davoli2013mora}. In the context of variational fracture mechanics~\cite{bourdin2008variational}, dimension-reduction results are comparatively scarce, both in linearized~\cite{almi2021dimension,almi2023brittle,babadjian2016reduced,ginster2024euler} and nonlinear elasticity~\cite{almi2023brittlemembranes,babadjian2006quasistatic,braides2001brittle,schmidt2017griffith,schmidt2023continuum}. While some of these works incorporate either a orientation-preserving condition (positive Jacobian determinant, in the Griffith setting) or an incompressibility constraint (unit Jacobian determinant, in the Sobolev setting), the case of jump-dependent surface energies under determinant constraints has so far remained open. The present work addresses this gap.

\medskip

\noindent\textbf{Setting of the problem.}
Let $\Sigma \subset \mathbb{R}^2$ be a bounded open set with Lipschitz boundary, and let $\Sigma_\rho := \Sigma \times (-\rho/2,\, \rho/2)$ denote the thin reference configuration of thickness $\rho > 0$. We study the asymptotic behavior, as $\rho \to 0$, of the free-discontinuity functional
\begin{equation}\label{equation1}
	\mathcal{F}_\rho(v) :=
	\begin{cases}
		\displaystyle\int_{\Sigma_\rho} W(\nabla v)\, dy + \int_{J_v} \psi([v], \nu_v)\, d\mathcal{H}^2
		& \text{if } v \in GSBV^p(\Sigma_\rho;\mathbb{R}^3), \\[8pt]
		+\infty & \text{otherwise in } L^0(\Sigma_\rho;\mathbb{R}^3),
	\end{cases}
\end{equation}
where $p \in (1,+\infty)$, $L^0(\Sigma_\rho;\mathbb{R}^3)$ is the space of measurable functions, and $GSBV^p(\Sigma_\rho;\mathbb{R}^3)$ denotes the class of generalized special functions of bounded variation whose approximate gradient belongs to $L^p$ and whose jump set has finite $\mathcal{H}^2$-measure. Here $W : \mathbb{M}^{3 \times 3} \to [0,+\infty]$ is the stored-energy density and $\psi : \mathbb{R}^3 \setminus \{0\} \times \mathbb{S}^2 \to [0,+\infty)$ is the surface energy density.

To reduce the problem to a fixed domain, we perform the change of variables $(y_1,y_2,y_3) = (x_1,x_2,\rho x_3)$, so that $\mathcal{F}_\rho$ can be rewritten as the rescaled energy
\begin{equation}\label{intro resc func}
	\rho^{-1}\mathcal{G}_\rho(u) =
	\begin{cases}
		\displaystyle\int_{\Sigma_1}
		W\!\left(\nabla_\alpha u \,\bigg|\, \frac{1}{\rho}\partial_3 u\right) dx
		+ \int_{J_u} \psi_\rho([u], \nu_u)\, d\mathcal{H}^2
		& \text{if } u \in GSBV^p(\Sigma_1;\mathbb{R}^3), \\[10pt]
		+\infty & \text{otherwise},
	\end{cases}
\end{equation}
where $\Sigma_1 := \Sigma \times (-1/2,\, 1/2)$, $\nabla_\alpha u := (\partial_1 u \mid \partial_2 u) \in \mathbb{M}^{3 \times 2}$, and $\psi_\rho(z,\nu) := \psi(z,\, \nu_\alpha,\, \nu_3/\rho)$ with $\nu_\alpha:=(\nu_1,\nu_2)$.

Without a determinant constraint, the $\Gamma$-limit of $\rho^{-1}\mathcal{G}_\rho$ has been identified by Braides and Fonseca~\cite{braides2001brittle} as
\begin{equation}\label{000}
	\mathcal{G}_0(u) =
	\begin{cases}
		\displaystyle\int_\Sigma \mathcal{Q}W_0(\nabla_\alpha u)\, dx + \int_{J_u \cap \Sigma} \mathcal{B}\psi_0([u], \nu_u)\, d\mathcal{H}^1
		& \text{if } u \in \mathcal{A}, \\[8pt]
		+\infty & \text{otherwise},
	\end{cases}
\end{equation}
where $\mathcal{A} := \{u \in GSBV^p(\Sigma;\mathbb{R}^3) \mid u \text{ is independent of } x_3\}$, and the reduced densities are
\[
W_0(E) := \inf_{\xi \in \mathbb{R}^3} W(E \mid \xi), \qquad
\psi_0(z,\nu) := \inf_{\zeta \in \mathbb{R}} \psi(z,\, \nu_\alpha,\, \zeta).
\]
Here $\mathcal{Q}W_0$ and $\mathcal{B}\psi_0$ denote the quasiconvex envelope and the $BV$-elliptic envelope, respectively (see Definitions~\ref{defquasiconvex} and~\ref{defBV-ellipticity}). The main difficulty in extending this result to the determinant-constrained setting is the lack of uniform $p$-growth of the bulk energy density $W$, which prevents the direct application of classical integral representation theorems~\cite{bouchitte2002global,braides1996homogenization,cagnetti2019gamma}.

\medskip

\noindent\textbf{Assumptions on the bulk energy.}
We treat two cases. In the \emph{orientation-preserving} case, we assume:
\begin{enumerate}[label=(\subscript{A}{\arabic*})]
	\item $W \in C(\mathbb{M}^{3 \times 3};[0,+\infty])$;
	\item $W(F) < +\infty$ if $\det F > 0$, and $W(F) = +\infty$ if $\det F \le 0$;
	\item there exists $C_1 > 0$ such that $W(F) \ge C_1|F|^p - C_1^{-1}$ for every $F \in \mathbb{M}^{3 \times 3}$;
	\item for every $\delta > 0$ there exists $c_\delta > 0$ such that $W(F) \le c_\delta(1+|F|^p)$ whenever $\det F \ge \delta$.
\end{enumerate}
Condition $(A_2)$ enforces $\det \nabla u > 0$ a.e.\ for every admissible deformation, while $(A_4)$ replaces the standard $p$-growth from above, which is incompatible with the constraint $\det F > 0$ (see~\cite{hafsa2008nonlinear,hafsa2008nonlinear2}).

In the \emph{incompressible} case, we replace $(A_1)$--$(A_4)$ with:
\begin{enumerate}
	\item[$(\widetilde{A}_1)$] $W \in C(\mathrm{SL}(3);[0,+\infty))$, where $\mathrm{SL}(3) := \{F \in \mathbb{M}^{3 \times 3} \mid \det F = 1\}$;
	\item[$(\widetilde{A}_2)$] $W(F) = +\infty$ if $\det F \neq 1$; otherwise, there exists $c \ge 1$ such that 
    $$
    c^{-1}|F|^p - c \le W(F) \le c|F|^p + c.
    $$
\end{enumerate}
Here $(\widetilde{A}_2)$ enforces $\det \nabla u = 1$ a.e.\ for every admissible deformation, and the $p$-growth and coercivity conditions are simultaneously available on $\mathrm{SL}(3)$.

\medskip

\noindent\textbf{Assumptions on the surface energy.}
In both cases, the surface energy density $\psi : \mathbb{R}^3 \setminus \{0\} \times \mathbb{S}^2 \to [0,+\infty)$ satisfies:
\begin{enumerate}[label=(\subscript{B}{\arabic*})]
	\item for every $\nu \in \mathbb{S}^2$ there exists a decreasing continuous function $\sigma : [0,+\infty) \to [0,+\infty)$ with $\sigma(0) = 0$ such that, for every $z_1, z_2 \in \mathbb{R}^3 \setminus \{0\}$,
	$$
    |\psi(z_1,\nu) - \psi(z_2,\nu)| \le \sigma(|z_1-z_2|)(\psi(z_1,\nu) + \psi(z_2,\nu));
    $$
	\item for every $\nu \in \mathbb{S}^2$ there exists $C_2 \ge 1$ such that $\psi(z_1,\nu) \le C_2\, \psi(z_2,\nu)$ whenever $|z_1| \le |z_2|$, and $\psi(z_1,\nu) \le \psi(z_2,\nu)$ whenever $C_2|z_1| \le |z_2|$;
	\item there exist a Borel function $\varphi : \mathbb{R}^3 \setminus \{0\} \to [0,+\infty)$ and constants $0 < C_3 \le C_4$ such that, for all $z \in \mathbb{R}^3 \setminus \{0\}$ and $\nu \in \mathbb{S}^2$, 
    \[
    1 \le \varphi(z) \le 1+|z| \qquad \mbox{and} \qquad C_3\varphi(z) \le \psi(z,\nu) \le C_4\varphi(z);
    \]
	\item $\psi(z,\nu) = \psi(-z,-\nu)$ ;
	\item $\psi$ is upper semicontinuous on $(\mathbb{R}^3 \setminus \{0\}) \times \mathbb{S}^2$.
\end{enumerate}
We may extend $\psi$ to a function defined on 
			$(\mathbb{R}^3 \setminus \{0\}) \times \mathbb{R}^3$ by positive $1$-homogeneity, namely
			$$
			\psi(z,\nu) =
			\begin{cases}
				|\nu| \, \psi\!\left(z, \dfrac{\nu}{|\nu|}\right) & \text{if } \nu \neq 0, \\[6pt]
				0 & \text{otherwise}.
			\end{cases}
			$$
			Since assumptions $(B_1)$--$(B_5)$ remain valid for the extended function, from now on 
			we may regard $\psi$ as defined on 
			$\mathbb{R}^3 \setminus \{0\} \times \mathbb{R}^3$.
            
Assumptions $(B_1)$ and $(B_3)$ together imply continuity of $\psi$ in $z$. We remark that $(B_3)$ yields a uniform positive lower bound $\psi(z,\nu) \ge C_3 > 0$ for all $z \neq 0$: our framework thus models \emph{cohesive fracture with activation threshold}, intermediate between Griffith's theory~\cite{griffith1921rupture} (where $\psi$ is independent of $[u]$) and Barenblatt's model~\cite{barenblatt1962equilibrium} (where $\psi(z,\nu) \to 0$ as $|z| \to 0$). The latter regime, which lacks the compactness properties guaranteed by $(B_3)$, remains an interesting open problem. We observe that if $\varphi(z) = 1 + |z|$, the coercivity in $(B_3)$ allows one to work directly in $SBV^p(\Sigma_\rho;\mathbb{R}^3)$ rather than $GSBV^p$. Finally, assumption $(B_2)$ is used in the proof of the lower bound, which relies on a result from \cite{cagnetti2019gamma}.

\medskip

\noindent\textbf{Main results.}
Our main theorems (Theorems~\ref{teorfinal1} and~\ref{teorfinal2}) establish that, under assumptions $(A_1)$--$(A_4)$ (resp.\ $(\widetilde{A}_1)$--$(\widetilde{A}_2)$) and $(B_1)$--$(B_5)$, the rescaled functionals $\rho^{-1}\mathcal{G}_\rho$ $\Gamma$-converge with respect to the convergence in measure to the functional $\mathcal{G}_0$ defined in~\eqref{000}, in both the orientation-preserving and incompressible cases.

The $\Gamma$-$\liminf$ inequality follows by adapting the arguments of~\cite{braides2001brittle}, using the compactness guaranteed by $(B_3)$ to show that the limit deformation is independent of $x_3$ and that $(\nu_u)_3 = 0$ $\mathcal{H}^2$-a.e.\ on $J_u$. A key point is that the coupled growth encoded in $(B_3)$ allows one to absorb the surface energy remainder arising from the third component of the normal.

The $\Gamma$-$\limsup$ inequality is considerably more delicate and constitutes the main technical contribution of this work. To date, the $\Gamma$-convergence of $\rho^{-1}\mathcal{G}_\rho$ under a determinant constraint has been established only in the orientation-preserving setting and only for Griffith-type energies ($\psi(z,\nu) = |\nu|$) in~\cite{almi2023brittlemembranes}. The proof relies on the density of piecewise affine functions with maximal-rank gradient in $GSBV^p$, building on a classical result of Gromov and Eliashberg~\cite{gromov1971nonsingular,gromov1986pdr} and on the observation~\cite{dephilippis2017approximation} that a domain with smooth cuts can be mapped onto a Lipschitz domain via a bi-$W^{1,\infty}$ diffeomorphism close to the identity. However, this density result applies only to the estimation of jump set's length, thus a substantial refinement is needed to handle the jump-dependent surface energy.

We address this through two main constructions. First, for the incompressible case, since piecewise affine approximants lack the regularity needed to control the incompressibility constraint, we develop an approximation result for $GSBV^p$ functions (Proposition~\ref{propu_jbuona}) that provides $W^{k,\infty}$-regular approximants away from the jump set, with controlled determinant. This relies on mapping the cracked domain onto a Lipschitz domain through smooth diffeomorphisms constructed via convolutions with variable kernels (Lemma~\ref{lemmadiffeolisciokernelvariabile}). We emphasize that this approximation result (as the one in~\cite{almi2023brittlemembranes}) is proved in the two-dimensional setting, which is the one required for the membrane application considered in this paper.

Second, we build a family of $C^\infty$ diffeomorphisms $f_{j,\rho}$ converging to the identity in $W^{k,\infty}$ for every $k \ge 1$ (Lemma~\ref{lemmalinearalgebra} and Lemma~\ref{corollariodet}). These maps act as isometries near each segment of the jump set, rotating the normal so as to optimize the third component of $\psi$ through the infimum defining $\psi_0$. In the incompressible case, the additional constraint $\det \nabla f_{j,\rho} = 1$ is enforced via an ODE-based correction (Lemma~\ref{lemma1}).

The crucial intermediate step is to bound the $\Gamma$-$\limsup$ by the auxiliary functional
\begin{equation}\label{eqG_0^w}
	\mathcal{G}^w_0(u) :=
	\begin{cases}
		\displaystyle\int_\Sigma W_0(\nabla_\alpha u)\, dx + \int_{J_u \cap \Sigma} \psi_0([u], \nu_u)\, d\mathcal{H}^1
		& \text{if } u \in Y, \\[8pt]
		+\infty & \text{otherwise},
	\end{cases}
\end{equation}
where $Y$ consists of functions in $\mathcal{A}$ whose trace on $\Sigma$ belongs to $\mathrm{Aff}^*(\Sigma \setminus \overline{J_u};\mathbb{R}^3) \cap \widehat{\mathcal{W}}(\Sigma;\mathbb{R}^3)$ (see Definitions~\ref{defAff^*} and~\ref{defwidehatW}). The passage from $\mathcal{G}^w_0$ to $\mathcal{G}_0$ then follows by an adaptation of the relaxation arguments of~\cite{hafsa2008nonlinear2}, replacing the density of local immersions with Theorem~\ref{teorapproxaff^*}.

\medskip

To the best of our knowledge, these are the first dimension-reduction results for membranes with jump-dependent surface energy under determinant constraints. We believe that our techniques may open the way to further reduced theories involving different geometries (e.g., shells~\cite{almi2021dimension,almi2025generalized,ciarlet2000volIII}) or weaker growth assumptions on $\psi$ or $W$.

We finally point out that some of the techniques introduced
in this paper to construct the recovery sequences are not specific to the membrane scaling. 
In particular, the use of local diffeomorphisms to optimize the surface term and to preserve the determinant constraint could be useful in the construction of recovery
sequences for other dimension-reduction regimes, such as incompressible plate
theories in the non-fractured setting (see e.g. \cite{conti2009gamma}). 
However, a full extension to (for instance) fractured plate regimes appears to require substantially different additional ingredients. In particular, one would need compactness and rigidity estimates compatible with the presence of cracks, as well as suitable density results for the corresponding constrained limiting classes, for example density of smooth isometries in \(W^{2,2}\)-isometries on cracked domains. These issues are not addressed here and lie beyond the scope of the present paper.

\medskip

\noindent\textbf{Plan of the paper.} In Section~2 we fix notation and recall some preliminary density results \EEE in $SBV$ and $GSBV$. Section~3 contains the approximation results: the density theorem with maximal-rank constraint (Theorem~\ref{teorapproxaff^*}), and the higher-regularity approximation for the incompressible case (Proposition~\ref{propu_jbuona}). Section~4 is devoted to the proof of the $\Gamma$-convergence result in the orientation-preserving case, and Section~5 in the incompressible case. In Appendix A we recall some basic definitions and results concerning the spaces $SBV$ and $GSBV$, quasiconvexity, and $BV$-ellipticity. Finally, in Appendix B we give the proof of Lemma~\ref{lemmadiffeolisciokernelvariabile}.

	\section{Notations and Preliminaries}
	\subsection{Basic notation}
Let $n,k \in \mathbb{N}$. We denote by $\mathcal{L}^n$ the Lebesgue measure on $\mathbb{R}^n$ and by $\mathcal{H}^k$ the $k$-dimensional Hausdorff measure. The symbol $\mathbb{M}^{m \times n}$ stands for the space of real $m \times n$ matrices, and $\mathrm{Id}$ denotes the identity matrix. For every $r>0$ and $x \in \mathbb{R}^n$, we write $B_r(x)$ for the open ball of radius $r$ centred at $x$. Given $\delta>0$ and $U \subseteq \mathbb{R}^n$, we set $(U)_\delta := U + B_\delta(0)$. For an open set $G \subset \mathbb{R}^2$ and $h>0$, we write $G_h := G \times (-h/2,\, h/2)$. For $x \in \R^3$ we write $x_\alpha:=(x_1,x_2)$.

Given a matrix $A \in \mathbb{M}^{3 \times 2}$, we write $A = (A^1 \mid A^2)$, where $A^1$ and $A^2$ denote the first and second column of $A$, respectively. For two vectors $u,v \in \mathbb{R}^3$, we denote by $u \wedge v$ their cross product. We say that two matrices $A,B \in \mathbb{M}^{3 \times 2}$ are \emph{rank-one connected} if $\mathrm{rank}(A-B) = 1$. We denote by $\mathrm{span}(A)$ the vector space generated by the columns of $A$.

A subset $K$ of $\Omega$ is said to be \emph{polyhedral} (with respect to $\Omega$) if it is the intersection of $\Omega$ with a finite union of $(n-1)$-dimensional simplices of $\mathbb{R}^n$. A set $K \subseteq \mathbb{R}^n$ is \emph{countably $\mathcal{H}^{n-1}$-rectifiable} if it can be written as
\[
K = N \cup \bigcup_{i \in \mathbb{N}} K_i,
\]
with $\mathcal{H}^{n-1}(N) = 0$ and each $K_i$ contained in a $C^1$ hypersurface $\Gamma_i$. Every such set carries a Borel measurable \emph{approximate unit normal} $\nu_K : K \to \mathbb{S}^{n-1}$, defined $\mathcal{H}^{n-1}$-a.e.

Let $f : \mathbb{R}^n \to \mathbb{R}^m$ with $m,n \ge 1$. At every point $x \in \mathbb{R}^n$ where $f$ is differentiable, the Jacobian matrix $\nabla f(x) \in \mathbb{M}^{m \times n}$ is defined by
\[
(\nabla f(x))^j_{i} = \frac{\partial f_i}{\partial x_j}(x), \qquad i=1,\dots,m, \quad j=1,\dots,n.
\]

We now recall some standard notions from the calculus of variations.

\begin{dfnz}[Quasiconvexity]
	A Borel measurable, locally bounded function $f : \mathbb{M}^{m \times n} \to \mathbb{R}$ is \emph{quasiconvex} if
	\[
	\intavg_\Omega f(z + \nabla \varphi(x)) \, dx \ge f(z)
	\]
	for every $\varphi \in W^{1,\infty}_0(\Omega;\mathbb{R}^m)$, for every $z \in \mathbb{M}^{m \times n}$, and for some bounded Lipschitz domain $\Omega \subseteq \mathbb{R}^n$.
\end{dfnz}

\begin{dfnz}[Quasiconvex envelope]\label{defquasiconvex}
	Let $f : \mathbb{M}^{m \times n} \to [0,+\infty]$ be Borel measurable. The \emph{quasiconvex envelope} of $f$ is
	\[
	\mathcal{Q}f(\xi) := 
	\inf \bigg\{
	\frac{1}{|B_1|}\int_{B_1} f(\xi + \nabla \phi(x)) \, dx \;\bigg|\; 
	\phi \in W^{1,\infty}_0(B_1;\mathbb{R}^m)
	\bigg\},
	\qquad \xi \in \mathbb{M}^{m \times n}.
	\]
\end{dfnz}

\begin{oss}
	If $f : \mathbb{M}^{m \times n} \to [0,+\infty)$ is Borel measurable and locally bounded, then by \cite[Theorem~6.9]{dacorogna1989},
	\[
	\mathcal{Q}f(\xi)
	=
	\sup \big\{ h(\xi) \;\big|\; h : \mathbb{M}^{m \times n} \to [0,+\infty),\;
	h \le f,\; h \text{ quasiconvex} \big\}
	\qquad \text{for every } \xi \in \mathbb{M}^{m \times n}.
	\]
\end{oss}

\begin{dfnz}[Rank-one convexity]
	Let $f : \mathbb{M}^{m \times n} \to [0,+\infty]$ be Borel measurable.
	\begin{itemize}
		\item $f$ is \emph{rank-one convex} if for every $\lambda \in (0,1)$ and every $\xi,\xi' \in \mathbb{M}^{m \times n}$ with $\mathrm{rank}(\xi-\xi') = 1$,
		\[
		f(\lambda \xi + (1-\lambda)\xi') 
		\le 
		\lambda f(\xi) + (1-\lambda)f(\xi').
		\]
		\item The \emph{rank-one convex envelope} of $f$, denoted by $\mathcal{R}f$, is the largest rank-one convex function below $f$.
	\end{itemize}
\end{dfnz}

\begin{dfnz}[Piecewise affine functions]
	A function $f : \Omega \subseteq \mathbb{R}^n \to \mathbb{R}^m$ is \emph{piecewise affine} if there exists a finite family $\{V_i\}_{i \in I}$ of pairwise disjoint open subsets of $\Omega$ such that $\mathcal{L}^n(\Omega \setminus \bigcup_{i \in I} V_i)=0$ and $f|_{V_i}$ is affine for every $i \in I$. We denote
	\[
	\mathrm{Aff}(\Omega;\mathbb{R}^m)
	:= 
	\{ f \in C(\Omega;\mathbb{R}^m) \mid f \text{ is piecewise affine} \},
	\]
	\[
	\mathrm{Aff}_c(\Omega;\mathbb{R}^m)
	:= 
	\{ f \in C_c(\Omega;\mathbb{R}^m) \mid f \text{ is piecewise affine} \}.
	\]
\end{dfnz}

\begin{dfnz}[Triangulation]
	Let $\Omega$ be a bounded open subset of $\mathbb{R}^n$ and $u \in \mathrm{Aff}(\Omega;\mathbb{R}^m)$. The \emph{triangulation} of $u$, denoted $\mathcal{T}_u$, is the collection of sets $T_i \cap \Omega$, where each $T_i$ is an $n$-simplex on which $u$ is affine. The \emph{diameter} of $\mathcal{T}_u$ is $\max\{\mathrm{diam}(T_i) \mid T_i \cap \Omega \in \mathcal{T}_u\}$. A point in $\Omega$ that is a vertex of some $T_i$ is called a \emph{vertex} of the triangulation. When $n=2$, the closed segments in $\overline{\partial T_i \cap \Omega}$ are called \emph{edges}.
\end{dfnz}

\begin{dfnz}[Clarke subdifferential]
	Let $f : \Omega \subseteq \mathbb{R}^n \to \mathbb{R}^m$ be locally Lipschitz. The \emph{Clarke subdifferential} of $f$ at $x \in \overline{\Omega}$ is
	\[
	\partial f(x)
	:=
	\mathrm{conv}
	\Big\{
	\lim_{k \to \infty} \nabla f(x_k)
	\;\Big|\;
	x_k \to x,\; \nabla f(x_k) \text{ exists}
	\Big\}.
	\]
\end{dfnz}

\begin{dfnz}\label{defAff^*}
	Let $\Omega$ be a bounded open set with $n \le m$.
	\begin{itemize}
		\item If $n < m$:
		\[
		\mathrm{Aff}^*(\Omega;\mathbb{R}^m)
		:=
		\bigg\{
		v \in \mathrm{Aff}(\Omega;\mathbb{R}^m)
		\;\bigg|\;
		\min\{ \det(A^T A) \mid A \in \partial v(x),\; x \in \Omega \} > 0
		\bigg\}.
		\]
		\item If $n = m$:
		\[
		\mathrm{Aff}^*(\Omega;\mathbb{R}^m)
		:=
		\{ v \in \mathrm{Aff}(\Omega;\mathbb{R}^m) \mid \det(\nabla v) > 0 \text{ in } \Omega \}.
		\]
	\end{itemize}
\end{dfnz}

\subsection{Density results in $SBV$ and $GSBV$}

    In this subsection, we recall a density result for $GSBV$ functions that will be used later. This result can be found in \cite{cortesani1999density}. We first introduce the subclass $\mathcal{W}$ of “regular” $SBV$ functions.

     \begin{dfnz}
	The class $\mathcal{W}(\Omega;\mathbb{R}^m)$ consists of all $u \in SBV(\Omega;\mathbb{R}^m)$ such that:
	\begin{itemize}
		\item $u \in W^{1,\infty}(\Omega \setminus \overline{J_u};\mathbb{R}^m)$;
		\item $J_u$ is essentially closed, i.e., $\mathcal{H}^{n-1}(\overline{J_u} \setminus J_u) = 0$;
		\item $J_u$ is a finite union of $(n{-}1)$-dimensional polyhedral pieces.
	\end{itemize}
\end{dfnz}

	\begin{teo}[Cortesani–Toader]\label{teorcortesani}
		Let \(\Omega \subseteq \mathbb{R}^n\) be a bounded open set with Lipschitz boundary, and let \(u \in GSBV^p(\Omega;\mathbb{R}^m)\) with \(p>1\).  
		Then there exists a sequence \(u_j \in \mathcal{W}(\Omega;\mathbb{R}^m) \cap C^\infty(\Omega\setminus\overline{J_u};\mathbb{R}^m)\) such that
		\[
		u_j \to u \text{ in measure}, 
		\qquad 
		\nabla u_j \to \nabla u \text{ in } L^p(\Omega),
		\]
		and
		\[
		\limsup_{j\to\infty}
		\int_{J_{u_j} \cap \overline{A}}
		g(x,u_j^+,u_j^-,\nu_{u_j})\, d\mathcal{H}^{n-1}
		\;\le\;
		\int_{J_u \cap \overline{A}}
		g(x,u^+,u^-,\nu_u)\, d\mathcal{H}^{n-1}
		\]
		for every open \(A \Subset \Omega\) and every upper semicontinuous function  
		\[
		g : \Omega \times \mathbb{R}^m \times \mathbb{R}^m \times \mathbb{S}^{n-1} \to [0,+\infty)
		\]
		satisfying \(g(x,a,b,\nu) = g(x,b,a,-\nu)\) for all arguments.  
		If moreover \(u \in L^\infty(\Omega)\), then one may assume \(\|u_j\|_{L^\infty(\Omega)} \le \|u\|_{L^\infty(\Omega)}\) for all \(j\).
	\end{teo}

    Finally, we introduce a further, more regular subclass of functions in $\mathcal{W}$ in dimension two.

    \begin{dfnz}\label{defwidehatW}
	Let $\Omega \subseteq \mathbb{R}^2$ be a bounded open set with Lipschitz boundary. We denote by $\widehat{\mathcal{W}}(\Omega;\mathbb{R}^m)$ the subclass of $\mathcal{W}(\Omega;\mathbb{R}^m)$ consisting of those $u$ for which each connected component of $J_u$ is either a single segment or the union of two segments sharing exactly one endpoint, with the convex hull of each component contained in $\Omega$.
    \end{dfnz}
	
	\begin{oss}\label{oss3}
        Assume that in Theorem \ref{teorcortesani} the function $g$ is such that $g(x,a,b,\nu)=\psi(a-b,\nu)$ for some upper semicontinuous function $\psi : \mathbb{R}^{m} \times \mathbb{S}^{1} \to [0,+\infty)$ satisfying the linear growth bound
			\[
			\psi(z,\nu) \le C\,(1+|z|)
			\qquad \text{for all $z \in \mathbb R^m$ and all  $\nu\in\mathbb{S}^{1}$}.
			\]
		Then, since \(\psi\) is independent of \(x\) and satisfies the growth bound, it is locally bounded near \(\partial\Omega\).  
		Thus, using \cite[Remark~3.2]{cortesani1999density}, for any open set \(A\subseteq\Omega\) we may replace \(\overline{A}\) with the relative closure of \(A\) in \(\Omega\). In particular, we can suppose that $A=\Omega$. Moreover, using \cite[Lemma 5.2]{dephilippis2017approximation}, we can always assume that $J_{u_j} \Subset \Omega$. Finally, if $n=2$, using \cite[Lemma 2.20]{almi2023brittlemembranes} and \cite[Theorem C]{dephilippis2017approximation} together with Corollary~\ref{corlimsup},
		we may choose the sequence $u_j$ so that
		\[
		u_j \in \widehat{\mathcal{W}}(\Omega;\mathbb{R}^{m})
		\cap C^{\infty}\!\bigl(\Omega \setminus \overline{J_u};\R^m\bigr)
		\]
		using a diagonal argument.
	\end{oss}

	\section{Approximation results}

    Throughout this section we work in the two-dimensional setting. We emphasize
    that the approximation results below are stated and proved in this dimension,
    which is sufficient for the membrane dimension-reduction problem addressed in
    this paper.

	The goal is to prove approximation results for $GSBV^p$-functions in the spirit of Theorem~\ref{teorcortesani}, with the additional requirement that the approximating sequence satisfies a maximal-rank constraint on the Jacobian. We present two results. The first, Theorem~\ref{teorapproxaff^*}, extends \cite[Theorem~3.1]{almi2023brittlemembranes} to more general surface energies and will be used in the construction of recovery sequences under the orientation-preserving constraint. The second, Proposition~\ref{propu_jbuona}, provides higher-order regularity of the approximating sequence away from the jump set and will be instrumental in the construction of recovery sequences under the incompressibility constraint.
    
	\begin{teo}\label{teorapproxaff^*}
		Let $\Omega \subseteq \mathbb{R}^2$ be a bounded open set with Lipschitz boundary, and let $1<p<+\infty$. 
		Let $u \in GSBV^p(\Omega;\mathbb{R}^m)$, with $m \ge 2$. 
		Then there exists a sequence of functions
		\[
		u_j \in \mathrm{Aff}^*(\Omega \setminus \overline{J_{u_j}};\mathbb{R}^m)
		\cap \widehat{\mathcal{W}}(\Omega;\mathbb{R}^m)
		\]
		such that
		\begin{itemize}
			\item[(i)] $u_j \to u$ in measure;
			\item[(ii)] $\nabla u_j \to \nabla u$ in $L^p(\Omega;\mathbb{M}^{m \times 2})$;
			\item[(iii)]  for all $\psi : \mathbb{R}^{m} \times \mathbb{S}^{1} \to [0,+\infty)$ upper semicontinuous functions which satisfy the symmetry
			\[
			\psi(z,\nu) = \psi(-z,-\nu) 
			\qquad \text{for all } (z,\nu),
			\]
			and the linear growth bound
			\[
			\psi(z,\nu) \le C\,(1+|z|)
			\qquad \text{for all } (z,\nu)\in\mathbb R^m \times \mathbb{S}^{1},
			\]
            and every $E \subseteq \Omega$ Borel set, it holds
			\[
			\limsup_{j \to\infty}
			\int_{J_{u_j} \cap E}
			\psi\bigl([u_j], \nu_{u_j}\bigr)\,
			d\mathcal{H}^{1}
			\;\le\;
			\int_{J_u \cap E}
			\psi\bigl([u], \nu_u\bigr)\,
			d\mathcal{H}^{1}.
			\]
		\end{itemize}
	\end{teo}
	The proof of Theorem \ref{teorapproxaff^*} follows the lines of \cite[Theorem~3.1]{almi2023brittlemembranes}, while also accounting carefully for the contribution of the surface energy term. To this aim, we need \cite[Lemma 3.2]{almi2023brittlemembranes} which we report here.
	\begin{lemma}\label{lemma01}
		Let $\Pi \subseteq \mathbb{R}^2$ be either a line segment or the union of two line segments sharing exactly one endpoint. 
		Then, for every sufficiently small $\delta > 0$, there exist a set $\Delta \subseteq \mathbb{R}^2$, depending on $\delta$, and a homeomorphism 
		\[
		\Phi_\delta : \mathbb{R}^2 \setminus \Pi \to \mathbb{R}^2 \setminus \Delta
		\]
		such that the following properties hold:
		\begin{itemize}
			\item[(i)] $\mathbb{R}^2 \setminus \Delta$ is a Lipschitz domain;
			\item[(ii)] $\Phi_\delta(x) = x$ for every $x \in \mathbb{R}^2 \setminus (\Pi)_\delta$;
			\item[(iii)] $\|\Phi_\delta - \mathrm{Id}\|_{W^{1,\infty}(\mathbb{R}^2 \setminus \Pi)} \to 0$ as $\delta \to 0$.
		\end{itemize}
	\end{lemma}

    \begin{oss}\label{ossPi}
    We briefly describe the geometry of the set $\Delta$ produced by the 
    construction of $\Phi_\delta$ in Lemma~\ref{lemma01}, following 
    \cite{almi2023brittlemembranes} (see also Figure~\ref{figure1}).
    
    \begin{itemize}
        \item If $\Pi$ is a single line segment, then $\Delta$ is bounded 
        by two curves $\Pi^+$ and $\Pi^-$: one is a line segment and the 
        other consists of two line segments meeting at exactly one point. 
        Moreover, $\partial \Pi^+ = \partial \Pi^- = \partial \Pi$, where $\partial \Pi$ and $\partial \Pi^\pm$ denote the relative boundaries of $\Pi$ and $\Pi^\pm$, respectively.
        
        \item If $\Pi$ is the union of two line segments sharing exactly 
        one endpoint, then $\Delta$ is bounded by two curves $\Pi^+$ and 
        $\Pi^-$, each of which is itself the union of two line segments 
        sharing exactly one endpoint. Again, 
        $\partial \Pi^+ = \partial \Pi^- = \partial \Pi$ (the $\partial$ denotes once again the relative boundary).
    \end{itemize}
    
    \noindent In both cases, the set $\Delta$ is a region enclosed between 
    $\Pi^+$ and $\Pi^-$, which are the images under $\Phi_\delta$ of the 
    two sides of $\Pi$, and whose endpoints coincide with those of $\Pi$.
    \end{oss}

     \begin{figure}[h!]
		\begin{tikzpicture}[scale=1.2]
			
			\begin{scope}
				
				\draw (0,0) circle (2cm);
				
				\draw (-1,0.3) -- (0,1) -- (1,0.3);
				
				\draw (-1,-0.3) -- (1,-0.3);
				
				\node at (1.3,0.7) {$\Pi_2$};
				\node at (1.3,-0.3) {$\Pi_1$};
				
			\end{scope}
			
			\begin{scope}[xshift=7cm]
				\draw (0,0) circle (2cm);
				
				\fill[blue!20] (-1,0.3) -- (0,1) -- (1,0.3) -- (0,1.5)-- cycle;
				\draw (-1,0.3) -- (0,1) -- (1,0.3);
				\draw (-1,0.3) -- (0,1.5) -- (1,0.3);
				
				\draw (-1,-0.3) -- (1,-0.3);
				\draw (-1,-0.3) -- (0,0.3)--(1,-0.3);
				\fill[blue!20] (-1,-0.3) -- (0,0.3)--(1,-0.3);-- cycle;
				\node at (1.3,0.7) {$\Pi^+_2$};
				\node at (0,0.6) {$\Pi^-_2$};
				\node at (0.7,0.1) {$\Pi^+_1$};
				\node at (0,-0.5) {$\Pi^-_1$};
			\end{scope}
			
		\end{tikzpicture} 
		\caption{Construction in Lemma \ref{lemma01}: The shaded part represents $\Delta$. }
        \label{figure1}
	\end{figure}
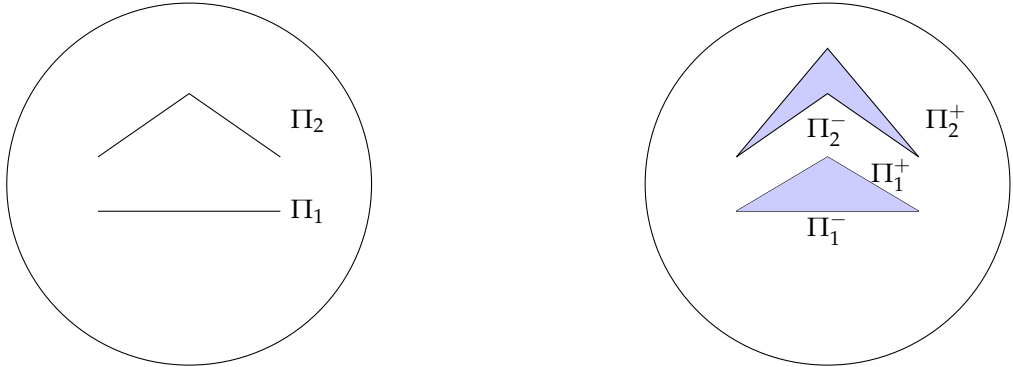

	\begin{proof}[Proof of Theorem~\ref{teorapproxaff^*}]

By a diagonal argument and Remark~\ref{oss3}, it suffices to prove the result for
\[
u \in \widehat{\mathcal{W}}(\Omega;\mathbb{R}^{m}) \cap C^{\infty}\!\bigl(\Omega \setminus \overline{J_u};\R^m\bigr).
\]

\textit{Step 1: Construction of the diffeomorphism.}
Let $\Pi_1, \dots, \Pi_k$ denote the connected components of $\overline{J_u}$, and for each $i = 1, \dots, k$, fix an open neighborhood $A_i \Subset \Omega$ of $\Pi_i$. We may assume the sets $A_i$ are pairwise disjoint.

Applying Lemma~\ref{lemma01} to each $\Pi_i$, we obtain $W^{1,\infty}$-piecewise affine homeomorphisms
\[
\Phi_i : A_i \setminus \Pi_i \;\longrightarrow\; A_i \setminus D_i,
\]
where $A_i \setminus D_i$ is a Lipschitz domain and $\Phi_i(x) = x$ outside a neighborhood of $\Pi_i$ compactly contained in $A_i$. By property~(iii) of Lemma~\ref{lemma01}, taking $\delta > 0$ sufficiently small, we may also assume that for every $i = 1, \dots, k$,
\begin{equation}\label{eq:det-bound}
\|D\Phi_i\|_{L^\infty(A_i \setminus \Pi_i)}
+ \|D\Phi_i^{-1}\|_{L^\infty(A_i \setminus D_i)} \le 3,
\end{equation}
and
\begin{equation}\label{eq:det-pos}
\det M \ge \tfrac{1}{2}
\quad \text{for every } M \in \partial\Phi_i(x),\ x \in A_i \setminus \Pi_i.
\end{equation}
Define the global map $\Phi : \Omega \setminus \overline{J_u} \to V$ by
\[
\Phi(x) :=
\begin{cases}
x, & \text{if } x \in \Omega \setminus \bigcup_{i=1}^{k} A_i, \\[4pt]
\Phi_i(x), & \text{if } x \in A_i \setminus \Pi_i, \quad i = 1, \dots, k,
\end{cases}
\]
where
\[
V := \Omega \setminus \bigcup_{i=1}^{k} D_i.
\]
Since the sets $A_i$ are pairwise disjoint, $V$ is a bounded open Lipschitz domain, and \eqref{eq:det-bound}--\eqref{eq:det-pos} yield the global bounds
\begin{equation}\label{eq:global-det-bound}
\|D\Phi\|_{L^\infty(\Omega \setminus \overline{J_u})}
+ \|D\Phi^{-1}\|_{L^\infty(V)} \le 3,
\qquad
\det M \ge \tfrac{1}{2}
\quad \text{for a.e. } x \in \Omega \setminus \overline{J_u}.
\end{equation}
By Remark~\ref{ossPi}, the image of each $\Pi_i$ under $\Phi$ consists of two curves $\Pi_i^+$ and $\Pi_i^-$, each being a union of finitely many line segments; in particular, $\partial V$ is piecewise linear.

\textit{Step 2: Definition of the auxiliary function and approximation in $V$.}
Set $v : V \to \mathbb{R}^m$ by
\[
v(x) := u\bigl(\Phi^{-1}(x)\bigr), \quad x \in V.
\]
By construction and \eqref{eq:global-det-bound}, we have
\[
v \in W^{1,\infty}(V;\mathbb{R}^m) \cap C(\overline{V};\mathbb{R}^m).
\]
Since $V$ is a bounded open Lipschitz domain and $v \in W^{1,\infty}(V;\mathbb{R}^m)$, reasoning as in the proof of \cite[Theorem~3.1]{almi2023brittlemembranes}, there exists a sequence $\{w_j\} \subset \mathrm{Aff}^*(V;\mathbb{R}^m)$ such that for every $p \in (1,\infty)$
\begin{equation}\label{eq:W1p-conv}
\|v - w_j\|_{W^{1,p}(V;\mathbb{R}^m)} \to 0 \quad \text{as } j \to +\infty.
\end{equation}
By choosing $p > 2 = \dim V$, in particular we have
\begin{equation}\label{eq:unif-conv}
w_j \to v \quad \text{uniformly on } \overline{V}.
\end{equation}

\textit{Step 3: Definition of the approximating sequence.}
For $x \in \Omega \setminus \overline{J_u}$, set
\[
u_j(x) := w_j\bigl(\Phi(x)\bigr).
\]
By construction, $u_j \in \mathrm{Aff}^*(\Omega \setminus \overline{J_{u_j}};\mathbb{R}^m) \cap \widehat{\mathcal{W}}(\Omega;\mathbb{R}^m)$, and $J_{u_j} \subseteq J_u$.
Properties~(i) and~(ii) follow from \eqref{eq:W1p-conv} and \eqref{eq:global-det-bound} by the same change-of-variables argument as in \cite[Theorem~3.1]{almi2023brittlemembranes}.

\textit{Step 4: Proof of property~(iii).}
By construction (up to a consistent orientation of the normal vectors), for every $i=1,\dots,k$, and every $x \in \Pi_i$, 
\[
u_j^\pm(x) = w_j(\Phi^\pm(x))
\qquad \text{and} \qquad
u^\pm(x) = v(\Phi^\pm(x)).
\]
The uniform convergence~\eqref{eq:unif-conv} therefore gives
\[
u_j^\pm \to u^\pm \quad \text{uniformly on } J_u.
\]
In particular, $[u_j] \to [u]$ uniformly on $J_u$. Combined with the fact that $J_{u_j} \subset J_u$, we infer
\begin{equation}\label{eq:sym-diff}
\mathcal{H}^1\bigl(J_u \,\triangle\, J_{u_j}\bigr) \to 0 \quad \text{as } j \to +\infty.
\end{equation}
Collecting the uniform convergence of the traces, (ii), and \eqref{eq:sym-diff}, we conclude that
\[
|D u_j - D u|(\Omega) \to 0  \quad \text{as } j \to +\infty.
\]
Property~(iii) then follows directly from Corollary~\ref{corlimsup}.
\end{proof}

    Since piecewise affine homeomorphisms do not provide sufficient regularity for the construction of recovery sequences under the incompressibility constraint, we adopt an alternative approach based on smooth diffeomorphisms mapping the cracked domain onto a Lipschitz domain. We recall that given a set $U \subset \R^n$ we indicate its $\delta$-neighborhood with $(U)_\delta$.

	\begin{lemma}\label{lemmadiffeolisciokernelvariabile}
		Let $S \subseteq \mathbb{R}^2$ be a regular curve such that
		\[
		S = \{ (x_1, g(x_1)) \mid x_1 \in [a,b] \},
		\]
		where $g \in C^\infty([a,b])$. Then, for every sufficiently small $\delta > 0$, there exists a function
		\[
		\Phi_\delta : \mathbb{R}^2 \setminus \overline{S} \to \mathbb{R}^2 \setminus \Delta_\delta
		\]
		such that:
		\begin{itemize}
			\item[(i)] $\Phi_\delta \in C^1(\mathbb{R}^2 \setminus \overline{S}; \mathbb{R}^2)$, $\nabla \Phi_\delta \in L^1(\mathbb{R}^2; \mathbb{M}^{2 \times 2})$, and for every integer $k \geq 1$ one has
			\[
			\nabla \Phi_\delta \in W^{k,\infty}\!\left(\mathbb{R}^2 \setminus \overline{\bigl(S \cup B_\delta(g(a)) \cup B_\delta(g(b))\bigr)};\, \mathbb{M}^{2 \times 2}\right);
			\]
			
			\item[(ii)] $\mathbb{R}^2 \setminus \Delta_\delta$ is a set with Lipschitz boundary;
			
			\item[(iii)] $\Phi_\delta(x) = x$ for every $x \in \mathbb{R}^2 \setminus (S)_{2\delta}$;
			
			\item[(iv)] $\|\Phi_\delta - \operatorname{Id}\|_{W^{1,\infty}(\mathbb{R}^2 \setminus \overline{S}; \mathbb{R}^2)} \to 0$ as $\delta \to 0$.
		\end{itemize}
	\end{lemma}

    Since we believe this to be a standard result but were unable to locate a precise reference in the literature, and since the argument relies on convolutions with variable kernels, a tool not used elsewhere in the paper, we include a sketch of its proof in the Appendix.

	\begin{oss}\label{osslemmakernelvariabile}
		As a consequence of {(i)}, {(iii)}, and {(iv)} of Lemma \ref{lemmadiffeolisciokernelvariabile}, we also have that,
		for $\delta > 0$ sufficiently small, the inverse map $\Phi_\delta^{-1}$ satisfies
		{(iii)}, {(iv)}, and
		\begin{itemize}
			\item[\emph{(i$'$)}] 
			$\Phi_\delta^{-1} \in C^\infty(\mathbb{R}^2 \setminus \Delta_\delta;\mathbb{R}^2)$,
			$\nabla \Phi_\delta^{-1} \in L^\infty(\mathbb{R}^2;\mathbb{R}^{2 \times 2})$, and, for every integer
			$k \geq 2$,
			\[
			\Phi_\delta^{-1} \in W^{k,\infty}\!\left(
			\mathbb{R}^2 \setminus \bigl[\Delta_\delta \cup \Phi_\delta(B_\delta(g(a))) \cup \Phi_\delta(B_\delta(g(b)))\bigr];
			\mathbb{R}^2
			\right).
			\]
		\end{itemize}
		
	\end{oss}

    Using the diffeomorphisms constructed in Lemma~\ref{lemmadiffeolisciokernelvariabile}, we can now establish the second approximation result for $GSBV^p$ functions. The dimensional assumption $m>n$ is essential: the extra codimension provides the flexibility needed to approximate by locally orientation-preserving maps (which are not piecewise affine).

	\begin{prop}\label{propu_jbuona}
		Let $\Omega \subseteq \mathbb{R}^2$ be a bounded set with Lipschitz boundary and let $p \in (1,+\infty)$. For every
		$u \in \widehat{\mathcal{W}}(\Omega;\mathbb{R}^m) \cap \mathrm{Aff}^*(\Omega \setminus J_u;\mathbb{R}^m)$ with $m \geq 3$, there exists a sequence
		$$u_j \in SBV^p(\Omega;\mathbb{R}^m) \cap W^{k,\infty}(\Omega \setminus \overline{J_{u_j}};\mathbb{R}^m)$$ 
        for every $k \in \mathbb{N}$ such that:
		\begin{itemize}
			\item[(i)] $\det\big((\nabla u_j(x))^T (\nabla u_j(x))\big) \geq \theta$ for every $j \geq 1$ and for a.e.\ $x \in \Omega$, with
			$\theta = \theta(u) > 0$;
			
			\item[(ii)] $u_j \to u$ in $L^p(\Omega;\mathbb{R}^m)$ and $\nabla u_j \to \nabla u$ in
			$L^p(\Omega;\mathbb{M}^{m \times 2})$ as $j \to +\infty$;
			
			\item[(iii)] $J_{u_j} \Subset \Omega$ consists of a finite union of regular curves for every $j \geq 1$, and
			\[
			\lim_{j \to +\infty} \mathcal{H}^1(J_{u_j} \,\triangle\, J_u) = 0;
			\]
			
			\item[(iv)] $\|u_j\|_{L^\infty(\Omega;\mathbb{R}^m)} \leq \|u\|_{L^\infty(\Omega;\mathbb{R}^m)} + 1$ and
			$\|\nabla u_j\|_{L^\infty(\Omega;\mathbb{M}^{m \times 2})} \leq 8 \|\nabla u\|_{L^\infty(\Omega;\mathbb{M}^{m \times 2})}$;
			
			\item[(v)] For every $\varepsilon > 0$, and for $j$ sufficiently large, there exists an open set $A_\varepsilon$ such that
        \[
        J_{u_j} \cap A_\varepsilon \subset J_u \cap A_\varepsilon,
        \qquad 
        \mathcal{H}^1(J_{u_j} \setminus A_\varepsilon) < \varepsilon,
        \]
    and the traces $u_j^\pm$ are equicontinuous on $J_u \cap A_\varepsilon$. Moreover, denoting by $h$ the number of connected components of $J_u$ and by $h'$ the number of connected components of $J_u \cap A_\varepsilon$, we have
    \[
    h \leq h' \leq 2h .
    \]
		\end{itemize}
	\end{prop}
	\begin{proof}
    We divide the proof in several steps. Take $\sigma > 0$ sufficiently small (to be
	determined later).
    
    \textit{Step 1: smoothing the angle.}
	Let $\Pi \Subset \Omega$ be a connected component of $\overline{J_u}$. Since the connected 
    components of $J_u$ are well separated, we may argue locally and assume without loss of 
    generality that $J_u$ consists of a single component.  Since 
    $u \in \widehat{\mathcal{W}}(\Omega;\mathbb{R}^m)$, the set $\Pi$ is either a segment or the 
    union of two segments that meet at a single endpoint $x_0$. Let us consider the latter case, as the first one is in fact simpler. We can find a 
    bi-Lipschitz homeomorphism $\Xi : \mathbb{R}^2 \to \mathbb{R}^2$ such that $\Xi(\Pi) = S \Subset \Omega$,
    where $S$ is a smooth curve that can be written (up to a rotation) as the graph of a 
    $C^\infty$ function. Moreover, $\Xi$ can be chosen so that $\Xi = \mathrm{Id}$ outside $B_\sigma(x_0)$ and so that there exists $C>0$ depending only on $\Pi$ with the property that
    \begin{equation*}
         \mathrm{Lip}(\Xi), \ \mathrm{Lip}(\Xi^{-1}) \leq C.
    \end{equation*}
	Let us define
	\[
	w := u \circ \Xi^{-1} \in SBV^p(\Omega;\mathbb{R}^m).
	\]

    \textit{Step 2: mapping into the Lipschitz domain.}
	By definition, $J_w$ satisfies the assumptions of Lemma \ref{lemmadiffeolisciokernelvariabile}. Hence, we can find a $W^{1,\infty}$-diffeomorphism
	\[
	\Phi : \Omega \setminus J_w \to \Omega \setminus \Delta =: U
	\]
	satisfying properties (i)--(iv) of Lemma \ref{lemmadiffeolisciokernelvariabile}, such that $\Delta \subset (J_w)_\sigma$, and 
	\begin{equation}\label{3.2.8}
    \begin{aligned} 
    &\|\Phi - \text{Id}\|_{W^{1,\infty}(\Omega \setminus J_w;\mathbb{R}^2)}
	+
	\|\Phi^{-1} - \text{Id}\|_{W^{1,\infty}(U;\mathbb{R}^2)}
	\leq \sigma, \\
    &\det(\nabla \Phi(x)) \geq 1 - \sigma
	\quad \text{for every } x \in \Omega \setminus \overline{J_w}.
    \end{aligned}
    \end{equation}
	Let $V \subset \Omega$ be the open set such that
	\[
	\Phi \in W^{k,\infty}(\Omega \setminus \overline{J_w \cup V};\mathbb{R}^2).
	\]
	By Lemma~\ref{lemmadiffeolisciokernelvariabile}, the set $V$ is given by
	\begin{equation}\label{3.2.9}
		V = \bigcup_{i=1}^2 B_\sigma(x_i),
	\end{equation}
	where $\{x_i\}_{i=1}^2 \subset \Omega$ are the two endpoints of $J_w$.
	Possibly taking $\sigma$ smaller, we may additionally assume that $V \Subset \Omega$ and that
	\eqref{3.2.9} is a disjoint union. Furthermore, by Remark~\ref{osslemmakernelvariabile}, we also have
	\[
	\Phi^{-1} \in W^{k,\infty}(U \setminus \Phi(\overline{V});\mathbb{R}^2)
	\quad \text{for every } k \geq 1.
	\]
	Setting
	\[
	v := w \circ \Phi^{-1},
	\]
	we have $v \in W^{1,\infty}(U;\mathbb{R}^m)$ and, by construction,
	\[
	\|v\|_{W^{1,\infty}(U;\mathbb{R}^m)}
	\leq 2 \|u\|_{W^{1,\infty}(\Omega \setminus J_u;\mathbb{R}^m)}.
	\]
	Since $u \in \mathrm{Aff}^*(\Omega;\mathbb{R}^m)$, there exists $\eta > 0$ such that
	\[
	\det(M^T M) \geq \eta
	\quad \text{for every } M \in \partial u(x), \ \text{for every } x \in \Omega \setminus J_u.
	\]

    \textit{Step 3: Smoothing the $\mathrm{Aff}^*$ function.}
Let $\delta_j \to 0$ and let $\{\rho_\tau\}_{\tau > 0}$ be a family of standard mollifiers. 
Setting $(U)_{-\delta} := \{x \in U : \mathrm{dist}(x,\partial U) > \delta\}$, we define
\[
v_j := \rho_{\delta_j} * v \in C^\infty\!\left((U)_{-2\delta_j};\mathbb{R}^m\right).
\]
By standard properties of convolution,
\begin{equation}\label{3.2.11}
\|v_j\|_{L^\infty((U)_{-\delta_j};\mathbb{R}^m)} \leq \|v\|_{L^\infty(U;\mathbb{R}^m)},
\qquad
\|\nabla v_j\|_{L^\infty((U)_{-\delta_j};\mathbb{R}^{m\times 2})}
\leq \|\nabla v\|_{L^\infty(U;\mathbb{R}^{m\times 2})}.
\end{equation}
Moreover, we preliminary observe that since $u \in \mathrm{Aff}^* (\Omega;\R^m)$, we can take $\sigma>0$ and $\delta_j>0$ sufficiently small so that for every $z \in \Omega \setminus \overline{J_u}$ there exists $\zeta \in \Omega \setminus \overline{J_u}$ such that
\begin{equation}\label{ufff}
\nabla u(y) \in \partial u(\zeta) \quad \text{for every } y \in B_{(1+\sigma)\delta_j}\!\left(z\right) 
\cap \Omega \setminus \overline{J_u}.
\end{equation}

Let us take $\sigma>0$ and $\delta_j>0$ small so that \eqref{ufff} is verified. We claim that for every 
$x \in (U)_{-2\delta_j} \setminus (\Phi(\Xi(B_\sigma(x_0))))_{\delta_j}$,
\[
\det\!\big((\nabla v_j(x))^T \nabla v_j(x)\big) \geq \frac{\eta}{2}.
\]
Fix such an $x$. Since $x \notin (\Phi(\Xi(B_\sigma(x_0))))_{\delta_j}$, for every 
$y \in B_{\delta_j}(x)$ we have $\Xi^{-1}(\Phi^{-1}(y)) = \Phi^{-1}(y)$, hence
\begin{equation}\label{3.2.12}
\nabla v_j(x)
= \int_{B_{\delta_j}(x)} \nabla u\!\left(\Phi^{-1}(y)\right)\nabla \Phi^{-1}(y)
\,\rho_{\delta_j}(x-y)\,dy.
\end{equation}
By~\eqref{3.2.8}, $\mathrm{Lip}(\Phi^{-1}) \leq 1+\sigma$, thus
\begin{equation}\label{3.2.14}
\Phi^{-1}(B_{\delta_j}(x)) \subset B_{(1+\sigma)\delta_j}\!\left(\Phi^{-1}(x)\right).
\end{equation}
There exists $\zeta \in \Omega \setminus \overline{J_u}$ such that
\begin{equation}\label{3.2.14.5}
\nabla u(y) \in \partial u(\zeta)
\quad \text{for every } y \in B_{(1+\sigma)\delta_j}\!\left(\Phi^{-1}(x)\right) 
\cap \Omega \setminus \overline{J_u}.
\end{equation}
Using the piecewise affine structure of $u$, \eqref{3.2.14}, and~\eqref{3.2.14.5}, we find 
a finite convex combination
\begin{equation}\label{3.2.15}
\int_{B_{\delta_j}(x)} \nabla u\!\left(\Phi^{-1}(y)\right)\rho_{\delta_j}(x-y)\,dy
= \sum_{\ell=1}^L \lambda_\ell \nabla u(y_\ell) \in \partial u(\zeta).
\end{equation}
Combining~\eqref{3.2.8}, \eqref{3.2.11}, \eqref{3.2.12}, and~\eqref{3.2.15}, we obtain
\[
\det\!\big((\nabla v_j(x))^T \nabla v_j(x)\big)
\geq \eta - \omega\!\left(\sigma C \|\nabla u\|_{L^\infty(\Omega;\mathbb{M}^{m\times 2})}^2\right),
\]
where $C > 0$ is independent of $x$ and $j$, and $\omega$ is the modulus of continuity of 
the determinant on $\{F \in \mathbb{R}^{m\times 2} : |F| \leq 
C\|\nabla u\|_{L^\infty}^2 + 1\}$. Since $\omega(s) \to 0$ as $s \to 0$, taking $\sigma$ 
smaller if necessary yields the claim.

Since $U$ has Lipschitz boundary, there exists a sequence of diffeomorphisms 
$\varphi_j \in C^\infty(U;U)$ with
\[
\varphi_j(U) \Subset (U)_{-2\delta_j}
\quad \text{and} \quad
\|\varphi_j - \mathrm{Id}\|_{C^1(U;\mathbb{R}^2)} \to 0.
\]
Setting $\widetilde{v}_j(x) := v_j(\varphi_j(x))$, we have 
$\widetilde{v}_j \in C^\infty(U;\mathbb{R}^m)$ with
\[
\|\widetilde{v}_j\|_{L^\infty(U;\mathbb{R}^m)} \leq \|v\|_{L^\infty(U;\mathbb{R}^m)},
\qquad
\|\nabla \widetilde{v}_j\|_{L^\infty(U;\mathbb{M}^{m\times 2})}
\leq 2\|\nabla v\|_{L^\infty(U;\mathbb{M}^{m\times 2})}.
\]
Moreover, for $\delta_j$ sufficiently small and every 
$x \in U \setminus (\Phi(\Xi(B_\sigma(x_0))))_{2\delta_j}$,
\begin{equation}\label{3.2.16}
\det\!\big((\nabla \widetilde{v}_j(x))^T \nabla \widetilde{v}_j(x)\big)
= \big[\det(\nabla\varphi_j(x))\big]^2
\det\!\big((\nabla v_j(\varphi_j(x)))^T \nabla v_j(\varphi_j(x))\big)
\geq \frac{\eta}{4}.
\end{equation}
Finally, $v_j \to v$ in $W^{1,p}((U)_{-2\delta_j};\mathbb{R}^m)$ and, by the properties of 
$\varphi_j$,
\[
\widetilde{v}_j \to v \quad \text{in } W^{1,p}(U;\mathbb{R}^m) \text{ for every } p \in (1,\infty).
\]
Fixing $p > 2$, this convergence is uniform on $\overline{U}$.

\textit{Step 4: Definition of the approximating sequence.}
Let $z \in \Omega \setminus \overline{J_u}$ be such that $\nabla u(z) \in \partial u(z)$. 
For $\delta_j$ small we have 
$\Xi^{-1}(\Phi^{-1}((\Phi(\Xi(B_\sigma(x_0))))_{2\delta_j})) \subset B_{2\sigma}(x_0)$.
Setting $\mathcal{V} := \bigcup_{i=0}^2 B_{2\sigma}(x_i)$, we define
\[
u_{\sigma,j}(x) :=
\begin{cases}
\widetilde{v}_j(\Phi(x)) & x \in \Omega \setminus (J_w \cup \mathcal{V}), \\[4pt]
\nabla u(z)(x - x_i) & x \in B_{2\sigma}(x_i),\ i = 0,1,2.
\end{cases}
\]
Since $\Phi \in W^{k,\infty}(\Omega \setminus \overline{J_w \cup V};\mathbb{R}^2)$ for every 
$k \geq 1$ and $\widetilde{v}_j \in C^\infty(U;\mathbb{R}^m)$, we have
\[
u_{\sigma,j} \in W^{k,\infty}\!\left(\Omega \setminus \overline{J_{u_{\sigma,j}}};\mathbb{R}^m\right)
\quad \text{for every } k \geq 1,
\]
and
\[
J_{u_{\sigma,j}} \subset (J_w \setminus \mathcal{V}) \cup \partial\mathcal{V} \Subset \Omega,
\]
so $u_{\sigma,j} \in SBV^p(\Omega;\mathbb{R}^m)$.

\textit{Verification of~(i).}
By~\eqref{eq:det-pos} and~\eqref{3.2.16}, for $\sigma$ sufficiently small and every 
$x \in \Omega \setminus \overline{(\mathcal{V} \cup J_{u_{\sigma,j}})}$,
\[
\det\!\big((\nabla u_{\sigma,j}(x))^T \nabla u_{\sigma,j}(x)\big)
= \big[\det(\nabla\Phi(x))\big]^2
\det\!\big((\nabla \widetilde{v}_j(\Phi(x)))^T \nabla \widetilde{v}_j(\Phi(x))\big)
\geq \frac{\eta}{8},
\]
so~(i) holds with $\theta = \eta/8$.

\textit{Verification of~(iii).}
By construction,
\begin{equation}\label{3.2.17}
\mathcal{H}^1(J_{u_{\sigma,j}} \,\triangle\, J_w)
\leq \mathcal{H}^1(J_w \cap \mathcal{V}) + 12\pi\sigma.
\end{equation}
A diagonal argument in $j$ and $\sigma \to 0$ yields~(iii). We henceforth drop $\sigma$ from 
the notation.

\textit{Verification of~(v).}
Fix $\varepsilon > 0$ and let $A_\varepsilon$ be an open set with 
$\mathcal{H}^1(J_u \setminus A_\varepsilon) < \varepsilon$ and $A_\varepsilon \cap \mathcal{V} = \emptyset$. 
Since $\widetilde{v}_j \to v$ uniformly, the traces $u_j^\pm$ converge uniformly to $u^\pm$ on 
$J_u \cap A_\varepsilon$, giving equicontinuity. The number of connected components of 
$J_{u_j} \cap A_\varepsilon$ is exactly two, which gives~(v).

\textit{Verification of~(iv).}
For $\sigma$ sufficiently small and every $j \geq 1$,
\[
\|u_j\|_{L^\infty(\Omega;\mathbb{R}^m)}
\leq \|\widetilde{v}_j\|_{L^\infty(U;\mathbb{R}^m)}
+2 \sigma\|\nabla u\|_{L^\infty(\Omega;\mathbb{M}^{m\times 2})}
\leq \|u\|_{L^\infty(\Omega;\mathbb{R}^m)} + 1,
\]
\[
\|\nabla u_j\|_{L^\infty(\Omega;\mathbb{R}^{m\times 2})}
\leq \max\!\Big\{
\|\nabla u\|_{L^\infty},\,
\|\nabla \widetilde{v}_j\|_{L^\infty}\|\nabla\Phi\|_{L^\infty}
\Big\}
\leq 8\|\nabla u\|_{L^\infty(\Omega;\mathbb{M}^{m\times 2})},
\]
which yields~(iv).

\textit{Verification of~(ii).}
It suffices to prove~(ii) with $w$ in place of $u$ and with $\widetilde{v}_j(\Phi)$ in place of $u_j$, as (ii) will then follow from 
the definition of $u_j$, the definition of $w$, and the properties of $\Xi$. Since $\widetilde{v}_j \to v$ in 
$W^{1,p}(U;\mathbb{R}^m)$, a change of variables using~\eqref{3.2.8} gives
\[
\int_\Omega |\widetilde{v}_j(\Phi(x)) - w(x)|^p\,dx
\leq (1+\sigma)\|\widetilde{v}_j - v\|_{L^p(U)}^p \to 0.
\]
For the gradient, writing
\[
\nabla(\widetilde{v}_j \circ \Phi) - \nabla w
= \nabla\widetilde{v}_j(\Phi)\,(\nabla\Phi - I)
+ \big(\nabla\widetilde{v}_j(\Phi) - \nabla w\,\nabla\Phi^{-1}(\Phi)\big)
+ \nabla w\,(\nabla\Phi^{-1}(\Phi) - I),
\]
and using~\eqref{3.2.8} and another change of variables, we obtain
\[
\begin{aligned}
\int_\Omega |\nabla(\widetilde{v}_j(\Phi(x))) - \nabla w(x)|^p\,dx
&\leq 4^p(1+\sigma)\|\nabla\widetilde{v}_j - \nabla v\|_{L^p(U;\mathbb{M}^{2 \times m})}^p \\
&\quad + 4^{p+2}\sigma^p\,\mathcal{L}^2(\Omega)\,\|\nabla w\|_{L^\infty(\Omega;\mathbb{M}^{2 \times m})}^p.
\end{aligned}
\]
Letting $j \to +\infty$ and $\sigma \to 0$ yields $\nabla(\widetilde{v}_j \circ \Phi) \to \nabla w$ 
in $L^p(\Omega;\mathbb{M}^{m\times 2})$, completing the proof of~(ii).
\end{proof}

	\section{Dimension Reduction under Orientation-Preservation Constraints}
In this section we focus on the orientation preserving case. 
In this setting we recall that $W : \mathbb{M}^{3 \times 3} \to [0,+\infty]$ satisfies $(A_1)$--$(A_4)$, and $\psi \colon (\R^3 \setminus \{0\}) \times \R^3 \to [0,+\infty)$ satisfies $(B_1)$--$(B_5)$.

Our aim is to prove the $\Gamma$-convergence of the family of functionals $\rho^{-1}\mathcal{G}_\rho$ defined in \eqref{intro resc func}.

	\subsection{Properties of $W_0$ and $\psi_0$, and lower bound} 
	
	In order to show the $\Gamma$-convergence, we begin by listing some properties of $W_0$ and $\psi_0$ that will be useful later. The proofs of the properties of $W_0$ are given respectively in \cite[Lemma 2.4]{hafsa2008nonlinear2} and \cite[Lemma 4.11]{almi2023brittlemembranes}.
	\begin{lemma}\label{lemma2}
		Let $W : \mathbb{M}^{3 \times 3} \to [0,+\infty]$ satisfy $(A_1)$--$(A_4)$. Then the following facts hold:
		\begin{itemize}
			\item[(1)] $W_0$ is continuous as an extended–valued function and satisfies $(A_3)$ with the same
			constants $C_1$ and $p$ as $W$;
			
			\item[(2)] for $F \in \mathbb{M}^{3 \times 2}$, one has $W_0(F) = +\infty$ if and only if $F^1 \wedge F^2 = 0$;
			
			\item[(3)] for every $\delta > 0$, there exists $c_\delta > 0$ such that for every $F \in \mathbb{M}^{3 \times 2}$, if 
			$|F^1 \wedge F^2| \ge \delta$, then
			\[
			W_0(F) \le c_\delta \,(1 + |F|^p).
			\]
		\end{itemize}
	\end{lemma}
	\begin{lemma}\label{lemma3}
		Let $W_0 \in C(\mathbb{M}^{3 \times 2}; [0,+\infty])$ be a $p$-coercive function satisfying (2) and (3)
		of Lemma \ref{lemma2}. Then:
		\begin{itemize}
			\item $\mathcal R W_0$ is finite valued, continuous, and $p$-coercive;
			
			\item there exists $C > 0$ such that 
			\[
			\mathcal R W_0(F) \le C(1 + |F|^p)
			\]
			for every $F \in \mathbb{M}^{3 \times 2}$;
			
			\item $QW_0$ is finite valued, continuous, $p$-coercive, and rank-one convex.
		\end{itemize}	
	\end{lemma}

    We now show that $\psi_0$ enjoys the same properties of $\psi$.
    
	\begin{lemma}\label{lemma4}
		Let 	$\psi : \mathbb{R}^3 \setminus \{0\} \times \mathbb{S}^2 \to [0,+\infty)$ satisfy $(B_1)-(B_5)$. Then, $\psi_0$ also satisfies $(B_1)-(B_5)$. 
	\end{lemma}
	\begin{proof}
		We observe that properties \((B_2)\), \((B_4)\), and \((B_5)\) follow directly from the definition, together with the fact that the infimum of upper semicontinuous functions is upper semicontinuous. So it remains to prove $(B_1)$ and $(B_3)$. As the 1-homogeneity is conserved for $\psi_0$ by definition, we can consider, without loss of generality that the second variable belongs to $\mathbb S^2$. 
		To prove property $(B_1)$, let us assume, without loss of generality, that in $(B_1)$ we have $\sigma \le 1$.
		Consider two points $z_1, z_2 \in \mathbb{R}^3$ and $\nu \in \mathbb{S}^2$. Then, given $\zeta \in \R$,
		\[
		(1 - \sigma(|z_1 - z_2|))\, \psi_0(z_1,\nu)
		\le (1 - \sigma(|z_1 - z_2|))\, \psi(z_1,\nu,\zeta)
		\le (1 + \sigma(|z_1 - z_2|))\, \psi(z_2,\nu,\zeta),
		\]
		by property $(B_1)$ of $\psi$. Taking the infimum over $\zeta$ on the right-hand side, we obtain
		\[
		\psi_0(z_1,\nu) - \psi_0(z_2,\nu)
		\le \sigma(|z_1 - z_2|)\bigl(\psi_0(z_1,\nu) + \psi_0(z_2,\nu)\bigr).
		\]
		Exchanging the roles of $z_1$ and $z_2$ yields property $(B_1)$.
		
		\medskip
		
		For $(B_3)$, note that by property $(B_3)$ of $\psi$, for every $z \in \mathbb{R}^3$ and $\nu \in \mathbb{S}^2$,
		\[
		C_3 \phi(z)
		\le C_3 \sqrt{1 + |\zeta|^2}\,\phi(z)
		\le \psi(z,\nu,\zeta).
		\]
		Taking the infimum over $\zeta$ gives
		\[
		C_3\,\phi(z) \le \psi_0(z,\nu).
		\]
		
		To obtain the upper bound, we simply observe that
		\[
		\psi_0(z,\nu) \le \psi(z,\nu,0) \le C_4\,\phi(z),
		\]
		which concludes the proof of $(B_3)$.
	\end{proof}
	\begin{oss}\label{osspsi_0}
		Reasoning as in the case of $\psi$, we observe that $\psi_0$ is also continuous in $z$.
		Furthermore, by Lemma \ref{lemma3} and Lemma \ref{lemma4}, both $\mathcal{Q}W_0$ and $\psi_0$
		satisfy the properties of Theorem~\ref{teor69}.
	\end{oss}
    
	We now prove the lower bound. The proof of the $\Gamma\text{-}\liminf$ inequality follows, although with some modifications, from the one without determinant constraints presented in \cite[Proposition 2.4]{braides2001brittle}
	
	\begin{prop}\label{prop1}
		Let $\Sigma$ be a bounded open subset of $\mathbb{R}^2$ with Lipschitz boundary and let $1<p<+\infty$. Let $W$ be a function satisfying $(A_1)$–$(A_4)$ and let $\psi$ be a function satisfying $(B_1)$–$(B_5)$. Then, for every sequence $u_\rho \in GSBV^p(\Sigma_1;\mathbb{R}^3)$ such that $u_\rho \to u$ in measure on $\Sigma_1$, one has
		\[
		\liminf_{\rho\to 0}\rho^{-1}\mathcal{G}_{\rho}(u_\rho)\ge\mathcal{G}_0(u).
		\]
	\end{prop}
	
	\begin{proof}
		If $\liminf_{\rho\to 0}\rho^{-1}\mathcal{G}_p(u_\rho)=+\infty$, there is nothing to prove. Assume now that 
		$$\liminf_{\rho\to 0}\rho^{-1}\mathcal{G}_p(u_\rho)<+\infty.
		$$ Up to extracting a subsequence, we may assume that the $\liminf$ is actually a limit and that there exists a constant $C>0$ such that $\sup\{\rho^{-1}\mathcal{G}_{\rho}(u_\rho)\}\le C$. Then, by $(A_3)$ and $(B_3)$, we obtain
		\begin{equation}\label{eqcompattezza}
			\int_{\Sigma_1} |\nabla u_\rho|^p + \mathcal{H}^2(J_{u_{\rho}}) \le C',
		\end{equation}
		so we may apply Ambrosio's compactness theorem \cite[Theorem 4.8]{ambrosio2000functions} and deduce that there exists $v \in GSBV^p(\Sigma_1;\mathbb{R}^3)$ such that, up to a further subsequence, $u_{\rho} \to v$ weakly in $GSBV^p(\Sigma_1,\mathbb{R}^3)$, that is, $u_\rho \to v$ pointwise a.e., and \eqref{eqcompattezza} holds. Hence $u=v$ a.e. in $\Sigma_1$. Moreover, by $(A_3)$, for every $\rho>0$,
		\[
		C \ge \rho^{-1}\mathcal{G}_{\rho}(u_\rho) \ge C_1\int_{\Sigma_1}\bigg(|\nabla_\alpha u_{\rho}|^p+\frac{1}{\rho^p}|\partial_3 u_\rho|^p\bigg)\,dx-C_2.
		\]
		This, together with the weak convergence of $\nabla u_\rho$ to $\nabla u$ in $L^p(\Sigma_1,\mathbb{M}^{3\times 3})$, implies that
		\[
		\int_{\Sigma_1} |\partial_3 u|^p \, dx\le \liminf_{\rho \to 0} \int_{\Sigma_1}|\partial_3 u_{\rho}|^p \,dx =0.
		\]
		Thus, $\partial_3 u=0$ a.e. in $\Sigma_1$. Furthermore, for every $\widetilde{\rho}\ge 0$, by $(B_3)$, 
		\[
		\begin{aligned}
			C &\ge \liminf_{\rho \to 0}\rho^{-1}\mathcal{G}_{\rho}(u_\rho) \ge \liminf_{\rho \to 0}\int_{J_{u_\rho}} \psi_{\rho}([u_\rho],\nu_{u_\rho})  \, d\mathcal{H}^2 \\
			&= \liminf_{\rho \to 0}\int_{J_{u_\rho}} \psi\bigg([u_\rho],\frac{(\nu_{\alpha,u_\rho}, \frac{1}{\rho}\nu_{3,u_\rho})}{|(\nu_{\alpha,u_\rho}, \frac{1}{\rho}\nu_{3,u_\rho})|}\bigg) \bigg|(\nu_{\alpha,u_\rho}, \frac{1}{\rho}\nu_{3,u_\rho})\bigg| \, d\mathcal{H}^2 \\
			&\ge C'\liminf_{\rho \to 0}\int_{J_{u_\rho}} \bigg|(\nu_{\alpha,u_\rho}, \frac{1}{\widetilde{\rho}}\nu_{3,u_\rho})\bigg| \, d\mathcal{H}^2 \\
			& \ge  C'\liminf_{\rho \to 0}\int_{J_{u_\rho}}\frac{1}{\widetilde{\rho}} |\nu_{3,u_\rho} |\, d\mathcal{H}^2 \ge \frac{1}{\widetilde{\rho}}\int_{J_{u}}|\nu_{3,u} |\, d\mathcal{H}^2,
		\end{aligned}
		\]
		where we have used the lower semicontinuity of 
		$$
		u \to \int_{J_u}|\nu_{3,u}|\,d\mathcal{H}^2
		$$
		with respect to the weak $GSBV^p$ convergence by Theorem \ref{thm:closure}.\\
		Since $\widetilde{\rho}>0$ is arbitrary, it follows that $\nu_{3,u}=0$ $\mathcal{H}^2$-a.e. on $J_u$, hence $u$ does not depend on $x_3$ and $u \in \mathcal{A}$.
		
		Then, repeating part of the previous argument and considering $\limsup$ instead of $\liminf$ (since we assumed that the limit of the family $\rho^{-1}\mathcal{G}(u_{\rho})$ exists) we obtain
		\[
		\begin{aligned}
			C&\ge \lim_{\rho \to 0}\rho^{-1}\mathcal{G}_{\rho}(u_\rho) \ge\limsup_{\rho \to 0}\int_{J_{u_\rho}} \psi\bigg([u_\rho],\frac{(\nu_{\alpha,u_\rho}, \frac{1}{\rho}\nu_{3,u_\rho})}{|(\nu_{\alpha,u_\rho}, \frac{1}{\rho}\nu_{3,u_\rho})|}\bigg) \bigg|(\nu_{\alpha,u_\rho}, \frac{1}{\rho}\nu_{3,u_\rho})\bigg| \, d\mathcal{H}^2 \\
			&\ge C_4\limsup_{\rho \to 0}\int_{J_{u_\rho}} \phi([u_\rho])\bigg|(\nu_{\alpha,u_\rho}, \frac{1}{\rho}\nu_{3,u_\rho})\bigg| \, d\mathcal{H}^2 \\
			&\ge C_4\limsup_{\rho \to 0}\int_{J_{u_\rho}} \phi([u_\rho]) \bigg|\frac{1}{\rho}\nu_{3,u_\rho}\bigg| \, d\mathcal{H}^2 \\
			& \ge C_4\frac{1}{\widetilde{\rho}}\limsup_{\rho \to 0}\int_{J_{u_\rho}} \phi([u_\rho]) |\nu_{3,u_\rho}| \, d\mathcal{H}^2.
		\end{aligned}
		\]
		Since $\widetilde{\rho}$ is arbitrary, we deduce that
		\[
		\lim_{\rho \to 0}\int_{J_{u_\rho}} \phi([u_\rho]) |\nu_{3,u_\rho}| \, d\mathcal{H}^2=0.
		\]
		
		Finally, since $W \ge W_0\ge \mathcal{Q}W_0$, $\psi \ge \psi_0\ge \mathcal{B}\psi_0$, and using also that
		\[
		\nu_{u_\rho(\cdot,t)}(x_\alpha)=\frac{\nu_{\alpha,u_\rho}(x_\alpha,t)}{|\nu_{\alpha,u_\rho}(x_\alpha,t)|} \qquad \mathcal{H}^1\text{-a.e. } x_\alpha \in J_{u_\rho(\cdot,t)},
		\]
		we obtain
		\[
		\begin{aligned}
			\lim_{\rho \to 0}\rho^{-1}\mathcal{G}_{\rho}(u_\rho)&\ge\liminf_{\rho \to 0}\bigg(\int_{-1/2}^{1/2}\bigg(\int_{\Sigma} \mathcal{Q}W_0(\nabla_\alpha u_\rho)\,dx_\alpha \\
			&+ \int_{J_{u_\rho(\cdot,t)} \cap \Sigma} \mathcal{B}\psi_0([u_\rho], \nu_{u_\rho(\cdot,t)}(x_\alpha))|\nu_{\alpha,u_\rho}(x_\alpha,t)|\, d\mathcal{H}^1(x_\alpha) \bigg)\,dt\bigg) \\
			&\ge \liminf_{\rho \to 0}\bigg( \int_{-1/2}^{1/2}\bigg(\int_{\Sigma} \mathcal{Q}W_0(\nabla_\alpha u_\rho)\,dx_\alpha \\
			&+ \int_{J_{u_\rho(\cdot,t)}\cap \Sigma} \mathcal{B}\psi_0([u_\rho], \nu_{u_\rho(\cdot,t)}(x_\alpha))\, d\mathcal{H}^1(x_\alpha) \bigg)\,dt\bigg) \\
			&- \limsup_{\rho \to 0}\int_{-1/2}^{1/2}\int_{J_{u_\rho(\cdot,t)}} \mathcal{B}\psi_0([u_\rho], \nu_{u_\rho(\cdot,t)}(x_\alpha))|\nu_{3,u_\rho}|\, d\mathcal{H}^1(x_\alpha) \,dt,
		\end{aligned}
		\]
		where we used the inequality $1\le |\nu_\alpha|+|\nu_3|$. Since $0\le \mathcal{B}\psi_0\le \psi \le C_2\phi$, the last term vanishes by the previous computation, and thus
		\[
		\begin{aligned}
			&\liminf_{\rho \to 0}\bigg( \int_{-1/2}^{1/2}\bigg(\int_{\Sigma} \mathcal{Q}W_0(\nabla_\alpha u_\rho)\,dx_\alpha \\
			&+ \int_{J_{u_\rho(\cdot,t)}\cap \Sigma} \mathcal{B}\psi_0([u_\rho], \nu_{u_\rho(\cdot,t)}(x_\alpha))\, d\mathcal{H}^1(x_\alpha)\bigg) \,dt\bigg)\\
			&\ge  \int_{-1/2}^{1/2}\bigg(\liminf_{\rho \to 0}\int_{\Sigma}\mathcal{Q}W_0(\nabla_\alpha u_\rho)\,dx_\alpha  \\
			&+ \int_{J_{u_\rho(\cdot,t)}\cap \Sigma} \mathcal{B}\psi_0([u_\rho], \nu_{u_\rho(\cdot,t)}(x_\alpha))\, d\mathcal{H}^1(x_\alpha)\bigg)\,dt \ge \mathcal{G}_0(u),
		\end{aligned}
		\]
		where we used Fatou's lemma, Theorem \ref{teor69}, and the fact that convergence in measure in $\Sigma_1$ implies, up to subsequences, convergence in measure on the sections $\Sigma \times \{t\}$ for a.e. $t$. Since $u\in\mathcal{A}$, the claim follows.
	\end{proof}

	\subsection{Upper bound}
	We begin by estimating the $\Gamma$-$\limsup$ of the functionals $\rho^{-1}\mathcal{G}_\rho$ in terms of an auxiliary functional defined on functions in $GSBV^{p}(\Sigma_1;\mathbb{R}^3)$ that are independent of the third variable. To this end, we need three preliminary results. The first, \cite[Lemma~4.5]{almi2023brittlemembranes}, provides an approximation of $\mathrm{Aff}^*$ maps on Lipschitz domains by smooth functions with positive determinant. The second, \cite[Lemma~4.10]{almi2023brittlemembranes}, concerns the construction of functions optimizing the third column of the bulk energy density $W$. The third, Lemma~\ref{lemmalinearalgebra}, deals with the construction of diffeomorphisms that optimize the third component of the surface energy density $\psi$ along the jump set.

	\begin{lemma}\label{lemmasucc} 
		Let $U$ be a bounded open subset of $\mathbb{R}^2$ with Lipschitz boundary, and let 
		$v \in \mathrm{Aff}^*(U;\mathbb{R}^3)$. Then there exists a sequence 
		$v_j \in C^1(\overline{U};\mathbb{R}^3)$ such that:
		\begin{itemize}
			\item $v_j \to v$ in $W^{1,p}(U;\R^3)$ and 
			$\| \nabla v_j \|_{L^\infty(U;\mathbb{M}^{3 \times 2})} \le 2 \| \nabla v \|_{L^\infty(U;\mathbb{M}^{3 \times 2})}$;
			\item there exists $\theta=\theta(v) > 0$ such that for every $j \ge 1$,
			\[
			|\partial_1 v_j(x) \wedge \partial_2 v_j(x)| \ge \theta 
			\qquad \text{for all } x \in \overline{U}.
			\]
		\end{itemize}
	\end{lemma}
	
	\begin{lemma}\label{lemmacontrollodet} 
		Let $U \subseteq \mathbb{R}^2$ be a bounded open set, let $\eta>0$, and let 
		$G \in C(\overline{U};\mathbb{R}^{3\times2} )$ be such that 
		\[
		|G^1(x)\wedge G^2(x)| \ge \eta 
		\qquad \text{for every } x \in \overline{U}.
		\]
		Let $\varepsilon \in (0,1/2)$, let $T$ be a finite triangulation of $U$, and let 
		$\Psi : U \to \mathbb{M}^{2\times 2}$ be a piecewise constant function on the 
		triangulation such that for every $x \in U$
		\[
		| \det(\Psi(x)) - 1 | \le \varepsilon.
		\]
		Then there exist a function $h \in C^\infty(\overline{U};\mathbb{R}^3)$ and a constant 
		$\beta = \beta(\eta,\|G\|_{L^\infty},\|\Psi\|_{L^\infty}) > 0$ such that 
		$\|h\|_{L^\infty(U;\R^3)} \le \beta$ and
		\[
		\begin{aligned}
			& \det\big(G(x)\Psi(x)\mid h(x)\big) \ge \frac{1}{\beta} \qquad \mbox{for every $x \in \overline{U}$}, \\
			& \int_{U} W\big(G(x)\Psi(x)\mid h(x)\big)\,dx 
			\le \varepsilon + \int_{U} W_0\big(G(x)\Psi(x)\big)\,dx.
		\end{aligned}
		\]
	\end{lemma}

    As stated beforehand, we will now construct a sequence of local $C^{\infty}$ diffeomorphisms converging to the identity in $W^{1,\infty}$ that optimize the third component of the surface energy density $\psi$.
    In the recovery sequence, the inverse diffeomorphisms will have to be defined on all of $\Sigma_1$. Since the optimizing diffeomorphisms perform a small "vertical" shift, points in $\Sigma_1$ may be mapped back outside $\Sigma_1$. We
    therefore construct them on the larger set $\Sigma_2 := \Sigma \times (-1,1)$.

    We clarify the role of the variables $\kappa$ and $\zeta$ in Lemma \ref{lemmalinearalgebra}. In the definition of $\psi_0$, $\zeta$ is an auxiliary scalar variable over which the
    surface density $\psi$ is minimized. Thus, for a fixed jump amplitude and a
    fixed normal $\nu$, $\zeta$ represents the additional degree of freedom used
    to optimize the third argument of $\psi$. In the recovery construction, we
    choose $\zeta$ as an almost minimizer (see \eqref{eqk_ibeta} below), and then construct a local diffeomorphism whose effect is precisely to realize this choice of the third variable. Accordingly, in the application of
    Lemma~\ref{lemmalinearalgebra}, the vector $\kappa$ is taken to be the normal
    $\nu_u \in \mathbb{S}^1$, while $\zeta$ is the prescribed almost-optimal value for $\psi$.

    \begin{lemma}\label{lemmalinearalgebra} 
		Let $x_0 \in \mathbb R^2$. Let $U,\, V,\, \Sigma$ be bounded open subsets of $\mathbb{R}^2$ such that
		\[
		x_0 \in U \Subset V \Subset  \Sigma .
		\]
		
		\noindent Let $\kappa \in \mathbb{S}^1$ (which we can also regard as a vector in $\mathbb{S}^2$ by setting its third component equal to zero) and let $\zeta \in \mathbb{R}$. Then, for every $\rho>0$ we can construct a $C^\infty$ map
		\[
		f_\rho : \mathbb R^3 \to \mathbb{R}^3 ,
		\]
		depending on $\rho,\kappa,\zeta,\,U,\,V,\, \Sigma$ and $x_0$, such that:
		
		\begin{itemize}
			\item[(a)] $f_\rho|_{\overline{U}_2}$ is an isometry, and its associated linear map $O_\rho$ satisfies
			\[
			O_\rho(\kappa,0)=\left( \frac{\kappa}{\sqrt{1+\rho^2 \zeta^2}},\, \frac{\rho \zeta}{\sqrt{1+\rho^2 \zeta^2}} \right);
			\]
			
			\item[(b)] $f_\rho \equiv \mathrm{Id}$ on $\Sigma_2 \setminus V_2$;
			
			\item[(c)] $\|f_\rho - \mathrm{Id}\|_{W^{k,\infty}(\Sigma_2;\R^3)} \to 0$ as $\rho \to 0$ for every $k \ge 1$;
			
			\item[(d)] for any compact set $K \Subset U$, for $\rho>0$ sufficiently small (depending on $K$), we have
			\[
			f_\rho\big( K \times (-1/2,1/2) \big) \subseteq U_2.
			\]

			\item[(e)] for $i=1,2$, one has
			\[
			\lim_{\rho\to 0}\bigg\Vert \frac{1}{\rho}\big(\nabla f_\rho\big)_{i}^3 \bigg\Vert_{L^\infty(\Sigma_2)} \le |\zeta|;
			\]
			
			\item[(f)] let $\Pi_{x_0,\kappa}$ be the plane passing through $x_0$ with normal $(\kappa,0)$. Then, $f_\rho(\Pi_{x_0,\kappa})$ is a $C^\infty$ surface such that for every $x \in \Pi_{x_0,\kappa}$ its normal $\nu$ satisfies
			\[
			\lim_{\rho\to 0}\, \bigg\Vert\frac{\nu_3\big( f_\rho \big)}{\rho} \bigg\Vert_{L^\infty(\Sigma_2)}
			\le 2|\zeta|.
			\]
		\end{itemize}
	\end{lemma}
	\begin{proof}
        It is not restrictive to consider $x_0=0$.
		Consider a linear map $O_\rho$ such that
		\[
		\begin{aligned}
			O_\rho(\kappa_1,\kappa_2,0)
			&=\left( 
			\frac{\kappa_1}{\sqrt{1+\rho^2 \zeta^2}},\,
			\frac{\kappa_2}{\sqrt{1+\rho^2 \zeta^2}},\,
			\frac{\rho \zeta}{\sqrt{1+\rho^2 \zeta^2}}
			\right),\\[4pt]
			O_\rho(-\kappa_2,\kappa_1,0)
			&=(-\kappa_2,\kappa_1,0),\\[4pt]
			O_\rho(0,0,1)
			&=\left(
			-\frac{\rho \zeta \kappa_1}{\sqrt{1+\rho^2 \zeta^2}},\,
			-\frac{\rho \zeta \kappa_2}{\sqrt{1+\rho^2 \zeta^2}},\,
			\frac{1}{\sqrt{1+\rho^2 \zeta^2}}
			\right).
		\end{aligned}
		\]
		This map is an isometry, since it sends an orthonormal basis to an orthonormal basis.  
		The matrix associated with $O_\rho$ with respect to the canonical basis is
		\[
		O_\rho=
		\begin{pmatrix}
			1+\kappa_1^2 \lambda & \kappa_1\kappa_2 \lambda & -\frac{\rho \kappa_1 \zeta}{s}\\[4pt]
			\kappa_1\kappa_2 \lambda & 1+\kappa_2^2 \lambda & -\frac{\rho \kappa_2 \zeta}{s}\\[4pt]
			\frac{\rho \kappa_1 \zeta}{s} & \frac{\rho \kappa_2 \zeta}{s} & \frac{1}{s}
		\end{pmatrix},
		\]
		where $s=\sqrt{1+\rho^2 \zeta^2}$ and $\lambda = \frac{1}{s}-1$ (see Figure \ref{figure2}).

				
					
				
					
				
        
		\begin{figure}[h!]
	\tdplotsetmaincoords{70}{130}
	\begin{tikzpicture}[tdplot_main_coords, scale=1.2]
		
		\begin{scope}
			\filldraw[draw=black, fill=blue!15] (5,0,0)--(0,0,0)--(0,0,5)--(5,0,5)--cycle;
			\draw[thick,->] (2.5,0,2.5)--(2.5,2,2.5);
			\node at (2.5,1.2,2.15) {$\kappa$};
		\end{scope}
		
		\begin{scope}[xshift=7.3cm]
			\draw[dashed, gray] (5,0,0)--(0,0,0)--(0,0,2.5);
			\draw[dashed, gray] (5,0,0)--(5,0,2.5);
			\filldraw[draw=black, fill=green!20] (5,0.4,0)--(0,0.4,0)--(0,-0.4,5)--(5,-0.4,5)--cycle;
			\draw[dashed, gray] (0,0,2.5)--(0,0,5)--(5,0,5)--(5,0,2.5);
			\draw[dashed, gray, ->] (2.5,0,2.5)--(2.5,2,2.5);
			\node[gray] at (2.5,1.2,2.15) {$\kappa$};
			\draw[thick, ->] (2.5,0,2.5)--(2.5,1.95,2.75);
			\node at (2.5,1.2,2.95) {$O_\rho(\kappa)$};
		\end{scope}
		
	\end{tikzpicture}
	\caption{The plane $\Pi_{x_0,\kappa}$ and its normal vector before and after the transformation through $O_\rho$.}
	\label{figure2}
\end{figure}
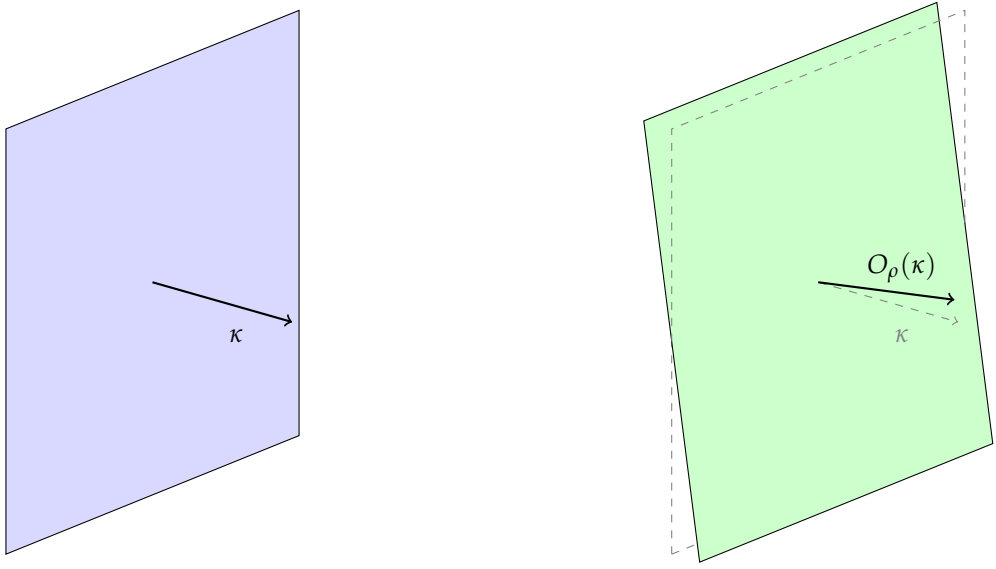
		
		Let $\varphi$ be a test function in $C^\infty$ such that $\varphi|_{\overline{U}} \equiv 1$, $\varphi \equiv 0$ outside of $V$, and $0 \le \varphi \le 1$.
		The function $\varphi$ can be extended to all of $\mathbb R^3$ by defining 
		\[
		\varphi(x) = \varphi(x_\alpha).
		\]
		
		We now define
		\[
		f_\rho(x)=\varphi(x)\,{O}_\rho(x) + (1-\varphi(x))\,x.
		\]
		This map clearly satisfies (a) and (b).  
		For (c), we notice that
		\[
		|O_\rho - I|\to 0 \qquad \mbox{as $\rho \to 0$},
		\]
		and that
		\[
		\begin{aligned}
			\|\nabla f_\rho - I\|_{L^\infty(\Sigma_2;\mathbb{M}^{3 \times 3})}
			&= \big\| (O_\rho - I)x \otimes \nabla \varphi 
			+ \varphi (O_\rho - I) \big\|_{L^\infty(\Sigma_2;\mathbb{M}^{3 \times 3})}\\[4pt]
			&\le 
			\|(O_\rho - I)x \otimes \nabla\varphi\|_{L^\infty(\Sigma_2;\mathbb{M}^{3 \times 3})}
			+ \|\varphi (O_\rho - I)\|_{L^\infty(\Sigma_2;\mathbb{M}^{3 \times 3})} \\[4pt]
			&\le 
			|O_\rho - I| \left( \|\nabla \varphi\|_{L^\infty(\Sigma_2;\mathbb{R}^{3})} \, \mathrm{diam}(\Sigma_2)  + 1 \right)
			\to 0,
		\end{aligned}
		\]
		as $\rho \to 0$, since $\Sigma_2$ is a bounded set in $\R^3$. For $k \geq 2$ the argument is analogous. 
        
        We now show (d). Call with $\partial_l U_2$ the lateral boundary of the cylinder $U_2$. Since $K \Subset U$, there exists $\eta>0$ so that
		$$
		\text{dist}(\partial_l U_2, \, K \times (-1/2,1/2) )=\text{dist}(\partial U , \,K)=\eta.
		$$
		Being $f_\rho$ a $C^\infty$ diffeomorphism converging to the identity by (c), we have that for $\rho$ sufficiently small it holds 
		$$
		\text{dist}(\partial_l U_2, f_\rho (K \times (-1/2,1/2) ))\ge \eta/2 >0.
		$$
		Furthermore, since $f_\rho \to \mathrm{Id}$ in $W^{1,\infty}(\Sigma_2;\R^3)$, for $\rho$ sufficiently small it follows that for every $x \in K \times (-1/2,1/2)$ we have $|(f_\rho(x))_3| \le 1$. Consequently,
        \[
        f_\rho\bigl(K \times (-1/2,1/2)\bigr) \subset U_2 .
        \]
		We are left to demonstrate (e) and (f). For (e), we notice that by definition of $f_\rho$ and the fact that $\partial \varphi(x)/\partial x_3=0$, for $i=1,2$
		$$
		\big(\nabla f_\rho\big)_{i}^3=-\varphi\frac{\rho\kappa_i \zeta}{s}.
		$$
		Therefore, for $i=1,2$
		$$
		\lim_{\rho \to 0}\bigg\Vert	\frac{1}{\rho}\big(\nabla f_\rho\big)_{i}^3\bigg\Vert_{L^\infty(\Sigma_2)}= \Vert\varphi\Vert_{L^\infty(\Sigma_2)} |\kappa_i \zeta|\le |\zeta|.
		$$
		Finally, since $x_0 = 0$, the plane $\Pi_{0,\kappa}$ can be parameterized as
		$$
		r(s,t) = (-s\kappa_2, \, s\kappa_1, \, t), \qquad s, t \in \mathbb{R}.
		$$
		The unit normal $\nu$ to the image surface $f_\rho(\Pi_{0,\kappa})$ at the point $f_\rho(r(s,t))$ is given by
		$$
		\frac{r_s \wedge r_t}{|r_s \wedge r_t|},
		$$
		where $r_s = \frac{\partial}{\partial s} f_\rho(r(s,t))$ and $r_t = \frac{\partial}{\partial t} f_\rho(r(s,t))$. By the chain rule,
		$$
		(r_s)_i = -\kappa_2 (\nabla f_\rho)_i^1 + \kappa_1 (\nabla f_\rho)_i^2, \qquad (r_t)_i = (\nabla f_\rho)_i^3,
		$$
		for $i = 1, 2, 3$. The third component of $r_s \wedge r_t$ is therefore
		$$
		\begin{aligned}
			(r_s \wedge r_t)_3 
			&= (r_s)_1\, (r_t)_2 - (r_s)_2\, (r_t)_1 \\
			&= \bigl(-\kappa_2 (\nabla f_\rho)_1^1 + \kappa_1 (\nabla f_\rho)_1^2\bigr)\, (\nabla f_\rho)_2^3 
			- \bigl(-\kappa_2 (\nabla f_\rho)_2^1 + \kappa_1 (\nabla f_\rho)_2^2\bigr)\, (\nabla f_\rho)_1^3.
		\end{aligned}
		$$
		By (c), as $\rho \to 0$, $(\nabla f_\rho)_i^\beta \to \delta_i^\beta$ uniformly, so that
		$$
		(r_s)_1 \to -\kappa_2, \qquad (r_s)_2 \to \kappa_1, \qquad |r_s \wedge r_t| \to 1.
		$$
		Since $\nu_3(f_\rho(r(s,t))) = (r_s \wedge r_t)_3 / |r_s \wedge r_t|$, dividing by $\rho$ and using (e) we obtain
		$$
		\begin{aligned}
			\lim_{\rho \to 0} \left\Vert\frac{\nu_3(f_\rho(r(s,t)))}{\rho}\right\Vert_{L^\infty(\Sigma_2)}
			&= \lim_{\rho \to 0} \left\Vert\frac{1}{|r_s \wedge r_t|} 
			\bigg((r_s)_1 \,\frac{(\nabla f_\rho)_2^3}{\rho} 
			- (r_s)_2 \,\frac{(\nabla f_\rho)_1^3}{\rho}\bigg)\right\Vert_{L^\infty(\Sigma_2)} \\
			&\le |\kappa_2| \lim_{\rho \to 0} 
			\left\Vert\frac{(\nabla f_\rho)_2^3}{\rho}\right\Vert_{L^\infty(\Sigma_2)}
			+ |\kappa_1| \lim_{\rho \to 0} 
			\left\Vert\frac{(\nabla f_\rho)_1^3}{\rho}\right\Vert_{L^\infty(\Sigma_2)} \\
			&\le |\kappa_2| \cdot |\zeta| + |\kappa_1| \cdot |\zeta| \le 2|\zeta|.
		\end{aligned}
		$$
		This concludes the proof.
	\end{proof}

	\begin{oss}\label{ossf_rho}
		From (b) and (c) of Lemma \ref{lemmalinearalgebra}, we infer that for $\rho$ sufficiently small the map $f_\rho$ is a $C^\infty$ diffeomorphism from $\Sigma_2$ to its image, and moreover
		\[
		\nabla f^{-1}_\rho\bigl(f_\rho(x)\bigr) = \bigl[\nabla f_\rho(x)\bigr]^{-1}.
		\]
		From this identity, one can show that properties (c),(e) and (f) of Lemma \ref{lemmalinearalgebra} also hold for $f^{-1}_\rho$ on the image of $\Sigma_2$ through $f_\rho$.
	\end{oss}

    We are finally ready to prove the $\Gamma$-$\limsup$ in terms of the auxiliary functional.

	\begin{teo}\label{teorgamma_limsup}
	Let $u \in GSBV^p(\Sigma_1;\mathbb R^3)$. Then
	\[
	\Gamma\text{-}\limsup_{\rho \to 0} \, \rho^{-1}\mathcal G_\rho(u)
	\le \overline{\mathcal G^w_0}(u),
	\]
	where $\overline{\mathcal G^w_0}(u)$ denotes the relaxation of $\mathcal G^w_0(u)$ and $\mathcal G^w_0(u)$ has been defined in \eqref{eqG_0^w}.
	\end{teo}
	\begin{proof} 
    We claim that 
    \[
    \Gamma\text{-}\limsup_{\rho \to 0} \, \rho^{-1}\mathcal G_\rho(u)
	\le \mathcal G^w_0(u), \qquad\text{for every }u\in GSBV^p(\Sigma_1;\R^m).
    \]
    We restrict to the case $u \in Y$. 
    Indeed, if $u\in GSBV^p(\Sigma_1;\mathbb R^3)\setminus Y$, then the inequality above is trivially satisfied by definition of $\mathcal G^w_0$.
Since $u \in Y$, there exists $\eta>0$ such that 
\[
|A^1 \wedge A^2| \ge \eta 
\quad\text{for every } A \in \partial u(x) \text{ and every } x \in \Sigma.
\]
We want to verify that for every sufficiently small $\varepsilon>0$ there exists a sequence 
\[
\{u_{j,\rho}\}_{j,\rho} \subseteq SBV(\Sigma_1;\mathbb R^3)
\]
such that $u_{j,\rho} \to u$ in $L^p(\Sigma_1;\mathbb R^3)$ as $\rho \to 0$ and $j \to +\infty$, and
\begin{equation}\label{eq1}
	\limsup_{j \to +\infty,\, \rho \to 0}
	\left(
	\int_{\Sigma_1} 
	W\!\left(\nabla_\alpha u_{j,\rho}\,\Big|\,
	\frac{1}{\rho}\partial_3 u_{j,\rho}\right)\, dx
	\;+\;
	\int_{J_{u_{j,\rho}}}
	\psi_\rho([u_{j,\rho}],\nu_{u_{j,\rho}})\, d\mathcal H^2
	\right)
	\le (1+\varepsilon)^2 \mathcal G_0^w(u) + C\varepsilon,
\end{equation}
where $\nabla_\alpha u=(\partial_1 u \,|\, \partial_2 u)$, and the $\limsup$ is taken in the given order, namely first $\rho\to 0$ and then $j\to+\infty$.
Let us divide the proof into three steps.\\
\textit{Step 1: Construction of the recovery sequence}. Since $u$ is independent of $x_3$, we may regard $u$ as an element of $GSBV^p(\Sigma;\mathbb{R}^3)$, identifying it with its trace on $\Sigma$. Moreover, since $u \in Y$, the jump set $J_u \Subset \Sigma$ consists of finitely many disjoint connected components, that is,
\[
J_u = \bigcup_{l=1}^m \gamma_l,
\]
each of which is either a single segment or the union of two segments sharing one endpoint.
Let $\delta_0>0$ be such that $(\gamma_l)_{\delta_0}\subseteq\Sigma$ and  
$(\gamma_l)_{\delta_0} \cap (\gamma_s)_{\delta_0} = \emptyset$ for every $s\neq l$.  
For every $0<\delta<\delta_0$, applying Lemma~\ref{lemmadiffeolisciokernelvariabile} to each connected component $\gamma_l$ and combining the resulting maps (which act on disjoint neighborhoods), we obtain a diffeomorphism
\[
\Phi_\delta : \Sigma \setminus J_u \to \Omega^\delta,
\]
where $\Omega^\delta := \Sigma \setminus \bigcup_{l=1}^m \Delta_{\delta,l}$ is a Lipschitz domain.
Fix $\varepsilon>0$. Up taking $\delta_0$ smaller, we may assume that for every $\delta\in(0,\delta_0)$
\begin{equation}\label{eq3}
	\|D\Phi_\delta\|_{L^\infty} + \|D\Phi_\delta^{-1}\|_{L^\infty} \le 3,
	\qquad
	\|\det(D\Phi_\delta)-1\|_{L^\infty} \le \varepsilon,
\end{equation}
\[
|\det(M)-1|\le \varepsilon
\quad\text{for every } M\in\partial\Phi_\delta^{-1}
\text{ and every } x\in\Omega^\delta.
\]
With a slight abuse of notation, for every $(x_1,x_2)\in\Omega^\delta$ we also set
\begin{equation}\label{eq4}
	\Phi_\delta(x_1,x_2,x_3) := (\Phi_\delta(x_1,x_2),\, x_3).
\end{equation}
Now fix $\delta \in (0,\delta_0)$ and define the function 
\[
v := u \circ \Phi_\delta^{-1}
\quad\text{in } \Omega^\delta.
\]
Using \eqref{eq3} and reasoning as in the proof of Theorem~\ref{teorapproxaff^*}, we obtain that
\[
v \in \mathrm{Aff}^*(\Omega^\delta;\mathbb R^3)
\quad\text{and}\quad
\|\nabla v\|_{L^\infty(\Omega^\delta;\mathbb{M}^{3 \times 2})} \le 3\|\nabla u\|_{L^\infty(\Sigma;\mathbb{M}^{3 \times 2})}.
\]

We now apply Lemma~\ref{lemmasucc} to the function $v$ on $\Omega^\delta$.  
We find a sequence $v_j \in C^1(\overline{\Omega^\delta};\mathbb R^3)$ satisfying the properties of Lemma \ref{lemmasucc} for some constant $\theta=\theta(v)>0$ independent of $j$.

Apply Lemma~\ref{lemmacontrollodet} with 
\[
G = \nabla v_j, \qquad 
\Psi = D\Phi_\delta(\Phi_\delta^{-1}), \qquad 
U = \Omega^\delta
\]
for each $j\ge 1$.  
We obtain $h_j \in C^\infty(\Omega^\delta;\mathbb R^3)\cap W^{1,\infty}(\Omega^\delta;\mathbb R^3)$ and 
\[
\beta=\beta\big(\theta,\,\|\nabla u\|_{L^\infty(\Sigma;\mathbb{M}^{3 \times 2})}\big)
\]
such that for every $j\ge 1$,
\[
\|h_j\|_{L^\infty(\Omega^\delta;\R^3)} \le \beta
\]
and
\begin{equation}\label{eq2}
	\det\big(\nabla v_j\, D\Phi_\delta(\Phi_\delta^{-1}) \,\big|\, h_j\big)
	\ge \frac{1}{\beta}
	\qquad\text{a.e. in } \Omega^\delta,
\end{equation}
\[
\int_{\Omega^\delta}
W\big(\nabla v_jD\Phi_\delta(\Phi_\delta^{-1})\,\big|\,h_j(x)\big)\, dx
\;\le\;
\int_{\Omega^\delta}
W_0\big(\nabla v_j(x)D\Phi_\delta(\Phi_\delta^{-1})\big)\, dx
+\varepsilon.
\]
Since $\delta$ and $j$ are independent of each other, for simplicity of notation we set
\[
\widehat{v}_j := v_j \circ \Phi_\delta,
\qquad
\widehat{h}_j := h_j \circ \Phi_\delta,
\]
suppressing the dependence on $\delta$.

We now focus on the surface part. Since $\mathcal L^2(J_u)=0$, there exists an open set $E$ such that 
\[
J_u \Subset E \subseteq \Sigma
\qquad \text{and} \qquad 
\mathcal L^2(E) < \varepsilon.
\] 
We first consider a component $\gamma_l$ of $J_u$ which consists of a single segment. Since the normal $\nu_u$ is constant along $\gamma_l$, since $\psi_0(\cdot,\nu_u)$ is locally bounded by Lemma \ref{lemma4} and continuous by Remark~\ref{osspsi_0}, and since the traces $\widehat{v}_j^\pm$ are equicontinuous on $\gamma_l$ by Lemma~\ref{lemmasucc}, there exists $\overline{n} \in \mathbb{N}$ (independent of $j$) such that for every $n \geq \overline{n}$ we can partition $\gamma_l$ into $n$ segments $\alpha_i$ of equal length, with midpoints $x_i$, satisfying
\begin{equation}\label{eqalpha_i}
	\int_{\gamma_l}\psi_0([\widehat{v}_j],\nu_u)\, d\mathcal H^1
	\ge 
	\sum_{i=1}^{n}
	\mathcal H^1(\alpha_i)\,
	\psi_0([\widehat{v}_j](x_i),\nu_u)
	-\varepsilon
\end{equation}
for every $j \in \mathbb{N}$.

Starting from these $\alpha_i$, we consider a segment $\alpha_i' \subseteq \alpha_i$ such that  
$\mathcal H^1(\alpha_i\setminus\alpha_i') \le \varepsilon/n$, the $\alpha_i'$ are mutually well-separated and $x_i$ is also the middle point of $\alpha_i'$.  
Then the segments $\alpha_i'$ constructed in this way satisfy
\begin{equation}\label{eq21}
\mathcal H^1\!\left(\gamma_l \setminus \bigcup_{i=1}^n \alpha_i'\right) < \varepsilon.
\end{equation}

For each $\alpha_i'$, we can find two open sets $U^i$ and $V^i$ in $\Sigma$ with the property that 
\[
\alpha_i' \subseteq U^i \Subset V^i \Subset E
\]
and the sets $\overline{V^i}$ are mutually disjoint (see Figure \ref{figure5}).

We now consider values $\zeta_{i,j}$, varying with $i=1,\dots,n$ and $j \in \mathbb N$, such that by the definition of $\psi_0$,
\begin{equation}\label{eqk_ibeta}
\varepsilon +\psi_0([\widehat{v}_j](x_i),\nu_u)
\;\ge\;
\psi([\widehat{v}_j](x_i),\nu_u,\zeta_{i,j}).
\end{equation}
The $\zeta_{i,j}$ are uniformly bounded in $i$ and $j$, since by $(B_3)$ and Lemma \ref{lemmasucc} we have
\begin{equation}\label{eqk_i}
|\zeta_{i,j}|
\;\le\;
C \psi([\widehat{v}_j](x_i),\nu_u,\zeta_{i,j})
\;\le\;
C \psi([\widehat{v}_j](x_i),\nu_u,0)
\;\le\;
C \bigl(1 + 2\|v_j\|_\infty\bigr)
\;\le\;
K.
\end{equation}
We are therefore in the position to construct functions $f_{i,l,j,\rho} \colon \Sigma_2 \to \R^3$ depending on 
$x_i,\, \nu_u,\, \zeta_{i,j},\, U^i$ and $V^i$, satisfying properties (a)--(f) of 
Lemma~\ref{lemmalinearalgebra}. Using again the equicontinuity of the functions $\widehat{v}_j^\pm$, we may take the same 
sets $U^i$ and $V^i$ for every $j\in\mathbb N$. 
Moreover, since the sets $\overline{V_i}$ are mutually disjoint, we may glue together the 
functions $f_{i,l,j,\rho}$ and obtain a single function $f_{l,j,\rho} \colon \Sigma_2 \to \R^3$ which coincides with each 
$f_{i,l,j,\rho}$ in the sets $V^i \times (-1,1)$ and it is the identity outside them.

Since the traces $\widehat{v}_j^\pm$ are uniformly equicontinuous on $\gamma_l$ and $\psi$ satisfies $(B_1)$, there exists $\overline{\tau} > 0$ such that for any $x, y \in \gamma_l$ with $|x - y| < \overline{\tau}$,
\begin{equation}\label{eqestimatepsi2}
	\big|\psi([\widehat{v}_j](x), \nu) - \psi([\widehat{v}_j](y), \nu)\big| < \varepsilon
\end{equation}
uniformly in $j \in \mathbb{N}$ and in $\nu \in \mathbb{R}^3 \setminus \{0\}$ with $\nu_\alpha \in \mathbb{S}^1$ and $|\nu_3| \le K$ (where $K$ is the constant from \eqref{eqk_i}).
We then choose $n_l \ge \max\{\overline{n},\, |\gamma_l|/\overline{\tau}\}$, so that both \eqref{eqestimatepsi2} and \eqref{eqalpha_i} are satisfied.

In the case where $\gamma_l$ consists of two segments meeting at a point, we discard a small neighborhood of the corner, introducing an error of order $C\varepsilon$, and argue on each segment separately. Therefore, up to relabeling, we may assume that each $\gamma_l$ is a single segment.

Taking $n \ge \max\{n_l \mid l = 1,\dots,m\}$, equations \eqref{eqestimatepsi2} and \eqref{eqalpha_i} are satisfied for every $l$. Finally, since the maps $f_{l,j,\rho}$ coincide with the identity outside disjoint neighborhoods of the respective $\gamma_l$, we can combine them into a single map $f_{j,\rho} : \Sigma_2 \to \mathbb{R}^3$, satisfying $f_{j,\rho} \to \mathrm{Id}$ in $W^{1,\infty}(\Sigma_2;\mathbb{R}^3)$ as $\rho \to 0$.

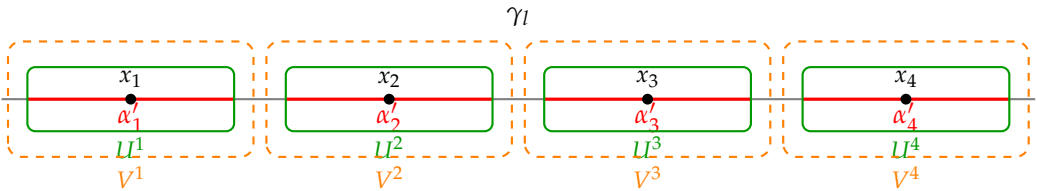
\begin{figure}[h!]
    \begin{tikzpicture}[scale=1.7]
\begin{scope}
\draw[thick, gray] (0,0) -- (8,0);
\node[above] at (4,0.5) {$\gamma_l$};

\foreach \i in {1,2,3,4} {
    \pgfmathsetmacro{\xl}{2*(\i-1)+0.3}
    \pgfmathsetmacro{\xr}{2*(\i-1)+1.7}
    \pgfmathsetmacro{\xm}{2*(\i-1)+1}
    
    \draw[red, very thick] (\xl-0.1,0) -- (\xr+0.1,0);
    \node[below, red] at (\xm,0.05) {$\alpha'_{\pgfmathprintnumber{\i}}$};
    \node[above] at (\xm,0.02) {\small $x_{\pgfmathprintnumber{\i}}$};
    \fill[black] (\xm,0) circle (1.2pt);
    
    \draw[green!60!black, thick, rounded corners=3pt] 
        (\xl-0.1, -0.25) rectangle (\xr+0.1, 0.25);
    \node[green!60!black, below] at (\xm,-0.22) {\small $U^{\pgfmathprintnumber{\i}}$};
    
    \draw[orange, thick, dashed, rounded corners=5pt] 
        (\xl-0.25, -0.45) rectangle (\xr+0.25, 0.45);
    \node[orange, below] at (\xm,-0.45) {\small $V^{\pgfmathprintnumber{\i}}$};
}
\end{scope}

\end{tikzpicture}
\caption{Construction of the segments $\alpha_i'$ and their neighborhoods $U^i \Subset V^i$ along a connected component $\gamma_l$ of the jump set.}
\label{figure5}
\end{figure}

Next, we define on $\Sigma_2$
\[
w_{j,\rho}(x):=\widehat{v}_j(x_\alpha)+\rho\, x_3\,\widehat{h}_j(x_\alpha),
\]
and take \( g_{j,\rho}:=f^{-1}_{j,\rho} \). Notice that by Remark \ref{ossf_rho}, for \(\rho\) sufficiently small we have that $g_{j,\rho}$ is a $C^\infty$ diffeomorphism such that \( g_{j,\rho}(\Sigma_1)\subseteq \Sigma_2 \).
Thus, we can finally define our \emph{recovery sequence} on $\Sigma_1$ as
\[
u_{j,\rho}(x):=w_{j,\rho}(g_{j,\rho}(x)),
\]
which is well defined for $\rho$ small and every $j \geq 1$.

By construction \( u_{j,\rho} \in SBV^p(\Sigma_1;\mathbb R^3) \). Indeed, both \( \widehat{v}_j \) and \( \widehat{h}_j \) belong to
\( SBV^p(\Sigma;\mathbb R^3) \), and moreover
\( J_{\widehat{v}_j},\,J_{\widehat{h}_j}\subseteq J_u \), hence we have 
\[
J_{u_{j,\rho}} \subseteq g^{-1}_{j,\rho}(J_u \times (-1/2,1/2))=f_{j,\rho}(J_u \times (-1/2,1/2)).
\]

Furthermore, by a change of variables given by \( \Phi^{-1}_\delta \)
and since \( g_{j,\rho}\to \text{Id} \) in \( W^{1,\infty}(\Sigma_1;\R^3) \), we obtain by dominated convergence theorem 
\[
\limsup_{ j \to +\infty}\limsup_{\rho \to 0}
\Vert u_{j,\rho}-u\Vert_{L^p(\Sigma_1)}
= \limsup_{j\to+\infty}\Vert v_j\circ\Phi_\delta -u\Vert_{L^p(\Sigma)}
\le \lim_{j\to+\infty}3^{1/p}\Vert v_j-v\Vert_{L^p(\Omega^\delta)}=0,
\]
from which the convergence in \(L^p\) follows.

\textit{Step 2: Estimate of the surface energy.}
We now consider the surface part of $\rho^{-1}\mathcal{G}_\rho$. By property~(f) of Lemma~\ref{lemmalinearalgebra} and \eqref{eqk_i},
\[
\lim_{\rho \to 0} \Vert(\nu_{u_{j,\rho}})_3/\rho\Vert_{L^\infty(\Sigma_2)} \le 2 \sup_i |\zeta_{i,j}| \le 2K,
\]
hence, for $\rho$ sufficiently small (possibly depending on $j$),
\begin{equation}\label{eq20}
	\Vert(\nu_{u_{j,\rho}})_3/\rho\Vert_{L^\infty(\Sigma_2)} \le 4K.
\end{equation}
Since $\|u_{j,\rho}\|_{L^\infty(\Sigma_1;\R^3)}$ is uniformly bounded in $\rho$ and $(\nu_{u_{j,\rho}})_3/\rho$ is uniformly bounded by \eqref{eq20}, assumption $(B_3)$ gives
\begin{equation}\label{eq22}
	\psi_\rho([u_{j,\rho}], \nu_{u_{j,\rho}}) \le C \qquad \text{on } J_{u_{j,\rho}}.
\end{equation}
Using that $J_{u_{j,\rho}} \subseteq \bigcup_{l=1}^m f_{j,\rho}(\gamma_l \times (-1/2,1/2))$, we estimate
\[
\limsup_{\rho \to 0} \int_{J_{u_{j,\rho}}} \psi_\rho([u_{j,\rho}], \nu_{u_{j,\rho}})\, d\mathcal{H}^2
\le \sum_{l=1}^m \limsup_{\rho \to 0} \int_{f_{j,\rho}(\gamma_l \times (-1/2,1/2))} \psi_\rho([u_{j,\rho}], \nu_{u_{j,\rho}})\, d\mathcal{H}^2.
\]
Thus, it suffices to estimate each summand separately. Fix $l \in \{1,\dots,m\}$. By property~(a) of Lemma~\ref{lemmalinearalgebra}, $f_{j,\rho}$ acts as the isometry $O_{i,\rho}$ on $U^i \times (-1/2,1/2)$. Using \eqref{eq21} and property~(c) of Lemma~\ref{lemmalinearalgebra}, we infer that $\mathcal{H}^2(f_{j,\rho}((\gamma_l \setminus \bigcup_{i=1}^n \alpha_i') \times (-1/2,1/2))) \le 2\varepsilon$ for $\rho$ small. Hence, recalling \eqref{eq22},
\[
\begin{aligned}
	&\limsup_{\rho \to 0} \int_{f_{j,\rho}(\gamma_l \times (-1/2,1/2))} \psi_\rho([u_{j,\rho}], \nu_{u_{j,\rho}})\, d\mathcal{H}^2 \\
	&\le \limsup_{\rho \to 0} \sum_{i=1}^n \int_{O_{i,\rho}(\alpha_i' \times (-1/2,1/2))} \psi_\rho([u_{j,\rho}], \nu_{u_{j,\rho}})\, d\mathcal{H}^2 + C\varepsilon \\
	&\le \sum_{i=1}^n \limsup_{\rho \to 0} \int_{\alpha_i' \times (-1/2,1/2)} \psi\!\left([w_{j,\rho}],\, \frac{\nu_u}{\sqrt{1+\rho^2 \zeta_{i,j}^2}},\, \frac{\zeta_{i,j}}{\sqrt{1+\rho^2 \zeta_{i,j}^2}}\right) d\mathcal{H}^2 + C\varepsilon,
\end{aligned}
\]
where in the last equality we performed a change of variables via the isometry $O_{i,\rho}$ and used property~(a) of Lemma~\ref{lemmalinearalgebra} to express the transformed normal. Since 
$$
[w_{j,\rho}](x_\alpha, x_3) = [\widehat{v}_j](x_\alpha) + \rho x_3 [\widehat{h}_j](x_\alpha) \to [\widehat{v}_j](x_\alpha)
$$ 
uniformly as $\rho \to 0$, by the reverse Fatou lemma, \eqref{eq22}, and the upper semicontinuity of $\psi$, we obtain
\[
\begin{aligned}
	&\sum_{i=1}^n \limsup_{\rho \to 0} \int_{\alpha_i' \times (-1/2,1/2)} \psi\!\left([w_{j,\rho}],\, \frac{\nu_u}{\sqrt{1+\rho^2 \zeta_{i,j}^2}},\, \frac{\zeta_{i,j}}{\sqrt{1+\rho^2 \zeta_{i,j}^2}}\right) d\mathcal{H}^2 + C\varepsilon \\
	&\le \sum_{i=1}^n \int_{\alpha_i'} \psi([\widehat{v}_j], \nu_u, \zeta_{i,j})\, d\mathcal{H}^1 + C\varepsilon \\
	&\le \sum_{i=1}^n \psi([\widehat{v}_j](x_i), \nu_u, \zeta_{i,j})\, \mathcal{H}^1(\alpha_i') + C\varepsilon \\
	&\le \sum_{i=1}^n \psi_0([\widehat{v}_j](x_i), \nu_u)\, \mathcal{H}^1(\alpha_i') + C\varepsilon \\
	&\le \sum_{i=1}^n \psi_0([\widehat{v}_j](x_i), \nu_u)\, \mathcal{H}^1(\alpha_i) + C\varepsilon \\
	&\le \int_{\gamma_l} \psi_0([\widehat{v}_j], \nu_u)\, d\mathcal{H}^1 + C\varepsilon,
\end{aligned}
\]
where we have used, in order: \eqref{eqestimatepsi2} to replace $[\widehat{v}_j](x)$ with $[\widehat{v}_j](x_i)$, \eqref{eqk_ibeta} to pass from $\psi$ to $\psi_0$, and \eqref{eqalpha_i}. Summing over $l = 1,\dots,m$, we obtain
\begin{equation}\label{eq5}
	\limsup_{\rho \to 0} \int_{J_{u_{j,\rho}}} \psi_\rho([u_{j,\rho}], \nu_{u_{j,\rho}})\, d\mathcal{H}^2
	\le \int_{J_u} \psi_0([\widehat{v}_j], \nu_u)\, d\mathcal{H}^1 + C\varepsilon.
\end{equation}
Finally, since $\widehat{v}_j^\pm \to u^\pm \circ \Phi_\delta^{-1} \circ \Phi_\delta = u^\pm$ uniformly on $J_u$ as $j \to +\infty$, by the upper semicontinuity of $\psi_0$, Lemma \ref{lemma4} and the reverse Fatou lemma,
\begin{equation}\label{eqlimsupj}
	\limsup_{j \to +\infty} \int_{J_u \cap \Sigma} \psi_0([\widehat{v}_j], \nu_u)\, d\mathcal{H}^1
	\le \int_{J_u \cap \Sigma} \psi_0([u], \nu_u)\, d\mathcal{H}^1.
\end{equation}

\textit{Step 3: Estimate of the bulk energy.}
We now estimate the bulk part of $\rho^{-1}\mathcal{G}_\rho(u_{j,\rho})$.   
Recall that $\nabla h_j \in L^\infty(\Omega^\delta;\mathbb R^3)$ for fixed $j$, although the bound on the norm is not uniform in $j$, that is, for every $j \geq 1$ there exists $C_j>0$ with the property that
\begin{equation}\label{eqcj}
\Vert \nabla h_j \Vert_{L^\infty(\Omega^\delta;\mathbb{M}^{3 \times 2})} \leq C_j.
\end{equation}
The dependence on $j$ won't be an issue, since we first let $\rho \to 0$.
We first observe that by \eqref{eq2}, for $\rho>0$ sufficiently small, it holds
\begin{equation*}
	\det\bigg(\nabla_\alpha v_j\big(\Phi_\delta(g_{j,\rho}(x))\big)\, D\Phi_\delta(g_{j,\rho}(x)) \,\bigg|\, h_j\big(\Phi_\delta(g_{j,\rho}(x))\big)\bigg)
	\ge \frac{1}{\beta} \qquad\text{for a.e. } x \in \Sigma_1.
\end{equation*}
By definition of $u_{j,\rho}$, on $\Sigma_1$ we have
\begin{equation*}
\det \left(\nabla_\alpha u_{j,\rho}\,\bigg|\, \frac{1}{\rho}\partial_3 u_{j,\rho}\right) 
= \det \left(\nabla_\alpha \widehat{v}_j(g_{j,\rho}) 
+ \rho\, (g_{j,\rho})_3\, \nabla_\alpha \widehat{h}_j(g_{j,\rho}) 
\,\bigg|\, \widehat{h}_j(g_{j,\rho})\right) \det \left(\nabla g_{j,\rho} \right).
\end{equation*}
Recall that, by Lemma~\ref{lemmasucc} and Lemma~\ref{lemmacontrollodet}, 
the norms $\|\nabla_\alpha v_j\|_{L^\infty(\Sigma_2)}$ and 
$\|h_j\|_{L^\infty(\Sigma_2)}$ are uniformly bounded in $j \geq 1$. 
Moreover, by Remark~\ref{ossf_rho}, $g_{j,\rho} \to \mathrm{Id}$ 
in $W^{1,\infty}(\Sigma_1;\mathbb{R}^3)$ as $\rho \to 0$. 
Combining these facts, we deduce that for every $j \geq 1$ there exists a constant $C_j>0$, possibly larger than the one appearing in~\eqref{eqcj}, such that for every $\rho > 0$ sufficiently small,
\begin{equation}\label{eqmatrixu_jrho}
\det\!\left(\nabla_\alpha u_{j,\rho}(x)\,\bigg|\,
\frac{1}{\rho}\partial_3 u_{j,\rho}(x)\right) \ge \frac{1}{2\beta} - \rho\, C_j \qquad\text{for a.e. } x \in \Sigma_1.
\end{equation}

We claim now that there exists $C_j>0$ possibly larger than the one appearing in~\eqref{eqcj}, such that for every $\rho > 0$ sufficiently small,
\begin{equation}\label{eqmatrixu_jrho2}
 \left\Vert \left(\nabla_\alpha u_{j,\rho}(x)\,\bigg|\,
\frac{1}{\rho}\partial_3 u_{j,\rho}(x)\right) \right\Vert_{L^\infty(\Sigma_1;\mathbb{M}^{3 \times 3})} \leq C+\rho C_j.
\end{equation}
To show the claim, we define
$$
A_\rho:=\bigg(\nabla_\alpha  u_{j,\rho}\,\bigg|\,
\frac{1}{\rho}\partial_3  u_{j,\rho}\bigg).
$$
Observe that for $i=1,2$
$$
\big(A_\rho)^i=\nabla u_{j,\rho}(g_{j,\rho}(x))(\nabla g_{j,\rho})^i.
$$
On the other hand, 
$$
\begin{aligned}
\big(A_\rho)^3= \frac{1}{\rho}\big(\nabla g_{j,\rho}\big)^3_{1}\big(\nabla u_{j,\rho}(g_{j,\rho}(x))\big)^1
+\frac{1}{\rho}\big(\nabla g_{j,\rho}\big)^3_{2}\big(\nabla u_{j,\rho}(g_{j,\rho}(x))\big)^2+\widehat{h}_j\big(g_{j,\rho}\big)\big(\nabla g_{j,\rho}\big)^3_{3}.
\end{aligned}
$$
Therefore, using (e) of Remark \ref{ossf_rho} (Lemma \ref{lemmalinearalgebra}) together with \eqref{eqk_i}, and reasoning as for \eqref{eqmatrixu_jrho}, we conclude the claim \eqref{eqmatrixu_jrho2}.

For the bulk contribution, we now distinguish two parts: the one in $E \times (-1/2,1/2)$ and the one in $\Sigma_1 \setminus \big( E \times (-1/2,1/2) \big)$. 
In $E \times (-1/2,1/2)$, since $\mathcal L^3\big(E \times (-1/2,1/2) \big) \le \varepsilon $, by \eqref{eqmatrixu_jrho}, \eqref{eqmatrixu_jrho2}, and $(A_4)$ we have 
$$
\limsup_{\rho \to 0} \int_{E \times (-1/2,1/2)}W\bigg(\nabla_\alpha u_{j,\rho}\bigg|\frac{1}{\rho}\partial_3u_{j,\rho}\bigg)\,dx\le C \mathcal L^3 \left(E \times (-1/2,1/2) \right)\le C\varepsilon.
$$
In $\Sigma_1 \setminus  \big(E \times (-1/2,1/2)\big)$ we notice that by (b) of Lemma \ref{lemmalinearalgebra}, $g_{j,\rho} \equiv \text{Id}$. Thus, we estimate
\[
\begin{aligned}
	&\limsup_{ j \to+\infty}\limsup_{\rho \to 0} 
	\int_{\Sigma_1\setminus \big(E \times (-1/2,1/2)\big)}
	W\!\left(\nabla_\alpha u_{j,\rho}\,\bigg|\,
	\frac{1}{\rho}\partial_3 u_{j,\rho}\right)\,dx \\
	&\le 
	\limsup_{ j \to+\infty}\limsup_{\rho \to 0} 
	\int_{\Sigma_1 }
	W\!\left(\nabla_\alpha w_{j,\rho}\,\bigg|\,
	\frac{1}{\rho}\partial_3 w_{j,\rho}\right)\,dx \\
	&=
	\limsup_{ j \to+\infty}\limsup_{\rho \to 0}
	\int_{\Sigma_1 }
	W\!\left(
	\nabla_\alpha \widehat{v}_j(x)D\Phi_\delta(x)
	+ \rho x_3\nabla_\alpha\widehat{h}_j(x)D\Phi_\delta(x)\,\bigg|\,
	\widehat{h}_j(x)
	\right)\,dx \\
	&=
	\limsup_{ j \to+\infty}
	\int_{\Sigma_1 }
	W\!\left(
	\nabla_\alpha \widehat{v}_j(x)D\Phi_\delta(x)\,\bigg|\,
	\widehat{h}_j(x)
	\right)\,dx,
\end{aligned}
\]
where in the last equality we applied dominated convergence thanks to $(A_4)$, \eqref{eqmatrixu_jrho}, and \eqref{eqmatrixu_jrho2}.
We now perform a change of variables using $\Phi_\delta^{-1}$.  
Then, using \eqref{eq3} and \eqref{eq4}, we obtain
\[
\begin{aligned}
	&\limsup_{ j \to+\infty} 
	\int_{\Sigma_1}
	W\!\left(
	\nabla_\alpha \widehat{v}_j(x)D\Phi_\delta(x)\,\bigg|\,
	\widehat{h}_j(x)
	\right)\,dx \\
	&=
	\limsup_{ j \to+\infty}
	\int_{\Omega^\delta_1}
	W\!\left(
	\nabla_\alpha v_j(y)D\Phi_\delta(\Phi^{-1}_\delta(y))\,\bigg|\,h_j(y)
	\right)
	\left|\det(D\Phi^{-1}_\delta(y))\right|
	\,dy \\
	&\le 
	\limsup_{ j \to+\infty}
	(1+\varepsilon)
	\int_{\Omega^\delta_1}
	W\!\left(
	\nabla_\alpha v_j(y)D\Phi_\delta(\Phi^{-1}_\delta(y))\,\bigg|\,h_j(y)
	\right)\,dy \\
	&\le 
	\limsup_{ j \to+\infty}
	(1+\varepsilon)
	\int_{\Omega^\delta_1}
	W_0\!\left(
	\nabla_\alpha v_j(y)D\Phi_\delta(\Phi^{-1}_\delta(y))
	\right)\,dy
	+2\varepsilon.
\end{aligned}
\]

We may now pass to the limit as $j\to\infty$ using Vitali’s convergence theorem, Lemma~\ref{lemma2}, and Lemma~\ref{lemmasucc}.  
Finally, performing a change of variables with $\Phi_\delta$ and using \eqref{eq3}, we obtain
\[
\begin{aligned}
	&\limsup_{ j \to+\infty}
	(1+\varepsilon)
	\int_{\Omega^\delta_1}
	W_0\!\left(
	\nabla_\alpha v_j(y)D\Phi_\delta(\Phi^{-1}_\delta(y))
	\right)\,dy
	+2\varepsilon \\
	&=
	(1+\varepsilon)
	\int_{\Omega^\delta_1}
	W_0\!\left(
	\nabla_\alpha v(y)D\Phi_\delta(\Phi^{-1}_\delta(y))
	\right)\,dy
	+2\varepsilon \\
	&\le
	(1+\varepsilon)^2
	\int_{\Sigma_1}
	W_0\!\left(
	\nabla_\alpha v(\Phi_\delta(x))D\Phi_\delta(x)
	\right)\,dx
	+2\varepsilon \\
	&=
	(1+\varepsilon)^2
	\int_{\Sigma_1}
	W_0\!\left(
	\nabla_\alpha u(x)
	\right)\,dx
	+2\varepsilon.
\end{aligned}
\]

Taking also into account \eqref{eq5} and \eqref{eqlimsupj}, this proves \eqref{eq1}. 
Since $\varepsilon$ is arbitrary, we conclude that  
\[
\Gamma\text{-}\limsup_{\rho\to 0}\rho^{-1}\mathcal{G}_\rho(u)
\le 
\mathcal{G}^w_0(u),
\] 
and the claim is proven.

Finally, by lower semicontinuity of the $\Gamma\text{-}\limsup$ (see \cite[Proposition 6.8]{dal2012introduction}), the assertion of the Theorem follows.
\end{proof}

Roughly speaking, what we have been able to prove is that 
\[
\mathcal{G}_0(u)
\le 
\Gamma\text{-}\liminf_{\rho \to 0}\,\rho^{-1}\mathcal{G}(u)
\le 
\Gamma\text{-}\limsup_{\rho \to 0}\,\rho^{-1}\mathcal{G}(u)
\le 
\overline{\mathcal{G}^w_0}(u).
\]

Observe that we cannot yet apply Theorem~\ref{teor69} to the functional $\overline{\mathcal{G}^w_0}$ since the upper $p$-growth condition appearing in Theorem is not satisfied by the bulk energy. Therefore, we prove as an intermediate step that 
\[
\overline{\mathcal{G}^w_0}=\overline{\mathcal{G}^{\mathcal{R}}_0},
\]
where
\[
\mathcal{G}^{\mathcal{R}}_0(u)=
\begin{cases}
	\displaystyle 
	\int_{\Sigma} \mathcal{R}W_0(\nabla u)\,dx
	+
	\int_{J_u}\psi_0([u],\nu_u)\, d\mathcal{H}^1 
	& \text{if } u \in GSBV^p(\Sigma;\mathbb{R}^3),\\[1.0em]
	+\infty & \text{otherwise}.
\end{cases}
\]

Here, $\mathcal{R}W_0$ denotes the rank–one convex envelope of $W_0$. In Lemma~\ref{lemma3} we have shown that $\mathcal{R}W_0$ is finite-valued, continuous, and satisfies a $p$-growth condition. This will allow us to apply Theorem \ref{teor69} to the relaxation of the functional $\mathcal{G}^{\mathcal{R}}_0$ and eventually show that $\overline{\mathcal{G}^{\mathcal{R}}_0}=\mathcal{G}_0$ on $GSBV^p$. Thus, we now focus on proving that $\overline{\mathcal{G}^w_0} = \overline{\mathcal{G}^{\mathcal{R}}_0}$. To this end, we need two preliminary lemmas. The first one is a result from Kohn and Strang \cite[Section 5C]{kohn1986optimal}, while the second one is a technical Lemma contained in \cite{hafsa2008nonlinear2} (see also \cite[Lemma A.1]{almi2023brittlemembranes} for a proof) which is crucial for our asymptotic analysis. 
\begin{lemma}\label{lemma6}
	We define a sequence $\mathcal{R}_i W_0$ such that $\mathcal{R}_0 W_0 = W_0$ and, for every $i \ge 0$ and every $F \in \mathbb{M}^{3 \times 2}$,
	\[
	\mathcal{R}_{i+1}W_0(F)
	=
	\inf_{a \in \mathbb{R}^2,\; b \in \mathbb{R}^2,\; \lambda \in [0,1]}
	\Big\{
	(1-\lambda)\,\mathcal{R}_i W_0(F - \lambda\, b \otimes a)
	+
	\lambda\,\mathcal{R}_i W_0(F - (1-\lambda)\, b \otimes a)
	\Big\}.
	\]
	We have that $\mathcal{R}_i W_0$ is upper semicontinuous for every $i \ge 0$ and finite-valued for every $i \ge 2$. Moreover, $\mathcal{R}_{i+1}W_0 \le \mathcal{R}_i W_0$ for every $i \ge 0$, and
	\[
	\mathcal{R}W_0 = \inf_{i \ge 0} \mathcal{R}_i W_0.
	\]
\end{lemma}

\begin{lemma}\label{lem belgacem}
	Let $V \subseteq \Sigma$ be an open set such that $|\partial V|=0$ and let $F \in \mathbb{M}^{3 \times 2}$ have rank $2$. Then there exists a sequence $\{h_{n,l,q}\}_{n,l,q \ge 1} \subseteq \mathrm{Aff}_C(V;\mathbb{R}^3)$ such that:
	\begin{enumerate}
		\item for every $l,q \ge 1$, we have $\lim_{n\to+\infty} h_{n,l,q} = 0$ in $L^p(V;\mathbb{R}^3)$;
		\item for every $n,l,q \ge 1$, the function $x \mapsto h_{n,l,q}(x) + F x + c$ belongs to $\mathrm{Aff}^*(V;\mathbb{R}^3)$ for every $c \in \mathbb{R}^3$;
		\item 
		\[
		\lim_{q\to+\infty}\,\lim_{l\to+\infty}\,\lim_{n\to+\infty}
		\int_V \mathcal{R}_i W_0(\nabla h_{n,l,q} + F)\, dx
		\;\le\;
		|V|\, \mathcal{R}_{i+1}W_0(F).
		\]
	\end{enumerate}
\end{lemma}

\begin{prop}\label{propgamma_limsup}
	Let $\mathcal{G}^{\mathcal{R}}_0$ be as defined above. For every $u \in GSBV^p(\Sigma;\mathbb{R}^3)$ we have
	\[
	\Gamma\text{-}\limsup_{\rho \to 0} \rho^{-1}\mathcal{G}_\rho(u)
	\;\le\;
	\overline{\mathcal{G}^{\mathcal{R}}_0}(u).
	\]
\end{prop}

\begin{proof}
	We divide the proof into two steps.

	\noindent\textit{Step~1.}
	We show that for every
	$v \in \mathrm{Aff}^*(\Sigma \setminus \overline{J_v};\mathbb{R}^3)
	\cap \widehat{\mathcal{W}}(\Sigma;\mathbb{R}^3)$,
	\[
	\overline{\mathcal{G}^w_0}(v)
	\le
	\int_{\Sigma} \mathcal{R}W_0(\nabla v)\,dx
	+
	\int_{J_v} \psi_0([v],\nu_v)\, d\mathcal{H}^1.
	\]
	The proof follows the same lines as~\cite[Proposition~4.14, Step~1]{almi2023brittlemembranes}. The key observation is that the optimizing sequence provided by Lemma~\ref{lem belgacem} has support compactly contained in the complement of the jump set of $v$, and therefore does not interact with the surface energy. Consequently, the argument of~\cite{almi2023brittlemembranes}, originally carried out for $\psi(z,\nu) =| \nu|$, extends verbatim to a general surface energy density $\psi$ satisfying $(B_1)$--$(B_5)$.

        \noindent \textit{Step 2.}
        We now move to the general case $u \in GSBV^p(\Sigma;\mathbb{R}^3)$.  
		From Theorem~\ref{teorgamma_limsup} we already know that for every $u \in GSBV^p(\Sigma;\mathbb{R}^3)$,
		\[
		\Gamma\text{-}\limsup_{\rho \to 0} \rho^{-1}\mathcal{G}_\rho(u)
		\le 
		\overline{\mathcal{G}^w_0}(u).
		\]
		We now want to show that for every $u \in GSBV^p(\Sigma,\mathbb{R}^3)$,
		\[
		\overline{\mathcal{G}^w_0}(u)
		\le
		\int_{\Sigma} \mathcal{R}W_0(\nabla u)\,dx
		+
		\int_{J_u} \psi_0([u],\nu_u)\, d\mathcal{H}^1.
		\]
		To this end, by applying Corollary \ref{corlimsup} and Remark \ref{oss3}, we obtain the existence of a sequence 
		\[
		u_j \in \mathrm{Aff}^*(\Sigma \setminus \overline{J_u}, \mathbb{R}^3)
		\cap \widehat{\mathcal{W}}(\Sigma;\mathbb{R}^3)
		\]
		such that $u_j \to u$ in measure, $\nabla u_j \to \nabla u$ in $L^p(\Omega;\mathbb{M}^{3 \times 2})$, and
		\[
		\limsup_{j\to+\infty}
		\int_{J_{u_j}} \psi([u_j],\nu_{u_j})\, d\mathcal{H}^{1}
		\le
		\int_{J_u} \psi([u],\nu_u)\, d\mathcal{H}^{1}.
		\]
		By applying Lemma~\ref{lemma3} together with the convergence 
		$\nabla u_j \to \nabla u$ in $L^p$, we obtain
		\[
		\lim_{j\to+\infty}
		\int_{\Sigma} \mathcal{R}W_0(\nabla u_j)\,dx
		=
		\int_{\Sigma} \mathcal{R}W_0(\nabla u)\,dx.
		\]
		Therefore, using the lower semicontinuity of the functional $\overline{\mathcal{G}^w_0}$, we have
		\[
		\begin{aligned}
			\overline{\mathcal{G}^w_0}(u)
			&\le 
			\liminf_{j\to+\infty} \overline{\mathcal{G}^w_0}(u_j) \\
			&\le 
			\liminf_{j\to+\infty}
			\left(
			\int_{\Sigma} \mathcal{R}W_0(\nabla u_j)\,dx
			+
			\int_{J_{u_j}} \psi([u_j],\nu_{u_j})\, d\mathcal{H}^1
			\right) \\
			&\le\int_{\Sigma} \mathcal{R}W_0(\nabla u)\,dx+
			\int_{J_u} \psi([u],\nu_u)\, d\mathcal{H}^1.
		\end{aligned}
		\]
		Passing to the relaxation, we obtain 
		\[
		\overline{\mathcal{G}^w_0}(u)
		\le
		\overline{\mathcal{G}^{\mathcal{R}}_0}(u),
		\]
		which proves the claim.
\end{proof}

\subsection{Main result}
We now have all the tools to prove the main result in the orientation-preserving case.

\begin{teo}\label{teorfinal1}
	Let $\Sigma$ be a bounded open subset of $\mathbb R^2$ with Lipschitz boundary and let $1<p<\infty$.  
	Let $W $ satisfy assumptions $(A_1)-(A_4)$, and let $\psi$ satisfy hypotheses $(B_1)-(B_5)$.
	Then,
	\[
	\Gamma\text{-}\lim_{\rho \to 0}\rho^{-1}\mathcal{G}_\rho=\mathcal{G}_0
	\]
	in the space of measurable functions with respect to convergence in measure.
\end{teo}

\begin{proof}
	By Lemma~\ref{lemma3}, we may apply relaxation to the functional 
	$\overline{\mathcal{G}^R_0}$. Therefore, using Theorem~\ref{teor69}, for every $u \in GSBV^p(\Sigma, \mathbb R^3)$ we obtain
	\[
	\Gamma\text{-}\limsup_{\rho \to 0}\rho^{-1}\mathcal{G}_\rho(u)
	\le 
	\int_{\Sigma}\mathcal{Q}(\mathcal{R}W_0)(\nabla u)\, dx
	+
	\int_{J_u} \mathcal{B}\psi_0([u],\nu_u)\, d\mathcal{H}^1.
	\]
	Combining this with Proposition~\ref{prop1}, we deduce that
	\[
	\int_{\Sigma}\mathcal{Q}(\mathcal{R}W_0)(\nabla u)\, dx
	+
	\int_{J_u} \mathcal{B}\psi_0([u],\nu_u)\, d\mathcal{H}^1
	\le 
	\mathcal{G}_0(u)
	\le 
	\Gamma\text{-}\liminf_{\rho \to 0}\rho^{-1}\mathcal{G}_\rho(u).
	\]
	Therefore, for every function $u \in GSBV^p(\Sigma, \mathbb R^3)$,
	\[
	\int_{\Sigma}\mathcal{Q}(\mathcal{R}W_0)(\nabla u)\, dx
	+
	\int_{J_u} \mathcal{B}\psi_0([u],\nu_u)\, d\mathcal{H}^1
	=
	\Gamma\text{-}\lim_{\rho\to 0}\rho^{-1}\mathcal{G}_\rho(u).
	\]
	
	It only remains to show that $\mathcal{Q}W_0=\mathcal{Q}(\mathcal{R}W_0)$.  
	One inequality, namely “$\ge$”, is trivial.  
	For the reverse inequality, notice that by the definition of rank-one convexification and by Lemma~\ref{lemma3}, we have $\mathcal{Q}W_0 \le \mathcal{R}W_0$.  
	Applying $\mathcal{Q}$ to both sides yields
	\[
	\mathcal{Q}W_0 \le \mathcal{Q}(\mathcal{R}W_0).
	\]
	This concludes the proof of the theorem.
\end{proof}
\section{Dimension Reduction under Incompressibility}
In this section we consider the incompressible case. We recall that in this setting the bulk energy $W : \mathbb{M}^{3 \times 3} \to [0,+\infty]$ satisfies $(\widetilde{A}_1)$--$(\widetilde{A}_2)$, while $\psi \colon (\R^3 \setminus \{0\}) \times \R^3 \to [0,+\infty)$ satisfies $(B_1)$--$(B_5)$.
We start observing that by assumption $(\widetilde{A}_2)$ on $W$, given $F = (F_1 \mid F_2) \in \mathbb{R}^{3 \times 2}$ with $|F_1 \wedge F_2| \neq 0$, the following bounds hold
\begin{equation}\label{3.3.2}
\frac{1}{c}|F|^p + \frac{1}{c}\frac{1}{|F_1 \wedge F_2|^p} - c
\le
W_0(F)
\le
c|F|^p + c\frac{1}{|F_1 \wedge F_2|^p} + c .
\end{equation}
If instead $\operatorname{rank}(F) \le 1$, then $W_0(F) = +\infty$.

Thanks to \eqref{3.3.2}, $W_0$ satisfies all the assumptions of Lemma \ref{lemma3}.

Thus, following the same arguments as in Proposition~\ref{prop1}, we obtain the following result. 

\begin{prop}\label{prop1I}
		Let $\Sigma$ be a bounded open subset of $\mathbb{R}^2$ with Lipschitz boundary and let $1<p<+\infty$. Let $W$ be a function satisfying $(\widetilde{A}_1)$ and $(\widetilde{A}_2)$, and let $\psi$ be a function satisfying $(B_1)$–$(B_5)$. Then, for every sequence $u_\rho \in GSBV^p(\Sigma_1;\mathbb{R}^3)$ such that $u_\rho \to u$ in measure on $\Sigma_1$, one has
		\[
		\liminf_{\rho\to 0}\rho^{-1}\mathcal{G}_{\rho}(u_\rho)\ge\mathcal{G}_0(u).
		\]
\end{prop}

To prove the analogue of Theorem~\ref{teorgamma_limsup}, however, we require the following technical lemma, which is based on a construction from \cite[Proposition~5.1]{ContiDolzmann2006}. For completeness, and in order to make explicit the dependence of certain constants, we include the full proof.

	\begin{lemma}\label{lemma1}
		Let $u \in W^{k,\infty}(\Sigma \setminus \overline{J_u}, \mathbb{R}^3) \cap SBV^p(\Sigma, \mathbb{R}^3)$ and  
		$b \in W^{k,\infty}(\Sigma \setminus \overline{J_u}, \mathbb{R}^3)$ for every $k \geq 1$,  
		such that
		$$
		\det(\nabla u(x)\,|\, b(x)) = 1 \quad \text{for almost every } x \in \Sigma \setminus \overline{J_u}.
		$$
		Assume moreover that $J_u$ consists of a finite union of regular curves. Then there exist $\rho > 0$  
		and a function
		$$
		v \in W^{k,\infty}(\Sigma_\rho \setminus \overline{J_v}, \mathbb{R}^3) \cap SBV^p(\Sigma_\rho, \mathbb{R}^3)
		$$
		for every $k \geq 1$, such that $v(x,0) = u(x)$,  
		$J_v \subseteq J_u \times (-\rho/2, \rho/2)$, and
		$$
		\det(\nabla v(x)) = 1 \quad \text{for almost every } x \in \Sigma_\rho.
		$$
		Moreover, there exists a constant $C = C(u,b) > 0$ such that for almost every $x \in \Sigma_\rho$
		\begin{equation}\label{pw bound}
		|\nabla v(x) - (\nabla u(x)\,|\, b(x))| \leq C |x_3|.
		\end{equation}
	\end{lemma}

    	\begin{proof}
		Let $\rho >0$ small. Define the function $w \in SBV^p (\Sigma_\rho,\R^3)$ as $w(x_\alpha,x_3):=u(x_\alpha)+x_3b(x_\alpha)$. By construction we also have $w \in W^{k ,\infty}(\Sigma_\rho \setminus \overline{J_u \times (-\rho/2,\rho/2)},\R^3)$ for every $k \geq 1$. Set $J_{u,\rho}:=J_u \times (-\rho/2,\rho/2)$ for brevity. We have
		\begin{equation}\label{cd1}
			\det \nabla w(x_\alpha,x_3)= \det (\nabla u(x_\alpha)+x_3 \nabla b(x_\alpha)|b(x_\alpha))=1+x_3 P(x_\alpha)+x_3^2 Q(x_\alpha),
		\end{equation}
		where $P$ and $Q$ are combinations of the components of $\nabla u$, $b$ and $\nabla b$ which are all elements of $L^\infty(\Sigma)$ by assumption, hence uniformly bounded. Therefore, for $x_3 \in (-\rho/2,\rho/2)$, we have $|\det \nabla w(x_\alpha,x_3)-1| \leq C|x_3|$ for every $x_\alpha \in \Sigma \setminus \overline{J_u}$, where $C=C(u,b)>0$.
		
		Let $\{\Omega^i\}_{i=1}^\infty$ be a family of open sets with smooth boundary such that for every $i \geq 1$, $\Omega^i \subset \Omega^{i+1}$, $\Omega^i \subset \Sigma \setminus \overline{J_u}$ and $\Omega^i \to \Sigma \setminus \overline{J_u}$ as $i \to +\infty$ in the Hausdorff sense. The sets $\Omega^i$ might not be connected, however, the number of connected components composing the sets $\Omega^i$'s is uniformly bounded in $i$. Set $\Omega^i_\rho:=\Omega^i \times (-\rho/2,\rho/2)$. Then, for every $i \geq 1$, we can find $\Gamma^i: \Omega^i_\rho \to \R$ such that $\Gamma^i(x_\alpha,\cdot)$ is the solution of the following ODE
		\begin{equation}\label{eq diff}
			\begin{cases*}
				\displaystyle \partial_3 \Gamma^i(x_\alpha,x_3)=\frac{1}{\det\nabla w(x_\alpha,\Gamma^i(x_\alpha,x_3))} & for $x_3 \in (-\rho/2,\rho/2)$, \\
				\Gamma^i(x_\alpha,0)=0;
			\end{cases*}
		\end{equation} 
		for $\rho>0$ small and for every $x_\alpha \in \Omega^i$. If $\Omega^i$ is not connected, we solve the problem in each connected component.
		The set $\Omega_\rho^i$ has Lipschitz boundary which gives $w \in C^\infty(\overline{\Omega^i_\rho},\R^3)$. Therefore, by standard theory of parameter dependent ODEs, we have that $\Gamma^i$ is smooth on $\overline{\Omega_\rho^i}$. Moreover, by the previous bound on $\det \nabla w$, we also have $|\Gamma^i(x_\alpha,x_3)-x_3| \leq Cx_3^2$ for every $i \geq 1$. Thus, we infer that the function $v^i(x_\alpha,x_3):=u(x_\alpha)+\Gamma^i(x_\alpha,x_3)b(x_\alpha)$ is smooth in $\overline{\Omega_\rho^i}$ for $\rho$ small enough.
		
		Differentiating \eqref{cd1} with respect to $x_1$, for $x_3 \in (-\rho/2,\rho/2)$ and $x_\alpha \in \Sigma \setminus \overline{J_u}$ we get
		$$
		|\partial_1 \det \nabla w(x_\alpha,x_3)| \leq |x_3 \partial_1 P(x_\alpha)|+|x_3^2 \partial Q(x_\alpha)| \leq C |x_3|,
		$$
		where $C$ may now depends also on the second derivatives of $u$ and $b$. Differentiating with respect to $x_1$ the identity $x_3=\int_0^{\Gamma^i(x_\alpha,x_3)} \det \nabla w(x_\alpha,t) \, dt$ we get
		$$
		0=\int_0^{\Gamma^i(x_\alpha,x_3)} \partial_1 \det \nabla w(x_\alpha,t) \, dt + \det \nabla w(x_\alpha, \Gamma^i(x_\alpha,x_3)) \partial_1 \Gamma^i(x_\alpha,x_3),
		$$
		for $x_\alpha \in \Omega_\rho^i$. Therefore, recalling that for $\rho$ small we have $\det \nabla w \in [1/2,2]$, 
		$$
		|\partial_1 \Gamma^i(x_\alpha,x_3)| \leq 2 \int_0^{\Gamma(x_\alpha,x_3)} |\partial_1 \nabla w(x_\alpha,t)| \, dt \leq C x_3^2 \qquad \mbox{ for $(x_\alpha,x_3) \in \Omega_\rho^i$}. 
		$$
		The bound for $|\partial_2 \Gamma^i|$ can be obtained analogously. Notice that the bounds on the derivatives of $\Gamma^i$ hold with a constant $C=C(u,b)>0$ independent of $i$.
		Moreover, by uniqueness of solution of the ODE \eqref{eq diff}, we have that $\Gamma^{i+1}|_{\Omega^i_\rho}=\Gamma^i$. Thus, we can define a global function $\Gamma \in W^{k ,\infty}(\Sigma_\rho \setminus \overline{J_{u,\rho}})$ for every $k \geq 1$ such that $\Gamma|_{\Omega^i_\rho}=\Gamma_i$ for every $i \geq 1$ and the bounds on the derivatives of $\Gamma$ hold for every $x \in \Sigma_\rho \setminus \overline{J_{u,\rho}}$. In particular, $J_\Gamma \subseteq \overline{J_{u,\rho}}$.
		
		By setting $v(x_\alpha,x_3):=u(x_\alpha)+\Gamma(x_\alpha,x_3)b(x_\alpha)$, we have $J_v \subseteq \overline{J_{u,\rho}}$ and,
		since $\Gamma$ satisfies \eqref{eq diff}, $\det(\nabla v(x))=1$ for almost every $x \in \Sigma_\rho$. Finally, observing that
		$$
		\nabla v=(\partial_1 u+ \partial_1 \Gamma b+ \Gamma \partial_1 b \ | \ \partial_2 u+ \partial_2 \Gamma b+ \Gamma \partial_2 b \ | \ \partial_3 \Gamma b)
		$$
		and using in turn that $|\partial_1 \Gamma|+|\partial_2 \Gamma| \leq C x_3^2$ and $|\partial_3 \Gamma -1|\leq C|x_3|$, we obtain \eqref{pw bound} for almost every $x \in \Sigma_\rho$ with $C=C(u,b)$.
	\end{proof}
	
    In the next result we show that the map constructed in Lemma~\ref{lemmalinearalgebra} can be modified so as to satisfy the incompressibility constraint $\det \nabla f_\rho = 1$.

	\begin{lemma}\label{corollariodet}
		Let $U$, $V$, and $\Sigma$ be the open sets appearing in Lemma~\ref{lemmalinearalgebra}.  
		Then, for $\rho>0$ sufficiently small, there exists a map $f_\rho$ satisfying (a)--(f) of Lemma~\ref{lemmalinearalgebra} and, in addition,
		\[
		\det \nabla f_\rho = 1
		\quad \text{for every } x \in \Sigma_{2} .
		\]
	\end{lemma}
	
	\begin{proof}
        Without loss of generality, we may assume that $\Sigma$ is connected; otherwise, it suffices to argue on each
		connected component separately.
        Let $\{g_\rho\}_\rho$ be the family of diffeomorphisms given by Lemma~\ref{lemmalinearalgebra}.
		  Consider the sequence of functions
		\[
		u_\rho(x_\alpha) := g_\rho(x_\alpha,0).
		\]
	    Since, by Lemma~\ref{lemmalinearalgebra},
		$g_\rho \to \mathrm{Id}$ in $W^{1,\infty}(\Sigma_2;\R^3)$, there exists $\rho>0$ sufficiently small such that
		\[
		\det(\nabla g_\rho) \geq \tfrac{1}{2}
		\]
		uniformly in $\Sigma_2$. Hence, the sequence
		\[
		b_\rho(x_\alpha) := \frac{\partial_3 g_\rho(x_\alpha,0)}{\det(\nabla g_\rho(x_\alpha,0))}
		\]
		is well defined.
		
		The assumptions of Lemma~\ref{lemma1} are therefore satisfied with $J_{u_\rho}=\emptyset$. 
        Moreover, using again that $g_\rho \to \mathrm{Id}$ in $W^{1,\infty}(\Sigma_2;\R^3)$, for $\rho$ sufficiently small the ODE \eqref{eq diff} with
        \[
        w_\rho(x_\alpha,x_3) := u_\rho(x_\alpha) + x_3 b_\rho(x_\alpha)
        \]
        is solvable on the full interval $(-1,1)$, giving $\Gamma_\rho \colon \Sigma_2 \to \R$ which solves the ODE. Hence, we infer the existence of a sequence of $C^\infty$ functions $f_\rho$ on $\Sigma_2$ such that $f_\rho(x_\alpha,0)=u_\rho(x_\alpha)$ and
		\[
		\det(\nabla f_\rho(x)) = 1 \qquad \mbox{for every $x \in \Sigma_2$}.
		\]
		We recall that the function $f_\rho$ is defined in the following way,
		\[
		f_\rho(x) = u_\rho(x_\alpha) + \Gamma_\rho(x)\, b_\rho(x_\alpha),
		\]
		where $\Gamma_\rho$ solves the ODE and thus the integral equation
		\[
		\Gamma_\rho(x_\alpha,x_3)
		= \int_0^{x_3} \frac{1}{\det \nabla w_\rho(x_\alpha,\Gamma(x_\alpha,s))}\, ds.
		\]

        We now verify that $f_\rho$ satisfies properties~(a)--(f).
        
        In $\overline{U}_2$, by Lemma \ref{lemmalinearalgebra} (a), $g_\rho$ acts as the isometry $O_\rho$, so $u_\rho(x_\alpha) = O_\rho(x_\alpha, 0)$, $b_\rho(x_\alpha) = O_\rho\, e_3$ (since $\det O_\rho = 1$), and therefore $w_\rho(x) = O_\rho\, x$. It follows that $\det \nabla w_\rho = 1$, so the ODE gives $\Gamma_\rho(x_\alpha, x_3) = x_3$. By construction, $f_\rho(x) = O_\rho\, x = g_\rho(x)$ in $\overline{U}_2$, which yields~(a).

        Similarly, by Lemma \ref{lemmalinearalgebra} (b), in $\Sigma_2 \setminus \overline{V_2}$, $g_\rho = \mathrm{Id}$, so $u_\rho(x_\alpha) = (x_\alpha, 0)$, $b_\rho = e_3$, $w_\rho = \mathrm{Id}$, and again $\Gamma_\rho = x_3$, giving $f_\rho = \mathrm{Id}$, which yields (b).

        We now prove (c). Since $g_\rho \to \mathrm{Id}$ in $W^{k,\infty}(\Sigma_2;\mathbb{R}^3)$ for every $k \geq 1$, we have $u_\rho \to (\cdot, 0)$ and $b_\rho \to e_3$ in $W^{k,\infty}(\Sigma_2;\mathbb{R}^3)$, hence $w_\rho \to \mathrm{Id}$ in $W^{k,\infty}(\Sigma_2;\mathbb{R}^3)$.
        Consequently,
		   \[
		   \det \nabla w_\rho \to 1 \quad \text{in } W^{k-1,\infty}(\Sigma_2;\R^3),
		   \]
        which implies $\Gamma_\rho \to x_3$ in $W^{k-1,\infty}(\Sigma_2;\R^3)$, and therefore $f_\rho \to \mathrm{Id}$ in $W^{k-1,\infty}(\Sigma_2;\mathbb{R}^3)$ for every $k \geq 1$.
        

        Point (d) follows from~(a) and~(c) by the same argument used in the proof of Lemma~\ref{lemmalinearalgebra}.
		
		It remains to establish (e) and (f). Since (f) is a direct consequence of (c) and (e), it suffices to prove (e). 
        
        This follows immediately by estimating, for $i=1,2$ and every $x \in \Sigma_2$,
		\[
        \frac{1}{\rho} \left|(\nabla f_\rho)^3_{i}(x)\right|
        = \frac{1}{\rho} |\partial_3 \Gamma_\rho(x)|\, \frac{|(\partial_3 g_\rho(x_\alpha,0))_i|}{|\det (\nabla g_\rho(x_\alpha,0))|} \leq 2 \frac{|(\partial_3 g_\rho(x_\alpha,0))_i|}{\rho} \leq 2|\zeta|,
        \]
		  where we used that $\det \nabla g_\rho \to 1$ and $\partial_3\Gamma_\rho \to 1$ uniformly as $\rho \to 0$, together with property~(e) of Lemma~\ref{lemmalinearalgebra}.
	\end{proof}

    \begin{oss}\label{ossf_rho2}
		As in Remark \ref{ossf_rho}, using (b) and (c) of Lemma  \ref{corollariodet}, we infer that for $\rho$ sufficiently small, the map $f_\rho$ is a $C^\infty$ diffeomorphism from $\Sigma_2$ to its image. Moreover, for every $x \in \Sigma_2$,
		\[
		\nabla f^{-1}_\rho\bigl(f_\rho(x)\bigr) = \bigl[\nabla f_\rho(x)\bigr]^{-1}.
		\]
		Again, using this identity, one can show that properties (c),(e) and (f) Lemma \ref{corollariodet} also hold for $f^{-1}_\rho$ on the image of $\Sigma_2$ through $f_\rho$.
	\end{oss}

	We are now in a position to prove an analogue of Theorem~\ref{teorgamma_limsup} for the upper bound of the $\Gamma$-$\limsup$.

	\begin{teo}\label{teor3+}
		Let $u \in GSBV^p(\Sigma_1;\mathbb{R}^3)$. Then
		\[
		\Gamma\text{-}\limsup_{\rho \to 0} \rho^{-1}\mathcal{G}_\rho(u)
		\le \overline{\mathcal{G}_0^{\,w}}(u),
		\]
		where $\overline{\mathcal{G}_0^{\,w}}(u)$ denotes the relaxation of $\mathcal{G}_0^{\,w}(u)$.
	\end{teo}
	
	\begin{proof}
    We divide the proof in three steps.
    
	\textit{Step 1: Construction of the recovery sequence}.	Let $u \in Y$. Its trace on $\Sigma$ belongs to the space
		\[
		\mathrm{Aff}^*(\Sigma \setminus \overline{J_u};\mathbb{R}^3)
		\cap \widehat{\mathcal{W}}(\Sigma;\mathbb{R}^3).
		\]
		In particular, there exists a finite family of open sets $\{T^i\}_{i=1}^N$ such that
		$\mathcal{L}^2\!\left(\Sigma \setminus \bigcup_{i=1}^N T^i\right)=0$,
		$\mathcal{L}^2(\partial T^i)=0$ for every $i=1,\dots,N$, and
		$\nabla u|_{T^i}=A_i$ with $A_i \in \mathbb{R}^{3\times 2}$ satisfying
		\[
		|A_i^1 \wedge A_i^2| \ge \eta > 0,
		\qquad \text{for every } i=1,\dots,N.
		\]
		
		By the coercivity of $W$, for each $A_i$ there exists $b_i \in \mathbb{R}^3$ such that
		$W(A_i \mid b_i)=W_0(A_i)$. Hence, we can find a piecewise constant function
		$b:\Sigma \to \mathbb{R}^3$ such that $b|_{T^i}=b_i$ for every $i=1,\dots,N$ and
		\begin{equation}\label{eqnabla}
			\det(\nabla u(x)\mid b(x))=1,
			\qquad
			W(\nabla u(x)\mid b(x))=W_0(\nabla u(x))
			\quad \text{for a.e. } x \in \Sigma.
		\end{equation}
		
		Since $b$ satisfies the determinant constraint in \eqref{eqnabla}, for a.e.\ $x\in\Sigma$
		the function $b$ can be written as
		\[
		b(x)=\lambda(x) \, \partial_1 u(x)+\Lambda(x) \, \partial_2 u(x)
		+|\partial_1 u(x)\wedge \partial_2 u(x)|^{-2}
		\,\partial_1 u(x)\wedge \partial_2 u(x),
		\]
		where $\lambda,\Lambda:\Sigma \to \mathbb{R}$ are piecewise constant functions.
		Notice that $b$ is piecewise constant on a finite partition of $\Sigma$, hence
		$b\in L^\infty(\Sigma)$, which in particular implies $\lambda,\Lambda\in L^\infty(\Sigma)$.
		Therefore, there exist sequences $\lambda_j,\Lambda_j\in C_c^\infty(\Sigma)$ such that
		$\lambda_j\to \lambda$ and $\Lambda_j\to \Lambda$ in $L^p(\Sigma)$ as $j\to+\infty$, with
		\begin{equation}\label{eq30}
		\|\lambda_j\|_{L^\infty(\Sigma)}\le \|\lambda\|_{L^\infty(\Sigma)},
		\qquad
		\|\Lambda_j\|_{L^\infty(\Sigma)}\le \|\Lambda\|_{L^\infty(\Sigma)}.
		\end{equation}
		
		Moreover, by applying Proposition~\ref{propu_jbuona} to $u$, we find a sequence
		\[
		u_j \in SBV^p(\Sigma;\mathbb{R}^3)
		\cap W^{k,\infty}(\Sigma \setminus \overline{J_{u_j}};\mathbb{R}^3)
		\quad \text{for every } k\ge 1,
		\]
		satisfying properties (i)--(iv) of the proposition.
		
		For every $x \in \Sigma \setminus J_u$, we define
		\[
		b_j(x)
		:= \lambda_j(x) \,\partial_1 u_j(x)+\Lambda_j(x) \,\partial_2 u_j(x)
		+|\partial_1 u_j(x)\wedge \partial_2 u_j(x)|^{-2}
		\,\partial_1 u_j(x)\wedge \partial_2 u_j(x).
		\]
		By construction, $J_{b_j}\subseteq J_{u_j}\Subset \Sigma$ and
		$b_j \in W^{k,p}(\Sigma \setminus \overline{J_{b_j}};\mathbb{R}^3)$ for every $k\ge 1$.
		Furthermore, $b_j \to b$ pointwise a.e.\ in $\Sigma$ and, by properties (i)--(v) of
		Proposition~\ref{propu_jbuona} and \eqref{eq30}, for every $j\ge 1$,
		\begin{equation}
		\|b_j\|_{L^\infty}(\Sigma;\R^3)
		\le
		8\big(\|\lambda\|_{L^\infty(\Sigma)}+\|\Lambda\|_{L^\infty(\Sigma)}\big)\|\nabla u\|_{L^\infty(\Sigma;\mathbb{M}^{3\times 2})}
		+\frac{64}{\theta^2}\|\nabla u\|_{L^2(\Sigma;\mathbb{M}^{3\times 2})}^2.
		\end{equation}
		This implies that $b_j \to b$ in $L^p(\Sigma;\mathbb{R}^3)$.
		
		We define $J_{j,\rho}:=J_{u_j}\times(-\rho/2,\rho/2)$ for $j\ge 1$ and $\rho>0$.
		Notice that the pair $(u_j,b_j)$ satisfies the assumptions of Lemma~\ref{lemma1}
		for every $j\ge 1$. Therefore, letting $j\to+\infty$, we can select a sequence
		$\rho_j\to 0$ and construct a family
		\[
		v_j \in W^{k,\infty}(\Sigma_{\rho_j}\setminus \overline{J_{j,\rho_j}};\mathbb{R}^3)
		\cap SBV^p(\Sigma_{\rho_j};\mathbb{R}^3)
		\quad \text{for every } k\ge 1,
		\]
		such that $v_j(x_\alpha,0)=u_j(x_\alpha)$, $J_{v_j}\subseteq J_{j,\rho_j}$, and
		\[
		\det(\nabla v_j(x))=1
		\quad \text{for a.e. } x\in \Sigma_{\rho_j}.
		\]
		Moreover, by \eqref{pw bound}, for every $j\ge 1$ there exists a constant $C_j>0$ such that, for a.e.\
		$x\in \Sigma_{\rho_j}$,
		\begin{equation}\label{eqb_j}
			\big|\nabla v_j(x)-(\nabla u_j(x_\alpha)\mid b_j(x_\alpha))\big|
			\le C_j |x_3|.
		\end{equation}
        We now focus on the jump set of $u_j$. By property~(v) of Proposition~\ref{propu_jbuona}, for every $\varepsilon > 0$ and $j$ sufficiently large, there exists an open set $A_\varepsilon \subseteq \Sigma$ such that $J_{u_j} \cap A_\varepsilon = J_u \cap A_\varepsilon$, $\mathcal{H}^1(J_{u_j} \setminus A_\varepsilon) < \varepsilon$, and $J_u \cap A_\varepsilon$ is a finite union of disjoint segments. Thus, arguing as in Step~1 of Theorem~\ref{teorgamma_limsup} (with Lemma~\ref{corollariodet} in place of Lemma~\ref{lemmalinearalgebra} to ensure the incompressibility constraint $\det \nabla f_{j,\rho} = 1$ and with $J_{u_j} \cap A_\varepsilon$ in place of $J_u$), we construct a map $f_{j,\rho} : \Sigma_2 \to \mathbb{R}^3$ satisfying properties~(a)--(f) of Lemma~\ref{lemmalinearalgebra}, $\det \nabla f_{j,\rho} = 1$ on $\Sigma_2$, and $f_{j,\rho} \to \mathrm{Id}$ in $W^{1,\infty}(\Sigma_2;\mathbb{R}^3)$ as $\rho \to 0$. We also find a set $E \subset \Sigma$ such that $\mathcal{L}^2(E) \leq \varepsilon$, $J_u \cap A_\varepsilon \subset E$, and $f_{j,\rho} \equiv \text{Id}$ on $\Sigma_2 \setminus \left(E \times (-1,1) \right)$.

		Next, for every $\rho \leq \rho_j/2$, we denote by
		\[
		w_{j,\rho}(x):=v_j(x_\alpha,\rho x_3)
		\]
		and finally, we define our \emph{recovery sequence} as
		\[
		u_{j,\rho}:=w_{j,\rho}(g_{j,\rho}(x)),
		\]
		where \( g_{j,\rho}=f^{-1}_{j,\rho} \).
		We note that, by Remark \ref{ossf_rho}, for \(\rho\) sufficiently small we have that $g_{j,\rho}$ is a $C^\infty$ diffeomorphism such that 
		\( g_{j,\rho}(\Sigma_1)\subseteq \Sigma_2 \); hence the function \( {u}_{j,\rho} \)
		is well defined.
        By construction \(  u_{j,\rho} \in SBV^p(\Sigma_1,\mathbb R^3) \). Moreover, $J_{{u}_{j,\rho}} \cap \left(A_\varepsilon \times (-1/2,1/2) \right) = J_{{u}_{j}} \cap \left(A_\varepsilon \times (-1/2,1/2) \right)$, and
	       \[
	       J_{{u}_{j,\rho}} \subseteq g^{-1}_{j,\rho}(J_{u_j} \times (-1/2,1/2))=f_{j,\rho}(J_{u_j} \times (-1/2,1/2)).
	       \]
        We can extend $u_j$ and $u$ to $\Sigma_1$ by setting them constant in $x_3$. Recalling Lemma \ref{lemma1}, we infer that for every $j \geq 1$
        \begin{equation}\label{eq31}
        \lim_{\rho \to 0} \Vert w_{j,\rho}-u_j \Vert_{L^\infty(\Sigma_1;\R^3)}=0.
        \end{equation}
		Furthermore, by a change of variables given by \( g_{j,\rho} \)
	and since \( g_{j,\rho}\to \text{Id} \) in \( W^{1,\infty} \), we obtain
    \[
    \begin{aligned}
        \limsup_{ j \to +\infty}\limsup_{\rho \to 0} \Vert u_{j,\rho}-u\Vert_{L^p(\Sigma_1)} & \le \limsup_{ j \to +\infty}\limsup_{\rho \to 0} \Vert u_{j,\rho}-w_{j,\rho}\Vert_{L^p}+\limsup_{ j \to +\infty}\limsup_{\rho \to 0} \Vert w_{j,\rho}-u\Vert_{L^p} \\
        & = \limsup_{ j \to +\infty}\limsup_{\rho \to 0} \Vert w_{j,\rho}-u\Vert_{L^p} = \limsup_{ j \to +\infty} \Vert u_j-u\Vert_{L^p}=0,
    \end{aligned}
    \]
    from which the convergence in \(L^p\) follows.
    

\textit{Step 2: Estimate of the surface energy}. Let us now consider the surface part of $\rho^{-1}\mathcal G_\rho$. Arguing as in the beginning of Step~2 of Theorem~\ref{teorgamma_limsup}, we infer that
    \begin{equation}\label{eq32}
        \psi_\rho([u_{j,\rho}],\nu_{u_{j,\rho}}) \leq C \qquad \mbox{on $J_{u_{j,\rho}}$}.
    \end{equation}
    Therefore, using that $\mathcal{H}^2\left(J_{u_{j,\rho}} \setminus \left(A_\varepsilon \times (-1/2,1/2) \right) \right) \leq \varepsilon$ by construction, we estimate
	\[
	\begin{aligned}
		\limsup_{\rho \to 0}
		\int_{J_{ u_{j,\rho}}}
		\psi_\rho\big([ u_{j,\rho}],\nu_{ u_{j,\rho}}\big)\, d\mathcal{H}^2
		&\le
		\limsup_{\rho \to 0}
		\int_{J_{ u_{j,\rho}} \cap \left(A_\varepsilon \times (-1/2,1/2) \right)}
		\psi_\rho\big([ u_{j,\rho}],\nu_{ u_{j,\rho}}\big)\, d\mathcal{H}^2 \\
		&\quad +
		\limsup_{\rho \to 0}
		\int_{J_{ u_{j,\rho}} \setminus \left(A_\varepsilon \times (-1/2,1/2) \right)}
		\psi_\rho\big([ u_{j,\rho}],\nu_{ u_{j,\rho}}\big)\, d\mathcal{H}^2 \\
		&\le
		\limsup_{\rho \to 0}
		\int_{J_{ u_{j,\rho}} \cap \left(A_\varepsilon \times (-1/2,1/2) \right)}
		\psi_\rho\big([ u_{j,\rho}],\nu_{ u_{j,\rho}}\big)\, d\mathcal{H}^2
		+ C \,\varepsilon.
	\end{aligned}
	\]
	Reasoning as in Step~2 of Theorem~\ref{teorgamma_limsup}, using Lemma~\ref{corollariodet} in place of Lemma~\ref{lemmalinearalgebra}, together with \eqref{eq32} and \eqref{eqb_j}, we obtain
\begin{equation}\label{eq5+}
	\limsup_{\rho \to 0} \int_{J_{u_{j,\rho}}} \psi_\rho([u_{j,\rho}], \nu_{u_{j,\rho}})\, d\mathcal{H}^2
	\le \int_{J_{u_j}} \psi_0([u_j], \nu_{u_j})\, d\mathcal{H}^1 + C\varepsilon.
\end{equation}
Since $u_j^\pm \to u^\pm$ uniformly on $\Sigma$ by Proposition~\ref{propu_jbuona} and $\mathcal{H}^1(J_{u_j} \,\triangle\, J_u) \to 0$, the upper semicontinuity of $\psi_0$ and the reverse Fatou lemma give
\begin{equation}\label{eqlimsupj+}
	\limsup_{j \to +\infty} \int_{J_{u_j}} \psi_0([u_j], \nu_{u_j})\, d\mathcal{H}^1
	\le \int_{J_u} \psi_0([u], \nu_u)\, d\mathcal{H}^1.
\end{equation}
	
    \textit{Step 3: Estimate of the bulk energy}. Notice that the derivative of ${u}_{j,\rho}$ is given by
	$$
	\nabla {u}_{j,\rho}(x)=\nabla w_{j,\rho}(g_{j,\rho}(x))\nabla g_{j,\rho}(x).
	$$
    Since $\det(\nabla v_j) = 1$ on $\Sigma_{\rho_j}$, $\|\nabla v_j\|_{L^\infty} \leq C_j$ for a constant $C_j > 0$ depending on $j$, and $g_{j,\rho} \to \mathrm{Id}$ in $W^{1,\infty}(\Sigma_2;\mathbb{R}^3)$, reasoning as in Step~3 of Theorem~\ref{teorgamma_limsup} we obtain
    \begin{equation}\label{eq33}
        \det\!\left(\nabla_\alpha u_{j,\rho}\,\bigg|\,
        \frac{1}{\rho}\partial_3 u_{j,\rho}\right) =1,
        \quad \text{and} \quad 
        \left\Vert\left(\nabla_\alpha u_{j,\rho}\,\bigg|\,
        \frac{1}{\rho}\partial_3 u_{j,\rho}\right)\right\Vert_{L^\infty(\Sigma_1;\mathbb{M}^{3 \times 3})} \leq C+ \rho C_j 
        \qquad \text{a.e.\ in } \Sigma_1.
    \end{equation}
    
    
    We split the bulk integral into contributions from $E \times (-1/2,1/2)$ and $\Sigma_1 \setminus (E \times (-1/2,1/2))$. 
	Since $\mathcal L^3\big(E \times (-1/2,1/2)  \le\varepsilon $, using \eqref{eq33} and $(\widetilde{A}_2)$, we infer 
	$$
	\limsup_{\rho \to 0} \int_{E \times (-1/2,1/2))}W\bigg(\nabla_\alpha u_{j,\rho}\bigg|\frac{1}{\rho}\partial_3 u_{j,\rho}\bigg)\,dx\le C \varepsilon.
	$$
	On $\Sigma_1 \setminus (E \times (-1/2,1/2))$, since $g_{j,\rho} \equiv \mathrm{Id}$ there, using \eqref{eqb_j}, the continuity of $W$, and the dominated convergence theorem (justified by \eqref{eq33} and $(\widetilde{A}_2)$),
		\[
		\begin{aligned}
				&\limsup_{ j \to+\infty}\limsup_{\rho \to 0} 
			\int_{\Sigma_1\setminus \big( E \times (-1/2,1/2)\big)}
			W\!\left(\nabla_\alpha w_{j,\rho}\,\bigg|\,
			\frac{1}{\rho}\partial_3 w_{j,\rho}\right)\,dx \\
			&\le \limsup_{j \to \infty}\limsup_{\rho \to 0}
			\int_{\Sigma_1}
			W\!\left(
			\nabla_\alpha w_{j,\rho}(x)\,\bigg|\,\frac{1}{\rho}\partial_3 w_{j,\rho}(x)
			\right)\, dx \\
			&=
			\limsup_{j \to \infty}
			\int_{\Sigma_1}
			W\!\left(
			\nabla u_j(x)\,\big|\, b_j(x)
			\right)\, dx.
		\end{aligned}
		\]
		Finally, by Vitali's convergence theorem, using that $b_j \to b$ in $L^p(\Sigma;\mathbb{R}^3)$, $\nabla u_j \to \nabla u$ in $L^p(\Sigma;\mathbb{M}^{3 \times 2})$, $\det(\nabla u_j \mid b_j) = 1$, the $p$-growth of $W$ on $\mathrm{SL}(3)$, and \eqref{eqnabla},
        \[
            \limsup_{j \to \infty}
			\int_{\Sigma_1}
			W\!\left(
			\nabla u_j(x)\,\big|\, b_j(x)
			\right)\, dx
			=
			\int_{\Sigma_1}
			W\!\left(
			\nabla u(x)\mid b(x)
			\right)\, dx 
			=
			\int_{\Sigma_1}
			W_0\!\left(
			\nabla u(x)
			\right)\, dx .
        \]
        Taking also into account \eqref{eq5+} and \eqref{eqlimsupj+}, this proves that
        \[
	\Gamma\text{-}\limsup_{\rho\to 0}\rho^{-1}\mathcal{G}_\rho(u)
	\le 
	\mathcal{G}^w_0(u) +C\varepsilon
	\qquad\text{for every }u\in Y.
	\]   
	Since $\varepsilon$ is arbitrary, we conclude that  
	\[
	\Gamma\text{-}\limsup_{\rho\to 0}\rho^{-1}\mathcal{G}_\rho(u)
	\le 
	\mathcal{G}^w_0(u)
	\qquad\text{for every }u\in Y.
	\]
	For the remaining functions $u\in GSBV^p(\Sigma_1,\mathbb R^3)\setminus Y$, the inequality above is trivially satisfied.  
	By semicontinuity of the $\Gamma\text{-}\limsup$, the assertion follows.
	\end{proof}

    We are now in a position to prove the main result in the incompressible case.
    
    \begin{teo}\label{teorfinal2}
	Let $\Sigma$ be a bounded open subset of $\mathbb R^2$ with Lipschitz boundary and let $1<p<\infty$.  
	Let $W $ satisfy assumptions $(\Tilde{A}_1)-(\Tilde{A}_2)$, and let $\psi$ satisfy hypotheses $(B_1)-(B_5)$.  
	Then,
	\[
	\Gamma\text{-}\lim_{\rho \to 0}\rho^{-1}\mathcal{G}_\rho=\mathcal{G}_0
	\]
	in the space of measurable functions with respect to convergence in measure.
\end{teo}

\begin{proof}
    We observe that $W_0$ satisfies the assumptions of Lemma~\ref{lemma3}. The proof then follows by combining Proposition~\ref{prop1I} and Theorem~\ref{teor3+} with the arguments of Proposition~\ref{propgamma_limsup} and Theorem~\ref{teorfinal1}.
\end{proof}

\begin{ackn}
    We are grateful to Stefano Almi and Francesco Solombrino for several useful discussions.
    D.R. gratefully acknowledge the support from the Cluster of Excellence EXC 2044-390685587, Mathematics Münster: Dynamics-Geometry-Structure funded by the Deutsche Forschungsgemeinschaft (DFG, German Research Foundation).
\end{ackn}

	\appendix

	\section*{Appendix A}

    \renewcommand{\thesection}{A} 
	
	\setcounter{figure}{0}
	\setcounter{equation}{0}
    \setcounter{teo}{0}

    In this section we recall several definitions and results concerning $SBV$ and $GSBV$ functions; we refer to~\cite{ambrosio2000functions} for a comprehensive treatment. Let $\Omega \subseteq \mathbb{R}^n$ be open and $f : \Omega \to \mathbb{R}^m$ be $\mathcal{L}^n$-measurable. We say that $a \in \mathbb{R}^m$ is the \emph{approximate limit} of $f$ at $x \in \Omega$ if
\[
\lim_{r \to 0}
\frac{
	\mathcal{L}^n\!\left(\Omega \cap B_r(x) \cap \{|f-a|>\varepsilon\}\right)
}{r^n}
= 0
\qquad \text{for every } \varepsilon > 0,
\]
and we write $\operatorname{ap}\text{-}\lim_{y\to x} f(y) = a$. A point $x \in \Omega$ is an \emph{approximate jump point} of $f$, written $x \in J_f$, if there exist $a,b \in \mathbb{R}^m$ with $a \neq b$ and $\nu \in \mathbb{S}^{n-1}$ such that
\[
\operatorname{ap}\text{-}\lim_{\substack{y\to x \\ (y-x)\cdot\nu >0}} f(y) = a,
\qquad
\operatorname{ap}\text{-}\lim_{\substack{y\to x \\ (y-x)\cdot\nu <0}} f(y) = b.
\]
The triple $(a,b,\nu)$ is unique up to the interchange $(a,b,\nu) \leftrightarrow (b,a,-\nu)$; we write $(f^+(x), f^-(x), \nu_f(x))$ and define the \emph{jump} $[f](x) := f^+(x) - f^-(x)$.

The space $BV(\Omega;\mathbb{R}^m)$ consists of those $u \in L^1(\Omega;\mathbb{R}^m)$ whose distributional gradient $Du$ is a finite $\mathbb{M}^{m \times n}$-valued Radon measure, which decomposes as $Du = D^a u + D^s u$, where $D^a u \ll \mathcal{L}^n$ and $D^s u \perp \mathcal{L}^n$. The density of $D^a u$ with respect to $\mathcal{L}^n$ is denoted $\nabla u$ and coincides a.e.\ with the approximate gradient. The jump set $J_u$ is countably $(n{-}1)$-rectifiable with approximate normal $\nu_u$.

The space $SBV(\Omega;\mathbb{R}^m)$ consists of those $u \in BV(\Omega;\mathbb{R}^m)$ with $|D^s u|(\Omega \setminus J_u) = 0$. For $p \in [1,+\infty)$,
\[
SBV^p(\Omega;\mathbb{R}^m)
:= \{ u \in SBV(\Omega;\mathbb{R}^m) : \nabla u \in L^p(\Omega;\mathbb{M}^{m \times n}),\;
\mathcal{H}^{n-1}(J_u) < +\infty \}.
\]

We say that $u \in GSBV(\Omega;\mathbb{R}^m)$ if $\phi(u) \in SBV_{\mathrm{loc}}(\Omega;\mathbb{R}^m)$ for every $\phi \in C^1(\mathbb{R}^m;\mathbb{R}^m)$ with compactly supported gradient. For $p \in [1,+\infty)$,
\[
GSBV^p(\Omega;\mathbb{R}^m)
:= \{ u \in GSBV(\Omega;\mathbb{R}^m) : \nabla u \in L^p(\Omega;\mathbb{M}^{m \times n}),\;
\mathcal{H}^{n-1}(J_u) < +\infty \}.
\]

	We now state and prove a preliminary result on $SBV$ functions used in the paper.
	\begin{prop}\label{prop3}
	Let $\Omega\subseteq \mathbb{R}^N$ be an open set, let $u \in SBV^{p}(\Omega;\mathbb{R}^{m})$, and let $(u_n)_{n\ge1} \subseteq SBV^{p}(\Omega;\mathbb{R}^{m})$ be such that, as $n\to\infty$,
	\[
	u_n \to u \quad \text{in } L^1(\Omega;\mathbb{R}^m), \qquad |D(u_n - u)|(\Omega) \to 0, \qquad \mathcal{H}^{N-1}(J_{u_n} \,\triangle\, J_u) \to 0.
	\]
	Then, up to a subsequence (not relabeled), there exist a set $A \subseteq \mathbb{R}^N$ with $\mathcal{H}^{N-1}(A)<+\infty$ and a non-negative function $g \in L^1(\mathcal{H}^{N-1}\!\llcorner\, A)$ such that:
	\begin{enumerate}[label=\emph{(\roman*)}]
		\item $A \supseteq J_{u_n}$ for every $n \in \mathbb{N}$;
		\item $(u_n^+ - u_n^-)\otimes \nu_{u_n} \to (u^+ - u^-)\otimes \nu_u$ in measure with respect to $\mathcal{H}^{N-1}\!\llcorner\, A$;
		\item up to a further subsequence,
		\[
		u_n^+ - u_n^- \to u^+ - u^-, \qquad \nu_{u_n} \to \nu_u \qquad \mathcal{H}^{N-1}\text{-a.e.\ on } A;
		\]
		\item $|u_n^+(x) - u_n^-(x)| \le g(x)$ for every $n \in \mathbb{N}$ and $\mathcal{H}^{N-1}$-a.e.\ $x \in A$.
	\end{enumerate}
	\end{prop}
	
	\begin{proof}
		We may assume that $u_n : \mathbb{R}^N \to \mathbb{R}$ are scalar-valued, since otherwise we apply the argument component-wise.

	Since $u_n \in SBV^p(\Omega;\mathbb{R})$, the singular part of $Du_n$ is concentrated on $J_{u_n}$ and satisfies
	\[
	D^s u_n = (u_n^+ - u_n^-)\,\nu_{u_n}\,\mathcal{H}^{N-1}\!\llcorner\, J_{u_n}.
	\]
	By assumption, for every $k = 1,\dots,N$,
	\[
	|D^s_k u_n - D^s_k u|(\Omega) \to 0,
	\]
	where $D^s_k u$ denotes the $k$-th component of $D^s u$.

	Passing to a subsequence, we may assume that
	\[
	\mathcal{H}^{N-1}(J_{u_n} \,\triangle\, J_u) < 2^{-n}
	\qquad \text{for every } n \in \mathbb{N}.
	\]
	Define
	\[
	A := J_u \cup \bigcup_{n=1}^\infty J_{u_n}.
	\]
	Then
	\[
	\mathcal{H}^{N-1}(A)
	\le \mathcal{H}^{N-1}(J_u) + \sum_{n=1}^\infty \mathcal{H}^{N-1}(J_{u_n} \setminus J_u)
	\le \mathcal{H}^{N-1}(J_u) + 1 < +\infty,
	\]
	and $A \supseteq J_{u_n}$ for every $n \in \mathbb{N}$, which proves~(i).

	We extend the functions $(u_n^+ - u_n^-)\nu_{k,u_n}$ and $(u^+ - u^-)\nu_{k,u}$ to all of $A$ by setting them equal to zero outside $J_{u_n}$ and $J_u$, respectively. With this convention,
	\[
	|D^s_k u_n - D^s_k u|(\Omega)
	= \int_A
	\bigl|(u_n^+ - u_n^-)\nu_{k,u_n} - (u^+ - u^-)\nu_{k,u}\bigr|
	\,d\mathcal{H}^{N-1}
	\to 0.
	\]
	In particular, for each $k = 1,\dots,N$, the sequence $(u_n^+ - u_n^-)\nu_{k,u_n}$ converges to $(u^+ - u^-)\nu_{k,u}$ both in $L^1(\mathcal{H}^{N-1}\!\llcorner\, A)$ and in measure with respect to $\mathcal{H}^{N-1}\!\llcorner\, A$. Summing over $k$, we obtain~(ii).

	Up to a further subsequence, we may additionally assume that, for each $k = 1,\dots,N$,
	\[
	(u_n^+ - u_n^-)\nu_{k,u_n} \to (u^+ - u^-)\nu_{k,u}
	\qquad \mathcal{H}^{N-1}\text{-a.e.\ on } A.
	\]
	By the locality properties of $SBV$ functions (see~\cite[Proposition~3.73]{ambrosio2000functions}), up to a change of sign of the normal vectors, we have $\nu_{u_n} = \nu_u$ $\mathcal{H}^{N-1}$-a.e.\ on $J_{u_n} \cap J_u$. Since $\mathcal{H}^{N-1}(J_{u_n} \,\triangle\, J_u) \to 0$, this yields
	\[
	\nu_{u_n} \to \nu_u
	\qquad \text{and} \qquad
	u_n^+ - u_n^- \to u^+ - u^-
	\qquad \mathcal{H}^{N-1}\text{-a.e.\ on } A,
	\]
	which proves~(iii).

	Finally, the $L^1(\mathcal{H}^{N-1}\!\llcorner\, A)$ convergence established above implies, for each $k = 1,\dots,N$, the existence of a non-negative function $g_k \in L^1(\mathcal{H}^{N-1}\!\llcorner\, A)$ such that
	\[
	|(u_n^+ - u_n^-)\nu_{k,u_n}| \le g_k
	\qquad \mathcal{H}^{N-1}\text{-a.e.\ on } A, \quad \text{for all } n \in \mathbb{N}.
	\]
	Therefore,
	\[
	|u_n^+ - u_n^-|
	= \sum_{k=1}^N |u_n^+ - u_n^-|\,|\nu_{k,u_n}|^2
	\le \sum_{k=1}^N |(u_n^+ - u_n^-)\nu_{k,u_n}|
	\le \sum_{k=1}^N g_k.
	\]
	Setting $g := \sum_{k=1}^N g_k \in L^1(\mathcal{H}^{N-1}\!\llcorner\, A)$ gives~(iv) and concludes the proof.
	\end{proof}
    
	\begin{cor}\label{corlimsup}
	Let $\Omega\subseteq \mathbb{R}^N$ be an open set, let $u \in SBV^{p}(\Omega;\mathbb{R}^{m})$, and let
	$(u_n)_{n\ge1} \subseteq SBV^{p}(\Omega;\mathbb{R}^{m})$ satisfy, as $n\to\infty$,
	\[
	u_n \to u \quad \text{in } L^1(\Omega;\mathbb{R}^m), \qquad |D(u_n - u)|(\Omega) \to 0, \qquad \mathcal{H}^{N-1}(J_{u_n} \,\triangle\, J_u) \to 0.
	\]
	Let $E \subseteq \Omega$ be a Borel set, and let
	$\psi : \mathbb{R}^{m} \times \mathbb{S}^{N-1} \to [0,+\infty)$
	be upper semicontinuous, with
	\[
	\psi(z,\nu) = \psi(-z,-\nu)
	\qquad \text{for all } (z,\nu) \in \mathbb{R}^m \times \mathbb{S}^{N-1},
	\]
	and satisfying the linear growth bound
	\[
	\psi(z,\nu) \le C\,(1+|z|)
	\qquad \text{for all } (z,\nu) \in \mathbb{R}^m \times \mathbb{S}^{N-1},
	\]
	for some constant $C>0$. Then
	\[
	\limsup_{n\to\infty}
	\int_{J_{u_n} \cap E}
	\psi\bigl([u_n], \nu_{u_n}\bigr)\,
	d\mathcal{H}^{N-1}
	\;\le\;
	\int_{J_u \cap E}
	\psi\bigl([u], \nu_u\bigr)\,
	d\mathcal{H}^{N-1}.
	\]
\end{cor}
\begin{proof}
	By Proposition~\ref{prop3}, up to a subsequence (not relabeled), there exist a set $A \subseteq \mathbb{R}^N$ with $A \supseteq J_{u_n}$ for every $n \in \mathbb{N}$, $\mathcal{H}^{N-1}(A) < +\infty$, and a non-negative function $g \in L^1(\mathcal{H}^{N-1}\!\llcorner\, A)$ such that
	\[
	u_n^+ - u_n^- \to u^+ - u^-, \qquad \nu_{u_n} \to \nu_u
	\qquad \mathcal{H}^{N-1}\text{-a.e.\ on } A,
	\]
	and
	\[
	|u_n^+ - u_n^-| \le g
	\qquad \mathcal{H}^{N-1}\text{-a.e.\ on } A, \quad \text{for all } n \in \mathbb{N}.
	\]
	In particular, for every $n \in \mathbb{N}$,
	\[
	\psi\bigl([u_n], \nu_{u_n}\bigr)\,\chi_{E \cap J_{u_n}}
	\le C\,(1 + g) \in L^1(\mathcal{H}^{N-1}\!\llcorner\, A).
	\]
	Since $A \supseteq J_{u_n}$, we may integrate over $A$, apply the reverse Fatou lemma and use the upper semicontinuity of $\psi$ to obtain
		\[
		\begin{aligned}
			\limsup_{n\to\infty}
			\int_{J_{u_n} \cap E}
			\psi\bigl([u_n], \nu_{u_n}\bigr)\, d\mathcal{H}^{N-1}
			&=
			\limsup_{n\to\infty}
			\int_{J_{u_n}}
			\psi\bigl([u_n], \nu_{u_n}\bigr)\,\chi_{E}\,
			d\mathcal{H}^{N-1}
			\\[0.4em]
			&\le
			\int_{A}
			\limsup_{n\to\infty}
			\psi\bigl([u_n], \nu_{u_n}\bigr)\,\chi_{E \cap J_{u_n}}\,
			d\mathcal{H}^{N-1}
			\\[0.4em]
			&\le
			\int_{J_{u}}
			\psi\bigl([u], \nu_{u}\bigr)\,\chi_{E}\,
			d\mathcal{H}^{N-1}
			\\[0.4em]
			&=
			\int_{J_{u}\cap E}
			\psi\bigl([u], \nu_{u}\bigr)\,
			d\mathcal{H}^{N-1},
		\end{aligned}
		\]
		which concludes the proof.
	\end{proof}
	We now introduce the notion of a $BV$–elliptic function and state a fundamental result concerning compactness and lower semicontinuity. Proofs can be found in \cite{ambrosio2000functions} and \cite{ambrosio1988existence}.
	
	\begin{dfnz}[$BV$–ellipticity]\label{defBV-ellipticity}
		Let $T \subseteq \mathbb{R}^n$ and let 
		\[
		\psi : T \times T \times \mathbb{S}^{n-1} \to [0,+\infty].
		\]
		We say that $\psi$ is \emph{$BV$–elliptic} if
		\[
		\int_{J_u \cap Q_1(\nu)} \psi(u^+,u^-,\nu_u)\, d\mathcal{H}^{n-1} \;\ge\; \psi(i,j,\nu)
		\]
		for every piecewise constant function \(u : Q_1(\nu) \to T\) such that   
		\(\{u \neq u_{i,j,\nu}\} \Subset Q_1(\nu)\).  
		Here \(Q_1(\nu)\) denotes the open cube of side length~1, centred at the origin, with two faces orthogonal to \(\nu\); and \(u_{i,j,\nu}\) is the pure jump function attaining the values \(i\) and \(j\) on the two sides of the hyperplane \(\pi_\nu\) orthogonal to \(\nu\).  The \emph{ BV-elliptic envelope} of $\psi$, denoted by $\mathcal B\psi$, is the largest BV-elliptic function that is less than or equal to $\psi$.
		
	\end{dfnz}
	
	\begin{teo}[Closure in $GSBV^p$]\label{thm:closure}
		Let \(W\) be quasiconvex and let \(\psi\) be $BV$–elliptic, with
		\[
		0 \le W(F) \le C(1+|F|^p), 
		\qquad 
		0 \le \psi(z,\nu) \le C(1+|z|)
		\]
		for some constant \(C>0\), for all \(F \in \mathbb{M}^{m\times n}\), \(z \in \mathbb{R}^m\), and \(\nu \in \mathbb{S}^{n-1}\).  
		
		Let \(u_h \in GSBV^p(\Omega;\mathbb{R}^m)\) be such that \(u_h\to u\) in measure and
		\begin{equation}\label{eq:bounded_energy}
			\sup_h \left\{
			\int_{\Omega} |\nabla u_h|^p \, dx 
			+ 
			\mathcal{H}^{n-1}(J_{u_h} \cap \Omega)
			\right\} 
			< +\infty.
		\end{equation}
		Then \(u \in GSBV^p(\Omega;\mathbb{R}^m)\) and the following lower–semicontinuity inequalities hold:
		\[
		\int_{J_u} \psi([u],\nu_u)\, d\mathcal{H}^{n-1}
		\;\le\;
		\liminf_{h\to\infty}
		\int_{J_{u_h}} \psi([u_h],\nu_{u_h})\, d\mathcal{H}^{n-1},
		\]
		and
		\[
		\int_{\Omega} W(\nabla u)\, dx
		\;\le\;
		\liminf_{h\to\infty}
		\int_{\Omega} W(\nabla u_h)\, dx.
		\]
	\end{teo}

    We conclude this section with a relaxation result for functionals defined on $GSBV^p$. The result follows by applying the $\Gamma$-convergence theorem
    of \cite[Theorems 3.5 and 3.8]{cagnetti2019gamma} to a constant sequence of
    functionals. In this case, the $\Gamma$-limit is the lower semicontinuous
    envelope of the original functional and is represented by the corresponding
    bulk and surface blow-up formulae. The bulk blow-up coincides with the
    quasiconvex envelope of the volume density. The surface blow-up coincides with the $BV$-elliptic envelope of the surface density. These identifications follow
    by the same argument as in \cite[Proof of Theorem 4]{bouchitte2002global},
    although the growth assumptions there are slightly different.
    Indeed, the lower linear bound in $|\zeta|$ is not needed for the local
    identification of the surface cell formula, once the integral representation is
    already available.

    \begin{teo}\label{teor69}
		Let \(\Omega\) be an open set, and let  
		\[
		f : \mathbb{M}^{m\times n} \to [0,+\infty), 
		\qquad
		g : \mathbb{R}^m \setminus \{0\} \times \mathbb{S}^{n-1} \to [0,+\infty),
		\]
		satisfy:
		\begin{itemize}
			\item $f$ is continuous;
			\item there exist \(C_1>0\) and \(C_2>0\) such that \(C_1(|\xi|^p - 1) \le f(\xi) \le C_2(|\xi|^p+1)\) for all \(\xi\);
			\item $g$ is Borel measurable;
			\item for each \(\nu\) there exists \(C_3\ge1\) such that  
			\[
			g(z_1,\nu) \le C_3 g(z_2,\nu)
			\quad\text{whenever } |z_1|\le|z_2|,
			\]
			and
			\[
			g(z_1,\nu) \le g(z_2,\nu)
			\quad\text{whenever } C_3|z_1|\le |z_2|;
			\]
			\item for each \(\nu\) there exists a decreasing continuous  
			\(\sigma : [0,+\infty)\to[0,+\infty)\) with \(\sigma(0)=0\) such that  
			\[
			|g(z_1,\nu) - g(z_2,\nu)|
			\le
			\sigma(|z_1-z_2|)\bigl(g(z_1,\nu)+g(z_2,\nu)\bigr);
			\]
			\item \(g(z,\nu)=g(-z,-\nu)\) for all \(z\neq 0\);
			\item there exist \(C_4>0\) and \(C_5>0\) such that \( C_4 \le g(z,\nu)\le C_5(|z|+1)\) for all \(z\neq 0\).
		\end{itemize}
		Then the lower semicontinuous envelope of
		\[
		u \mapsto 
		\int_\Omega f(\nabla u)\, dx
		+
		\int_{J_u} g([u],\nu_u)\, d\mathcal{H}^{n-1},
		\qquad
		u \in GSBV^p(\Omega;\mathbb{R}^m),
		\]
		with respect to convergence in measure, is given by
		\[
		\int_\Omega \mathcal{Q}f(\nabla u)\, dx
		+
		\int_{J_u} \mathcal{B}g([u],\nu_u)\, d\mathcal{H}^{n-1},
		\]
		where \(\mathcal{Q}f\) is the quasiconvex envelope of \(f\), and \(\mathcal{B} g\) is the $BV$–elliptic envelope of \(g\).
	\end{teo}
    
	\section*{Appendix B} 
    
	\renewcommand{\thesection}{B} 
	
	\setcounter{figure}{0}
	\setcounter{equation}{0}
    \setcounter{teo}{0}
    In this section we prove Lemma \ref{lemmadiffeolisciokernelvariabile}. Since we need to use convolution with variable kernels, we report here their main properties, referring mainly to \cite{dephilippis2017approximation}.
    
	Let $\Omega \subset \mathbb R^N$ be a bounded set with Lipschitz boundary and let $K \subset \overline{\Omega}$ be a compact set. We fix a regularized distance function $d: \Omega \to \mathbb R$ from $K$ such that $d \in C^\infty(\mathbb R^N \setminus K)$ and
	\begin{equation}\label{distance}
		\frac{\text{dist}(x,K)}{2} \leq d(x) \leq \text{dist}(x,K) \ \ \mbox{for every $x \in \mathbb R^N$}, \ \ \ \ \ \Vert \nabla^j d \Vert_{L^\infty(\mathbb R^N \setminus K)} \leq 2 \ \ \mbox{for every $j \in \NN$}.
	\end{equation}
	We also consider a function $h \in C^\infty([0,+\infty))$ such that $h(t)=1$ for every $t \geq 2$ and 
	\begin{equation}\label{f}
		h^{(j)}(0)=0 \ \ \mbox{and} \ \ 0 \leq h^{(j)}(t) \leq 1 \ \ \mbox{for every $t \geq 0$ and $j \in \NN$,} \ \ \ 0 \leq h(t) \leq 1 \ \ \mbox{for every $t \geq 0$}.
	\end{equation}
	Given $\sigma \in (0,1)$ and $y \in B_1(0)$ the generalized translation $T_{\sigma,y}$ is defined as $$T_{\sigma,y}(x)=x-\sigma h (d(x))y$$ for every $x \in \Omega \setminus K$. Notice also that it holds
	\begin{equation}\label{translation nabla}
		\nabla T_{\sigma,y}(x)= I-\sigma h'(d(x))y \otimes \nabla d(x).
	\end{equation}
	Finally let $\rho \in C_c^\infty(B_1(0))$ be a positive radial function such that $\int_{B_1(0)}\rho=1$. Recalling that $(\Omega)_\sigma:=\Omega+B_\sigma$, we are ready to define convolutions with variable kernel. 
	
	\begin{dfnz}\label{variabel kernel def}
		Let $\sigma \in (0,1)$, $T_{\sigma,y}$ and $\rho$ as above. For any $u \in L^1((\Omega)_\sigma)$ and given $x \in \Omega \setminus K$, we define
		\begin{equation*}
			u_\sigma(x)=\int_{B_1(0)} u(T_{\sigma,y}(x))\rho(y) \, dy.
		\end{equation*}
	\end{dfnz}
	
	We now present some properties of convolutions with variable kernel, we refer again to \cite{dephilippis2017approximation} for the proof.
	
	\begin{prop}\label{variable ker prop}
		Let $1 \leq p < \infty$ and let $u \in L^p((\Omega)_\sigma)$ for some $0<\sigma<1/2$. Then
		\begin{enumerate}
			\item[{(i)}] $u_\sigma \in C^\infty(\Omega \setminus K)$.
			\item[{(ii)}] $\Vert u_\sigma \Vert_{L^p(\Omega)} \leq 2 \Vert u \Vert_{L^p((\Omega)_\sigma)}$.
			\item[{(iii)}] $\Vert u_\sigma-u \Vert_{L^p(\Omega)} \to 0$ as $\sigma \to 0$.
			\item[{(iv)}] If $u \in \textnormal{BV}((\Omega)_\sigma)$, then $u_\sigma \in \textnormal{BV}(\Omega)$ and $$\nabla u_\sigma = (\nabla u)_\sigma + \sigma \xi^\sigma$$ where $\xi^\sigma$ is a Radon measure such that $\xi^\sigma \llcorner K=0$ and $|\xi^\sigma|(\Omega) \leq 2 |D u| (\Omega)$.
			Moreover, $Du_\sigma \llcorner K=Du \llcorner K$. 
			
			\noindent Finally, if $\nabla u \in L^p((\Omega)_\sigma)$ and $J_u \subseteq K$, then $\Vert \nabla u_\sigma-\nabla u \Vert_{L^p(\Omega)} \to 0$ as $\sigma \to 0$.
		\end{enumerate}
	\end{prop}
	If $u \in W^{1,\infty}((\Omega)_\sigma \setminus K)$ then from Proposition \ref{variable ker prop} we can deduce additional properties on $u_\sigma$.
	
	\begin{cor}\label{varianle ker cor}
		If $u \in W^{1,\infty}((\Omega)_\sigma \setminus K)$ and $\Omega$ has finite measure, then $\xi^\sigma$ is a function defined as
		$$
		\xi^\sigma = - h'(d(x))\nabla d(x) \int_{B_1(0)} \nabla u (T_{\sigma,y}(x))y\rho(y) \, dy
		$$
		and $\Vert \xi^\sigma \Vert_{L^\infty(\Omega)} \leq 2 \Vert \nabla u \Vert_{L^\infty((\Omega)_\sigma)}$ for every $\sigma \in (0,1/2)$. Moreover, we have that $u_\sigma \in W^{1,\infty}(\Omega \setminus K)$ and $\Vert \nabla u_\sigma \Vert_{L^\infty(\Omega)} \leq 3 \Vert  \nabla u \Vert_{L^\infty((\Omega)_\sigma)}$ for every $\sigma \in (0,1/2)$.
	\end{cor}

	The next result asserts that if $u \in W^{k,\infty}$ for some $k \in \NN$, then $u_\sigma \in W^{k,\infty}$ as well in a slightly smaller domain.
	
	\begin{lemma}\label{variable ker lem}
		Let $\Omega \subset \mathbb{R}^N$ be a bounded open set with Lipschitz boundary and let $K \subset \partial \Omega$ be a compact set. Let $u \in W^{k,\infty}(\Omega)$ for every integer $k \geq 1$. Then, the function $u_\sigma$ defined as in Definition \ref{variabel kernel def} is such that $u_\sigma \in W^{k,\infty}(U)$ for every integer $k \geq 1$, where
        \[
        U=\{ x \in \Omega \colon \, \textnormal{dist}(x,\partial \Omega \setminus K) >\sigma \}.
        \]
	\end{lemma}
	
	\begin{proof}
        By Corollary~\ref{varianle ker cor}, $u_\sigma \in W^{1,\infty}(U)$ and 
    $\nabla u_\sigma = (\nabla u)_\sigma + \sigma \xi^\sigma_1$, where
    \begin{equation}\label{xi}
        \xi^\sigma_1 = - h'(d(x))\nabla d(x) \int_{B_1(0)} \nabla u(T_{\sigma,y}(x))\, y\rho(y) \, dy.
    \end{equation}
    Differentiating iteratively, one finds
    \[
    \nabla^k u_\sigma = (\nabla^k u)_\sigma + \sigma \sum_{j=1}^k \nabla^{k-j} \xi^\sigma_j,
    \]
    where each $\xi^\sigma_j$ has an expression analogous to \eqref{xi} with $\nabla^j u$ instead of $\nabla u$. Therefore, it suffices to show $\xi^\sigma \in W^{k,\infty}(U)$ for every $k \geq 1$. 
    By \eqref{distance}, \eqref{f}, and \eqref{translation nabla}, each component of 
    $\nabla^k \xi^\sigma$ is a finite sum of terms of the form
    \[
    H(x) \int_{B_1(0)} \nabla^j u(T_{\sigma,y}(x))\,\underbrace{[y,\dots,y]}_\text{$j$-times}\,\rho(y)\,dy,
    \quad j \in \{1,\dots,k\},
    \]
    with $\|H\|_{L^\infty(U)} \leq 2$. Hence, for every $k \geq 1$ there exists $C=C(k)$ such that
    \[
    \|\nabla^k \xi^\sigma\|_{L^\infty(U)} \leq C \sum_{j=1}^k \|\nabla^j u\|_{L^\infty(\Omega)},
    \]
    which concludes the proof.
	\end{proof}

    We are now ready to prove Lemma~\ref{lemmadiffeolisciokernelvariabile}. 

      \begin{proof}[Proof of Lemma \ref{lemmadiffeolisciokernelvariabile}]
         It is not restrictive to assume that $S=\left\{(x_1,g(x_1))| \ x_1 \in [0,1] \right\}$. Let $\delta>0$. We divide the proof in two steps: first we define a general $W^{1,\infty}$-homeomorphism and then we smooth it with variable kernel convolutions. Finally, we check that it satisfies all the desired properties.

    \textit{Step 1: construction of $\Psi_\delta$.} 
    Let $f\in C^\infty([0,1])$ satisfy $f(0)=f(1)=0$, $f>0$ on $(0,1)$, $f'(0)>0$, $f'(1)<0$, 
    and define
    \[
    \Delta_\delta := \{(x_1,x_2) \mid x_1\in[0,1],\ x_2\in[g(x_1),\,g(x_1)+\delta^2 f(x_1)]\}.
    \]
    The assumptions on $g$ and $f$ ensure that $\Delta_\delta$ has Lipschitz boundary, giving~(ii).
    Set
    \[
    \varphi(t) =
    \begin{cases}
        t & t \leq 0, \\
        (1-\delta)t+\delta^2 & t \in (0,\delta), \\
        t & t \geq \delta,
    \end{cases}
    \]
    and define $\Psi_\delta:\mathbb{R}^2\setminus\overline{S}\to\mathbb{R}^2\setminus\Delta_\delta$ 
    by $\Psi_\delta=\mathrm{Id}$ outside $(0,1)\times\mathbb{R}$ and
    \[
    \Psi_\delta(x_1,x_2) := \left(x_1,\; \varphi\!\left(\frac{x_2-g(x_1)}{f(x_1)}\right)f(x_1)+g(x_1)\right)
    \quad \text{on } (0,1)\times\mathbb{R}\setminus\overline{S}.
    \]
    One checks that $\nabla\Psi_\delta\in L^\infty(\mathbb{R}^2;\mathbb{M}^{2\times 2})$, that 
    $\Psi_\delta\in W^{k,\infty}(U_\delta;\mathbb{R}^2)$ for every $k\geq 1$ where 
    \[U_\delta:=\{x_1\in[0,1],\; x_2\in[g(x_1),g(x_1)+\delta f(x_1)]\},
    \]
    and that $\|\Psi_\delta-\mathrm{Id}\|_{W^{1,\infty}(\mathbb{R}^2\setminus\overline{S};\R^2)}\to 0$ 
    as $\delta\to 0$.

    \textit{Step 2: smoothing via variable kernel convolution.}
    Let $\sigma\in(2\delta^2,\delta/2)$, $K=\overline{S}$, 
    and set $\Phi_\delta:=(\Psi_\delta)^\sigma$ as in Definition~\ref{variabel kernel def}. Using 
    Proposition~\ref{variable ker prop} and Corollary~\ref{varianle ker cor}, we have that
    $\Phi_\delta \in C^\infty(\mathbb{R}^2\setminus\overline{S};\mathbb{R}^2)$ and 
    $\nabla\Phi_\delta \in L^\infty(\mathbb{R}^2;\mathbb{M}^{2\times 2})$. 
    Property~(iii) holds since $x\notin(S)_{2\delta}$ implies 
    $B_{\sigma h(d(x))}(x)\subset\mathbb{R}^2\setminus U_\delta$, so $\Phi_\delta(x)=x$.

    The fact that $\Phi_\delta:\mathbb{R}^2\setminus\overline{S}\to\mathbb{R}^2\setminus\Delta_\delta$ is a bijection 
    follows from the identity $(\Phi_\delta(x))_1=x_1$ (which holds since $(\Psi_\delta)_1=x_1$ and $\rho$ 
    is radial), the strict monotonicity of $(\Psi_\delta(x_1,\cdot))_2$, and a trace argument 
    showing that the one-sided limits of $(\Phi_\delta(x_1,\cdot))_2$ at $x_2=g(x_1)$ 
    equal $g(x_1)$ from below and $g(x_1)+\delta^2 f(x_1)$ from above.

    Property~(i) follows, up taking $\delta>0$ small enough, from Lemma~\ref{variable ker lem} applied separately on each of the 
    two connected components of 
    $(S)_{2\sigma}\setminus\overline{(S\cup B_\delta(g(0))\cup B_\delta(g(1)))}$
    (see Figure~\ref{figure3}) and using that $\Psi_\delta\in W^{k,\infty}(U_\delta;\mathbb{R}^2)$ for every $k\geq 1$.
    
    We now prove (iv). We have that $\nabla\Phi_\delta = (\nabla\Psi_\delta)^\sigma + \sigma\xi^\sigma$,
    \[
    \|\xi^\sigma\|_{L^\infty(\mathbb{R}^2;\,\mathbb{M}^{2\times 2})} \leq 2\|\nabla\Psi_\delta\|_{L^\infty(\mathbb{R}^2;\,\mathbb{M}^{2\times 2})}, \ \ \mbox{and} 
    \ \ \|(\nabla\Psi_\delta)^\sigma - I\|_{L^\infty(\mathbb{R}^2;\,\mathbb{M}^{2\times 2})} \leq \|\nabla\Psi_\delta - I\|_{L^\infty(\mathbb{R}^2;\,\mathbb{M}^{2\times 2})}.
    \]
    Thus, recalling that $\|\Psi_\delta - \mathrm{Id}\|_{W^{1,\infty}(\mathbb{R}^2\setminus\overline{S};\,\mathbb{R}^2)} \to 0$, 
    we conclude.

    \end{proof}








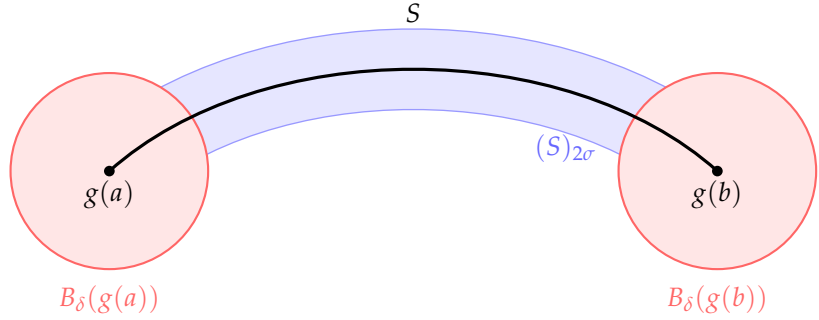
\begin{figure}[h!]
	\begin{tikzpicture}[line cap=round, line join=round]
		
		\def\R{1.3}
		\def\r{0.35}
		
		\coordinate (L) at (-4,0);
		\coordinate (R) at (4,0);
		
		\draw[
			line width=3*\r cm,
			blue!10,
			preaction={draw, blue!40, line width=3*\r cm + 1pt}
		]
		(L) .. controls (-2,1.8) and (2,1.8) .. (R);
		
		\filldraw[fill=red!10, draw=red!60, thick] (L) circle (\R);
		\filldraw[fill=red!10, draw=red!60, thick] (R) circle (\R);
		
		\draw[very thick]
		(L) .. controls (-2,1.8) and (2,1.8) .. (R);
		
		\fill (L) circle (2pt);
		\fill (R) circle (2pt);
		
		\node[below] at (L) {$g(a)$};
		\node[below] at (R) {$g(b)$};
		\node[above] at (0,1.85) {$S$};
		\node[blue!60] at (2,0.3) {$(S)_{2\sigma}$};
		\node[red!60] at (-4,-1.7) {$B_\delta(g(a))$};
		\node[red!60] at (4,-1.7) {$B_\delta(g(b))$};
		
	\end{tikzpicture}
	\caption{The curve $S$, its tubular neighborhood $(S)_{2\sigma}$, and the balls $B_\delta(g(a))$, $B_\delta(g(b))$ around the endpoints.}
	\label{figure3}
\end{figure}


\begin{thebibliography}{50}
\bibitem{acerbi1991variational}
{\sc E.~Acerbi, G.~Buttazzo, and D.~Percivale.}
\newblock \emph{A variational definition of the strain energy for an elastic string}.
\newblock Journal of Elasticity.
25(2):137--148, 1991.

\bibitem{almi2021dimension}
{\sc S.~Almi, S.~Belz, S.~Micheletti, and S.~Perotto.}
\newblock \emph{A dimension-reduction model for brittle fractures on thin shells with mesh adaptivity}.
\newblock Mathematical Models and Methods in Applied Sciences.
31(1):37--81, 2021.

\bibitem{almi2023brittlemembranes}
{\sc S.~Almi, D.~Reggiani, and F.~Solombrino.}
\newblock \emph{Brittle membranes in finite elasticity}.
\newblock ZAMM -- Journal of Applied Mathematics and Mechanics / Zeitschrift f\"ur Angewandte Mathematik und Mechanik.
103(11), 2023.

\bibitem{almi2023brittle}
{\sc S.~Almi and E.~Tasso.}
\newblock \emph{Brittle fracture in linearly elastic plates}.
\newblock Proceedings of the Royal Society of Edinburgh Section A: Mathematics.
153(1):68--103, 2023.


\bibitem{almi2025generalized}
{\sc S.~Almi, E.~Tasso.}
\newblock \emph{Generalized bounded deformation in non-Euclidean settings}.
\newblock Indiana University Mathematics Journal.
74:1099--1152, 2025.


\bibitem{ambrosio1988existence}
{\sc L.~Ambrosio.}
\newblock \emph{Existence theory for a new class of variational problems}.
\newblock Archive for Rational Mechanics and Analysis.
111:291--322, 1990.

\bibitem{ambrosio2000functions}
{\sc L.~Ambrosio, N.~Fusco, and D.~Pallara.}
\newblock  Functions of Bounded Variation and Free Discontinuity Problems.
\newblock {\em Oxford University Press}, 2000.


\bibitem{anzellotti1994dimension}
{\sc G.~Anzellotti, S.~Baldo, and D.~Percivale.}
\newblock \emph{Dimension reduction in variational problems, asymptotic development in $\Gamma$-convergence and thin structures in elasticity}.
\newblock Asymptotic Analysis.,
9(1):61--100, 1994.

\bibitem{babadjian2006quasistatic}
{\sc J.-F.~Babadjian.}
\newblock \emph{Quasistatic evolution of a brittle thin film}.
\newblock Calculus of Variations and Partial Differential Equations.,
26(1):69--118, 2006.

\bibitem{babadjian2016reduced}
{\sc J.-F.~Babadjian and D.~Henao.}
\newblock \emph{Reduced models for linearly elastic thin films allowing for fracture, debonding or delamination}.
\newblock Interfaces and Free Boundaries.,
18(4):545--578, 2016.

\bibitem{barenblatt1962equilibrium}
{\sc G.~I.~Barenblatt.}
\newblock \emph{The mathematical theory of equilibrium cracks in brittle fracture}.
\newblock Advances in Applied Mechanics.,
7:55--129, 1962.

\bibitem{bouchitte2002global}
{\sc G.~Bouchitt\'e, I.~Fonseca, G.~Leoni, and L.~Mascarenhas.}
\newblock \emph{A global method for relaxation in $W^{1,p}$ and in $SBV^p$}.
\newblock Archive for Rational Mechanics and Analysis.,
165(3):187--242, 2002.

\bibitem{bourdin2008variational}
{\sc B.~Bourdin, G.~A.~Francfort, and J.-J.~Marigo.}
\newblock \emph{The variational approach to fracture}.
\newblock Journal of Elasticity.,
91(1):5--148, 2008.

\bibitem{braides1996homogenization}
{\sc A.~Braides, A.~Defranceschi, and E.~Vitali.}
\newblock \emph{Homogenization of free discontinuity problems}.
\newblock Archive for Rational Mechanics and Analysis.,
135(4):297--356, 1996.

\bibitem{braides2001brittle}
{\sc A.~Braides and I.~Fonseca.}
\newblock \emph{Brittle thin films}.
\newblock Applied Mathematics \& Optimization.,
44(3):299--323, 2001.

\bibitem{cagnetti2019gamma}
{\sc F.~Cagnetti, G.~Dal~Maso, L.~Scardia, and C.~I.~Zeppieri.}
\newblock \emph{{$\Gamma$}-convergence of free-discontinuity problems}.
\newblock Ann.\ Inst.\ H.~Poincar\'e C, Anal.\ Non Lin\'eaire, 36(4):1035--1079, 2019.

\bibitem{cesana2015effective}
{\sc P.~Cesana, P.~Plucinsky, and K.~Bhattacharya.}
\newblock \emph{Effective behavior of nematic elastomer membranes}.
\newblock Archive for Rational Mechanics and Analysis.,
218(2):863--905, 2015.

\bibitem{cicalese2017global}
{\sc M.~Cicalese, M.~Ruf, and F.~Solombrino.}
\newblock \emph{On global and local minimizers of prestrained thin elastic rods}.
\newblock Calculus of Variations and Partial Differential Equations.,
56(4):115, 2017.

\bibitem{cicalese2017hemihelical}
{\sc M.~Cicalese, M.~Ruf, and F.~Solombrino.}
\newblock \emph{Hemihelical local minimizers in prestrained elastic bi-strips}.
\newblock Zeitschrift f\"ur Angewandte Mathematik und Physik.,
68(6):122, 2017.

\bibitem{ciarlet2000volIII}
{\sc P.~G.~Ciarlet.}
\newblock  Mathematical Elasticity. Vol.~III: Theory of shells.
\newblock {\em Vol.~29 of Studies in Mathematics and its Applications, North-Holland, Amsterdam}, 2000. 

\bibitem{ContiDolzmann2006}
{\sc S. Conti and G. Dolzmann.}
\newblock \emph{Derivation of elastic theories for thin sheets and the constraint of incompressibility}.
\newblock Analysis, Modeling and Simulation of Multiscale Problems,
Springer, Berlin, 2006, pp.~225--247.

\bibitem{conti2009gamma}
{\sc S.~Conti and G.~Dolzmann.}
\newblock \emph{$\Gamma$-convergence for incompressible elastic plates}.
\newblock Calculus of Variations and Partial Differential Equations.,
34(4):531--551, 2009.

\bibitem{cortesani1999density}
{\sc G.~Cortesani and R.~Toader.}
\newblock \emph{A density result in {SBV} with respect to non-isotropic energies}.
\newblock  Nonlinear Analysis-theory Methods \& Applications,. 38 (1999), pp.~585--604.

\bibitem{dacorogna1989}
{\sc B.~Dacorogna.}
\newblock Direct Methods in the Calculus of Variations.
\newblock {\em Vol.~78 of Applied Mathematical Sciences, Springer-Verlag, Berlin}, 1989.

\bibitem{dal2012introduction}
{\sc G.~Dal Maso.}
\newblock An Introduction to $\Gamma$-Convergence.
\newblock {\em Springer Science \& Business Media, Vol.~8}, 2012.

\bibitem{davoli2011thin}
{\sc E.~Davoli.}
\newblock \emph{Thin-walled beams with a cross-section of arbitrary geometry: derivation of linear theories starting from 3D nonlinear elasticity}.
\newblock Advances in Calculus of Variations.,
6:33--91, 2013.

\bibitem{davoli2014linearized}
{\sc E.~Davoli.}
\newblock \emph{Linearized plastic plate models as $\Gamma$-limits of 3D finite elastoplasticity}.
\newblock ESAIM: Control, Optimisation and Calculus of Variations.,
20:725--747, 2014.

\bibitem{davoli2014quasistatic}
{\sc E.~Davoli.}
\newblock \emph{Quasistatic evolution models for thin plates arising as low energy $\Gamma$-limits of finite plasticity}.
\newblock Mathematical Models and Methods in Applied Sciences.,
24:2085--2153, 2014.

\bibitem{davoli2013mora}
{\sc E.~Davoli and M.~G.~Mora.}
\newblock \emph{A quasistatic evolution model for perfectly plastic plates derived by $\Gamma$-convergence}.
\newblock Annales de l'Institut Henri Poincar\'e C, Analyse Non Lin\'eaire.,
30:615--660, 2013.

\bibitem{dephilippis2017approximation}
{\sc G.~De~Philippis, N.~Fusco, and A.~Pratelli.}
\newblock \emph{On the approximation of $SBV$ functions}.
\newblock Rendiconti Lincei -- Matematica e Applicazioni.,
28:369--413, 2017.

\bibitem{eleuteri2024asymptotic}
{\sc M.~Eleuteri, F.~Prinari, and E.~Zappale.}
\newblock \emph{Asymptotic analysis of thin structures with point-dependent energy growth}.
\newblock Mathematical Models and Methods in Applied Sciences.,
34:1401--1443, 2024.

\bibitem{engl2026membrane}
{\sc D.~Engl, A.~Molchanova, and H.~Sch\"onberger.}
\newblock \emph{Derivation of variational membrane models in the context of anisotropic nonlocal hyperelasticity}.
\newblock arXiv preprint. arXiv:2602.17278, 2026.

\bibitem{freddi2010dimension}
{\sc L.~Freddi, R.~Paroni, and C.~Zanini.}
\newblock \emph{Dimension reduction of a crack evolution problem in a linearly elastic plate}.
\newblock Asymptotic Analysis.,
70:101--123, 2010.

\bibitem{friedrich2024voids}
{\sc M.~Friedrich, L.~Kreutz, and K.~Zemas.}
\newblock \emph{Derivation of effective theories for thin 3D nonlinearly elastic rods with voids}.
\newblock Mathematical Models and Methods in Applied Sciences.,
34:723--777, 2024.

\bibitem{friedrich2026kirchhoff}
{\sc M.~Friedrich, L.~Kreutz, and K.~Zemas.}
\newblock \emph{Derivation of Kirchhoff-type plate theories for elastic materials with voids}.
\newblock Journal de Math\'ematiques Pures et Appliqu\'ees.,
103865, 2026.

\bibitem{friedrich2020derivation}
{\sc M.~Friedrich and M.~Kru\v{z}\'{\i}k.}
\newblock \emph{Derivation of von K\'arm\'an plate theory in the framework of
	three-dimensional viscoelasticity}.
\newblock Archive for Rational Mechanics and Analysis.,
238(1):489--540, 2020.

\bibitem{friedrich2022ribbon}
{\sc M.~Friedrich and L.~Machill.}
\newblock \emph{Derivation of a one-dimensional von K\'arm\'an theory for viscoelastic ribbons}.
\newblock NoDEA: Nonlinear Differential Equations and Applications.,
29:Paper No.~11, 2022.

\bibitem{friesecke2002theorem}
{\sc G.~Friesecke, R.~D.~James, and S.~M\"uller.}
\newblock \emph{A theorem on geometric rigidity and the derivation of nonlinear plate theory from three-dimensional elasticity}.
\newblock Communications on Pure and Applied Mathematics.,
55(11):1461--1506, 2002.

\bibitem{friesecke2006hierarchy}
{\sc G.~Friesecke, R.~D.~James, and S.~M\"uller.}
\newblock \emph{A hierarchy of plate models derived from nonlinear elasticity by $\Gamma$-convergence}.
\newblock Archive for Rational Mechanics and Analysis.,
180(2):183--236, 2006.

\bibitem{ginster2024euler}
{\sc J.~Ginster and P.~Gladbach.}
\newblock \emph{The Euler--Bernoulli limit of thin brittle linearized elastic beams}.
\newblock Journal of Elasticity.,
156(1):125--155, 2024.

\bibitem{griffith1921rupture}
{\sc A.~A.~Griffith.}
\newblock \emph{VI. The phenomena of rupture and flow in solids}.
\newblock Philosophical Transactions of the Royal Society of London. Series A.,
221(582--593):163--198, 1921.

\bibitem{gromov1971nonsingular}
{\sc M.~Gromov and J.~Eliashberg.}
\newblock \emph{Construction of nonsingular isoperimetric films}.
\newblock Trudy Matematicheskogo Instituta im. V.\,A. Steklova.,
(1971):18--33.

\bibitem{gromov1986pdr}
{\sc M.~Gromov.}
\newblock Partial Differential Relations.
\newblock {\em Vol.~9, Springer}, 1986.

\bibitem{hafsa2008nonlinear}
{\sc O.~A.~Hafsa and J.-P.~Mandallena.}
\newblock \emph{The nonlinear membrane energy: variational derivation under the constraint ``$\det \nabla u \neq 0$''}.
\newblock Journal de Math\'ematiques Pures et Appliqu\'ees.,
86:100--115, 2005.

\bibitem{hafsa2008nonlinear2}
{\sc O.~A.~Hafsa and J.-P.~Mandallena.}
\newblock \emph{The nonlinear membrane energy: variational derivation under the constraint ``$\det \nabla u > 0$''}.
\newblock Bulletin des Sciences Math\'ematiques.,
132:272--291, 2008.

\bibitem{hornung2014derivation}
{\sc P.~Hornung, S.~Neukamm, and I.~Vel\v{c}i\'c.}
\newblock \emph{Derivation of a homogenized nonlinear plate theory from 3D elasticity}.
\newblock Calculus of Variations and Partial Differential Equations.,
51(3):677--699, 2014.

\bibitem{kohn1986optimal}
{\sc R.~V. Kohn and G.~Strang.}
\newblock \emph{Optimal design and relaxation of variational problems, I}.
\newblock Communications on Pure and Applied Mathematics., 39(1):113--137, 1986.


\bibitem{kruvzik2026effects}
{\sc M.~Kru\v{z}\'\i k and F.~Riva.}
\newblock \emph{The effects of pressure loads in the dimension reduction of elasticity models}.
\newblock{\em Journal of Nonlinear Science}, 
36:37, 2026.

\bibitem{ledret1993modele}
{\sc H.~Le~Dret and A.~Raoult.}
\newblock \emph{Le mod\`ele de membrane non lin\'eaire comme limite variationnelle de l'\'elasticit\'e non lin\'eaire tridimensionnelle}.
\newblock Comptes Rendus de l'Acad\'emie des Sciences. S\'erie I, Math\'ematique.,
317(2):221--226, 1993.



\bibitem{schmidt2017griffith}
{\sc B.~Schmidt.}
\newblock \emph{A Griffith--Euler--Bernoulli theory for thin brittle beams derived from nonlinear models in variational fracture mechanics}.
\newblock Mathematical Models and Methods in Applied Sciences.,
27(9):1685--1726, 2017.

\bibitem{schmidt2023continuum}
{\sc B.~Schmidt and J.~Zeman.}
\newblock \emph{A continuum model for brittle nanowires derived from an atomistic
	description by $\Gamma$-convergence}.
\newblock Calculus of Variations and Partial Differential Equations,
62:243, 2023.

\bibitem{trabelsi2006modeling}
{\sc K.~Trabelsi.}
\newblock {\em Modeling of a membrane for nonlinearly elastic incompressible materials via $\Gamma$-convergence}.
\newblock Analysis and Applications.,
4(1):31--60, 2006.

\end{thebibliography}

\end{document}